\documentclass[11pt,reqno]{amsart}
\usepackage[T1]{fontenc}
\usepackage{amsmath,amssymb,amsthm,mathtools,mathrsfs}
\usepackage[textwidth=16cm,hcentering,top=2cm,bottom=2cm,headheight=10pt,headsep=14pt]{geometry}
\usepackage{needspace}
\usepackage{graphicx}
\usepackage{tikz}
\usepackage[hidelinks]{hyperref}
\usepackage{accsupp}
\hypersetup{
 pdflang={en-US},
 pdftitle={Rigidity of stationary black holes with small ADM angular momentum},
 pdfauthor={Alexandru F. Radu},
 pdfsubject={Stationary vacuum black hole rigidity},
 pdfkeywords={black hole rigidity, ADM angular momentum, Harrison transformation, maximal hypersurfaces, null geometry}
}
\allowdisplaybreaks[1]
\numberwithin{equation}{section}
\makeatletter
\renewcommand{\section}{\@startsection{section}{1}%
  \z@{.7\linespacing}{.5\linespacing}%
  {\normalfont\scshape\centering}}
\renewcommand{\subsection}{\@startsection{subsection}{2}%
  \z@{.5\linespacing}{-.5em}%
  {\normalfont\bfseries}}
\makeatother
\newtheorem{theorem}{Theorem}[section]
\newtheorem{proposition}[theorem]{Proposition}
\newtheorem{lemma}[theorem]{Lemma}
\theoremstyle{remark}

\newcommand{\Ric}{\operatorname{Ric}}
\newcommand{\g}{\mathbf g}
\newcommand{\T}{\mathbf T}
\newcommand{\K}{\mathbf K}
\newcommand{\Z}{\mathbf Z}
\newcommand{\D}{\mathbf D}
\newcommand{\R}{\mathbf R}
\newcommand{\M}{\mathbf M}
\newcommand{\F}{\mathbf F}
\newcommand{\Div}{\operatorname{div}}
\newcommand{\tr}{\operatorname{tr}}
\newcommand{\Lie}{\mathcal L}
\newcommand{\dd}{\,\mathrm d}
\newcommand{\gcirc}{\overset{\circ}{\gamma}}
\newcommand{\cF}{\mathcal F}
\newcommand{\eps}{\varepsilon}
\newcommand{\DOC}{\mathbf E}
\begin{document}
\flushbottom
\title[Rigidity with small ADM angular momentum]
{Rigidity of stationary black holes with small ADM angular momentum}
\author{Alexandru F. Radu}
\address{Simion Stoilow Institute of Mathematics, Romanian Academy,
Calea Grivitei Street, no.~21, 010702 Bucharest, Romania}
\email{alexandru.radu@imar.ro}
\date{}

\begin{abstract}
We prove global rigidity of smooth asymptotically flat stationary
vacuum black holes with sufficiently small ADM angular momentum,
under global regularity, bifurcation, and maximal-hypersurface
assumptions, with a threshold depending on fixed geometric bounds.
This answers the question of Alexakis, Ionescu, and Klainerman
concerning replacement of their smallness assumption on the
stationary field at the bifurcation sphere by smallness of the
angular momentum at infinity. The main step
establishes Kerr rigidity near infinity for arbitrary angular
momentum. We transform the vacuum metric into an Einstein--Maxwell
family and compare null focusing at its limiting end with the
area--charge inequality on compact maximal hypersurfaces, constructed
by continuation through the family. Equality determines both the
transformed metric and its stationary field, from which the inverse
transformation recovers Kerr. Two Kerr identities then give the
horizon estimate required for the global conclusion.
\end{abstract}
\maketitle
\enlargethispage{2pt}
\tableofcontents

\section{Introduction}\label{sec:introduction}

The black hole uniqueness conjecture predicts that the exterior
of a regular stationary vacuum black hole belongs to the Kerr
family and is therefore determined by its mass and angular
momentum. Since an isolated black hole is expected to approach
a stationary regime as radiation escapes to infinity or crosses
the horizon, this conjecture would characterize its possible
asymptotic states. The question of whether such a regime is
reached belongs to the dynamical stability problem, which is
distinct from the stationary classification considered here.
See~\cite{IKreview} and~\cite{CCH}.

In the rotating case, the classical uniqueness argument first
uses analyticity to obtain the additional rotational symmetry
through Hawking's rigidity theorem~\cite{Hawking}, and then
identifies the resulting stationary-axisymmetric metric by the
results of Carter~\cite{Carter} and Robinson~\cite{Robinson}.
A precise global treatment is given by Chru\'sciel and
Costa~\cite{ChruscielCosta}. Although stationarity implies
analyticity where the stationary field is timelike~\cite{MullerHagen},
this conclusion is unavailable in the ergoregion, where that
field is spacelike, so the classical extension argument does
not apply to an arbitrary smooth stationary metric.

For smooth nonexpanding bifurcate horizons, Friedrich, R\'acz,
and Wald~\cite{FRW} construct a Killing field
normal to both horizon branches and vanishing at their
intersection, but their construction is confined to the domain
of dependence of the two branches. Alexakis, Ionescu, and
Klainerman~\cite{AIKlocal} extend this
field to a full neighbourhood of the intersection and, in the
rotating stationary case, construct a commuting rotational
field. The continuation of this local symmetry gives global
rigidity under suitable comparison assumptions with
Kerr~\cite{IKuniqueness,AIKnearKerr}, or under the condition
in the author's recent preprint~\cite{RaduNontrapping}
excluding inextendible null geodesics orthogonal to the stationary
field whose projections remain in a compact part of the space of
stationary orbits. The need for further information in
this continuation is illustrated by the local counterexamples
of Ionescu and Klainerman~\cite{IKextension}.

Alexakis, Ionescu, and Klainerman~\cite{AIKsmall} prove global
rigidity under a smallness assumption on the stationary field
at the bifurcation sphere, where the future and past horizons
meet, and ask whether this assumption can be replaced by
smallness of the Arnowitt--Deser--Misner (ADM) angular momentum.
We answer this question under the same global regularity,
bifurcation, and maximal-hypersurface assumptions, with a smallness
threshold depending on fixed bounds for the geometry.
Here a hypersurface is maximal when its mean curvature vanishes.

The main ingredient is an identification of the metric and
stationary field with Kerr near infinity, obtained without a
restriction on the angular momentum. It determines all higher
asymptotic multipole moments from the mass and angular momentum,
in the sense of Geroch and Hansen~\cite{GerochMoments,Hansen}.

\subsection{Assumptions and the main results}\label{subsec:assumptions}

Let $(\M,\g)$ be a connected, oriented, time-oriented smooth
four-dimensional spacetime of signature $(-,+,+,+)$ satisfying
$\R_{\alpha\beta}=0$, and let $\T$ be a smooth Killing field.
We follow the notation of \cite{AIKsmall}, writing $\D$ for the
Levi--Civita connection and $\R_{\alpha\beta}$ for the Ricci tensor,
and recall its global regularity and bifurcation assumptions below.
For the causal relations in these assumptions, $p\ll q$ means that
a future-directed timelike curve joins $p$ to $q$, while
$\mathcal I^\pm(A)$ denote the chronological future and past of
$A$. A set is achronal if no two of its points can be joined by a
timelike curve. We also write $E_a=\{x\in\mathbb R^3:|x|>a\}$.

\medskip\noindent
GR. (Global regularity.)
There is an embedded spacelike partial Cauchy hypersurface
$\Sigma^0$ and a diffeomorphism $\Phi_0:E_{1/2}\to\Sigma^0$.
Here acausal means that no causal curve joins distinct points of
the hypersurface. A partial Cauchy hypersurface is acausal and has
no edge. The absence of an edge means that, locally, timelike curves
crossing from its past side to its future side must meet it. In particular it is achronal and closed in the
ambient spacetime.
Put
\[
 \Sigma_1=\Phi_0(E_1),\qquad
 \mathcal S_0=\Phi_0(\{|x|=1\}),\qquad
 \overline{\Sigma}_1=\Sigma_1\cup\mathcal S_0.
\]
For sufficiently large $R_0$, the restriction of $\Phi_0$ to $E_{R_0}$
extends to an asymptotic stationary chart
\[
 \Phi_0:\mathbb R\times E_{R_0}\longrightarrow
 \M^{(\mathrm{end})},\qquad \T=\partial_t,
\]
whose zero-time section agrees with $\Sigma^0$.
In this chart, with $r=|x|$, the metric satisfies
\begin{align}
 \g_{00}&=-1+\frac{2M}{r}+O_6(r^{-2}),&
 \g_{ij}&=\delta_{ij}+O_6(r^{-1}),\label{eq:AF}\\
 \g_{0i}&=-2\epsilon_{ijk}S^j x^k r^{-3}+O_6(r^{-3}),\nonumber
\end{align}
Here $M>0$ is the ADM mass. Following \cite{AIKsmall},
we write $S=(S^1,S^2,S^3)$ for the ADM angular momentum vector
and $J=|S|$ for its magnitude.
The notation $O_m(r^{-a})$ includes the estimates
$|\partial^I w|\le C_Ir^{-a-|I|}$ for $|I|\le m$,
for a scalar function $w$, and componentwise for tensor fields.
Here $i,j,k$ range from $1$ to $3$
and $\epsilon_{123}=1$.

The domain of outer communications
\[
 \DOC=\mathcal I^-(\M^{(\mathrm{end})})
             \cap \mathcal I^+(\M^{(\mathrm{end})})
\]
is globally hyperbolic, meaning that it is causal and every causal diamond
$J^+(p)\cap J^-(q)$ is compact. Here $J^\pm$ allow causal curves,
including null curves, and causality excludes closed causal curves.
We also assume
\[
 \Sigma^0\cap \mathcal I^-(\M^{(\mathrm{end})})
 =\Sigma^0\cap \mathcal I^+(\M^{(\mathrm{end})})
 =\Sigma_1.
\]
The field $\T$ is nowhere zero on $\DOC$. Every stationary orbit
there is complete and intersects $\Sigma_1$ transversally.

\medskip\noindent
SBS. (Smooth bifurcation sphere.)
In a neighbourhood of $\mathcal S_0$, write
\[
 \mathcal H^+=\partial \mathcal I^-(\M^{(\mathrm{end})}),\qquad
 \mathcal H^-=\partial \mathcal I^+(\M^{(\mathrm{end})}).
\]
The hypersurfaces $\mathcal H^\pm$ are smooth, null, transverse,
and nonexpanding, with
\[
 \mathcal H^+\cap\mathcal H^-=\mathcal S_0,\qquad
 \T\in T\mathcal H^\pm\quad\hbox{on }\mathcal H^\pm,
 \qquad \T|_{\mathcal S_0}\not\equiv0.
\]
For a null generator $L$ and vectors $e_A$
tangent to a spacelike section of a horizon, its null second
fundamental form is $\chi_{AB}=\g(\D_{e_A}L,e_B)$.
Nonexpansion means $\tr\chi=0$.

\medskip
Let $T_0$ be the future unit normal to $\Sigma^0$, and set
\[
 h=\g|_{T\Sigma^0},\qquad
 k(Y,V)=-\g(\D_YT_0,V),\qquad Y,V\in T\Sigma^0.
\]
We assume the maximality condition
\[
 \operatorname{tr}_h k=0\quad\hbox{on }\Sigma_1.
\]
Its induced geometry extends smoothly to $\mathcal S_0$ in $\Sigma^0$.
Apart from the restriction $J<M^2$ in that reference, these
are the hypotheses of \cite{AIKsmall} before imposing smallness
of $\T$ on $\mathcal S_0$. We rescale to $M=1$ throughout the proof.

For the quantitative global statement we assume $J<M^2$ and use
the effective constant
$\overline A$ defined in~\cite{AIKsmall}.
There are constants $A_1$, $R_1\ge\max(R_0,A_1^4)$, and
$\eps_0>0$ with the following bounds. In the exterior chart,
$A_1^{-1}|v|^2\le h_{ij}v^iv^j\le A_1|v|^2$ on $\Phi_0(E_{3/4})$,
and the differentiated remainders in \eqref{eq:AF} have coefficients
bounded by $A_1$. The exterior coordinates extend to a spacetime
neighbourhood through a map
\[
 \Phi_1:(-\eps_0,\eps_0)\times E_{1-\eps_0}
          \longrightarrow\widetilde{\M}.
\]
This is the chart constructed in that reference by interpolating
between the unit normal of $\Sigma^0$ and $\T$. It agrees with the
stationary chart for $r>2R_1$. The construction also gives
\begin{equation}\label{eq:coordinate-time-bound}
 \sum_{i=1}^3|\g_{0i}|+|\g_{00}+1|\le\frac{A_1}{R_1+r}.
\end{equation}
In these coordinates set
\begin{equation}\label{eq:effective-bound}
 \overline A=\max\left\{R_1,\eps_0^{-1},
       (M^2-J)^{-1},
       \sup_{\widetilde{\M}}\sum_{m=0}^6
          \bigl(|\partial^m\g_{\alpha\beta}|
                         +|\partial^m\T^\alpha|\bigr)\right\}.
\end{equation}
The component norms include summation over their indices.

\begin{theorem}\label{thm:main}
Assume GR, SBS, maximality of $\Sigma_1$, and $J<M^2$.
For every fixed $A$ there is $j_*(A)>0$ such that
\[
 M=1,\qquad \overline A\le A,\qquad J<j_*(A)
\]
imply that $(\DOC,\g)$ is isometric to a subextremal Kerr exterior.
\end{theorem}

The nonvanishing condition $\T|_{\mathcal S_0}\not\equiv0$ in
SBS is required for each approximating spacetime, but need not
persist in the $J=0$ limit used in the proof.

Theorem~\ref{thm:main} will follow from an identification of the
stationary geometry near infinity. The area comparison used to
obtain this identification requires only the qualitative
hypotheses, and gives the following result.

\begin{theorem}\label{thm:kerr-end}
Assume GR, SBS, and maximality of $\Sigma_1$, with $M=1$.
There is a diffeomorphism $\Xi$ from a neighbourhood of spatial
infinity onto a neighbourhood of spatial infinity in the Kerr
metric with mass $1$ and angular momentum $S$ such that
\[
 \Xi^*\g_{\mathrm{Kerr}}=\g,\qquad \Xi_*\T=\partial_t.
\]
\end{theorem}

\subsection{Main ideas}\label{subsec:ideas}

By~\cite{AIKsmall}, it suffices to control the stationary field
on the bifurcation sphere by the angular momentum at infinity.
The identity of Sudarsky and Wald~\cite{SudarskyWald,SudarskyWaldStaticity}
relates horizon data to the second fundamental form of a maximal
hypersurface, but expressing its horizon term through angular
momentum ordinarily requires the axial symmetry we seek to prove.
We introduce electric charge instead, whose conservation requires
no axial symmetry, and use an area comparison to determine the
stationary geometry near infinity.

The Harrison transformation~\cite{Harrison} acts on the stationary
gravitational and electromagnetic potentials. Applied to $(\g,0)$,
it gives Einstein--Maxwell metrics and electromagnetic two-forms
$(\g_q,\F_q)$, $0\le q\le1$, normalized by
\[
 (\g_1,\F_1)=(\g,0),\qquad
 Q_q=\frac1{4\pi}\int_C{}^*_{\g_q}\F_q=\sqrt{1-q^2}.
\]
Here ${}^*_{\g_q}$ is the spacetime Hodge star. Maxwell's equation
$\mathrm d{}^*_{\g_q}\F_q=0$ makes the flux independent of the
sphere $C$ homologous to infinity. To use this charge near the
horizon, we realize the transformation on the whole exterior.
The bound $f=-\g(\T,\T)<1$ supplied by the original maximal
slice keeps its denominator nonzero, including in the ergoregion.

At $q=0$ the sphere area at infinity tends to $4\pi$, the
area--charge lower bound for charge one. Beig and
Simon~\cite{BeigSimon} complete the conformally rescaled orbit
metric by adding a point at infinity. The cancellation of
conformal factors in \eqref{eq:limiting-connection} also lets us
extend the time-space cross terms of the rescaled spacetime
metric in a suitable stationary gauge. Infinity thus becomes
a timelike line, and intersecting past and future null cones
from this line gives spheres $C$.
On the past cone, write the area measure as $A(r,\omega)\,\mathrm d\omega$,
where $r$ is affine towards infinity, $A\to1$, and
$\mathrm d\omega$ is the unit sphere measure. If $\theta$ and
$\widehat\chi$ are the trace and trace-free part of the null second
fundamental form, Raychaudhuri's equation gives
\[
 -\log A(r,\omega)
 =\int_r^\infty(t-r)\Bigl(\tfrac12\theta^2+|\widehat\chi|^2
 +2\sum_{A=1}^2\F_0(L,e_A)^2\Bigr)(t,\omega)\,\mathrm dt\ge0,
\]
with $L=\partial_r$ and $e_1,e_2$ an orthonormal angular frame.
Thus $|C|_{\g_0}\le4\pi$, and the outgoing and ingoing null
expansions are respectively nonnegative and nonpositive.

Equality forces the shear, expansion, and transverse Maxwell
field to vanish. For the complex electric and magnetic potential
$\Phi_0$ relative to the normalized stationary field, Maxwell's
equations and the asymptotic data then give
\[
 \partial_r^2\Phi_0=0,\qquad
 \frac{\operatorname{Re}\Phi_0}{r}\longrightarrow-1,\qquad
 \operatorname{Im}\Phi_0\longrightarrow-S\cdot\omega.
\]
Choosing the affine origin along each generator so that
$r=-\operatorname{Re}\Phi_0$ yields $\Phi_0=-r-iS\cdot\omega$.
This potential determines
the remaining metric coefficients, and the inverse transformation
recovers Kerr in Lemma~\ref{lem:optical-equality}.
The stationary normalization matters here. Cl\'ement's
construction~\cite{Clement1998} transforms Kerr into the
Bertotti--Robinson product of a Lorentzian surface of curvature
$-1$ and a unit round sphere~\cite{Bertotti,IRobinson}, whose
local metric is independent of the spin. The magnetic potential
retains the angular momentum through the choice of stationary field.

For a non-Kerr end, we therefore have $|C|_{\g_0}<4\pi$.
Since $Q_q\to1$, a small deformation to $q>0$ preserves
$|C_q|_{\g_q}<4\pi Q_q^2$ by continuity. The expansion
inequalities at $q=0$, however, may be equalities.
The Schwarzschild family supplies a positive first-order gain
in the two required inequalities. Comparing parameter derivatives
transfers this gain to the actual cuts for arbitrarily small
$q>0$, as shown in \eqref{eq:strict-cone-path}.
We choose such a small $\delta>0$ and continue compact maximal
hypersurfaces from the original slice down to $q=\delta$,
keeping $\mathcal S_0$ as their inner boundary.
The causal assumptions concern $\g$, whereas the hypersurfaces
are spacelike for $\g_q$. The inverse-metric identity
\eqref{eq:inverse-charged} relates their conormals to covectors
timelike for $\g$, so original causal curves bound their heights,
as in~\cite{ChruscielWald}. Height and lapse bounds then bound
the proper time of timelike curves in compact regions of a
limiting metric. Lemma~\ref{prop:compact-causal} supplies a fixed
function with timelike gradient near each such region, allowing
Bartnik's estimate~\cite{Bartnik} to control the hypersurface
normals through the ergoregion.

The strict expansion signs make the outer boundary mean convex
in the maximal hypersurface. The minimizing construction
of~\cite{CruzMendes} gives a stable minimal surface $\Gamma$
homologous to $C_\delta$, with the same charge and no greater
area. Applying the area--charge inequality
\cite{Gibbons,DainJaramilloReiris} to $\Gamma$ gives the contradiction
\[
 4\pi Q_\delta^2\le |\Gamma|_{\g_\delta}
       \le |C_\delta|_{\g_\delta}<4\pi Q_\delta^2.
\]
This continuation takes place on $[\delta,1]$ and requires no
maximal hypersurface for $q=0$.

For the global result, we continue two scalar Kerr identities
through the component where $\T$ is timelike. When $J<1$,
they imply $|\psi|\le J$, where $\psi$ is the normalized twist
potential. Thus a limit with $\overline A_j\le A$ and $J_j\to0$
satisfies the Schwarzschild identities, without requiring a
common Kerr neighbourhood for the sequence. These identities
make any interior boundary of the timelike component a regular
level surface. Vanishing twist makes the tangential part of
$\T$ a nonzero normal there. If the component did not reach
$\mathcal S_0$, a closed component of its inner boundary would
therefore have nonzero flux, contrary to vanishing divergence
and the topology of the slice. The restriction
$\T|_{\mathcal S_0}$ is Killing and vanishes at a contact point.
Vanishing twist and the positive inward lapse derivative force
its tangential derivative to vanish there as well, so Killing
uniqueness gives vanishing on the whole sphere, as proved in
Lemma~\ref{lem:zero-spin-staticity}. Compactness then gives the
horizon estimate of~\cite{AIKsmall}.

Section~\ref{sec:preliminaries} establishes $f<1$ and the stationary
product structure. Section~\ref{sec:charging} constructs the Harrison
family and reduces Kerr rigidity near infinity to the two area
inequalities. Sections~\ref{sec:residue} and~\ref{sec:bands} prove these
inequalities by null focusing and maximal continuation, respectively.
Section~\ref{sec:finish} passes from the Kerr identities to staticity
in the zero-angular-momentum limit and proves Theorem~\ref{thm:main}.
The appendices supply the global Harrison verification and the
compactness estimates used in the continuation.

\section{The maximal slice and the stationary flow}
\label{sec:preliminaries}

We first use the maximal-slice equation to prove
$f=-\g(\T,\T)<1$, which ensures that the Harrison denominator
does not vanish, and then show that the stationary flow
identifies the exterior with $\mathbb R\times\Sigma_1$.
This identification places the transformed metrics on a common
domain and allows us to represent the maximal hypersurfaces
as graphs. Throughout the construction we work with one fixed
spacetime. Bounds uniform in $A$ enter only in the compactness
argument of Section~\ref{sec:finish}.

Our curvature convention is
\[
 \R(A,B,C,D)=\g\bigl(A,
   (\D_C\D_D-\D_D\D_C-\D_{[C,D]})B\bigr).
\]
We write $\R(Y,V)=\R_{\alpha\beta}Y^\alpha V^\beta$.
Greek indices range from $0$ to $3$, and repeated upper and lower
indices are contracted. We use $\mathrm d$ for exterior
differentiation, $i_Y$ for contraction with $Y$, $\Lie_Y$ for
Lie differentiation, and ${}^*$ for the spacetime Hodge star.
The operations $\flat$ and $\sharp$ use the metric under
consideration, so $Y^\flat=\g(Y,\cdot)$.

On a spacelike hypersurface, $\nabla$, $\Ric(h)$, and $R_h$
denote the connection, Ricci tensor, and scalar curvature of its
induced metric $h$. Norms also use $h$ unless another metric
is indicated. Constants may change from line to line,
with their dependence specified in each estimate.

\subsection{The maximal-slice identities}

Write $\T=nT_0+X$, where $n=-\g(\T,T_0)$ and $X$ is tangent
to $\overline{\Sigma}_1$. We use
$\Delta_h=\nabla^i\nabla_i$. The tangential Killing equation is
\begin{equation}\label{eq:tangential-Killing}
 0=(\Lie_\T \g)_{ij}
   =\nabla_iX_j+\nabla_jX_i-2nk_{ij}.
\end{equation}
To obtain the lapse equation, we differentiate the mean curvature
along the stationary translates of $\overline{\Sigma}_1$.
These hypersurfaces are all maximal, so the variation with
velocity $nT_0+X$ gives
\begin{equation}\label{eq:mean-curvature-variation}
 0=-\Delta_hn+
   (|k|_h^2+\R_{\alpha\beta}T_0^\alpha T_0^\beta)n+X(\tr_hk).
\end{equation}
Since $\tr_hk=0$ and $\R_{\alpha\beta}=0$, the momentum
constraint $\nabla^ik_{ij}=\nabla_j(\tr_hk)$ combines with
\eqref{eq:tangential-Killing} and \eqref{eq:mean-curvature-variation}
to give the identities
\begin{equation}\label{eq:KID}
 \begin{gathered}
 \Lie_Xh=2nk,\qquad \tr_hk=0,\qquad \Div_hk=0,\\
 \Delta_hn=n|k|_h^2,\qquad \Div_hX=0.
 \end{gathered}
\end{equation}
The last equality follows by taking the trace in
\eqref{eq:tangential-Killing} and using $\tr_hk=0$.
These conventions agree with~\cite{AIKsmall}.

We apply the maximum principle to the lapse equation in
\eqref{eq:KID}, with $n=0$ on $\mathcal S_0$ by SBS and
$n\to1$ at infinity by \eqref{eq:AF}. Together with the Hopf
boundary point lemma at $\mathcal S_0$, this gives, for the
unit normal $\nu_{\mathcal S_0}$ pointing into $\Sigma_1$,
\begin{equation}\label{eq:lapse-boundary}
 0<n<1\quad\hbox{on }\Sigma_1,\qquad
 \partial_{\nu_{\mathcal S_0}}n>0\quad\hbox{on }\mathcal S_0.
\end{equation}
Using \eqref{eq:lapse-boundary} in the stationary decomposition gives
\begin{equation}\label{eq:fbound}
 f=-\g(\T,\T)=n^2-|X|_h^2<1.
\end{equation}
Since $f$ is constant along stationary orbits and every orbit
meets $\Sigma_1$ by GR, the bound \eqref{eq:fbound} holds
throughout $\DOC$.

We shall also need the decay of the lapse, shift, and second
fundamental form. In the asymptotic stationary chart,
$\T=nT_0+X$ gives $X^i=h^{ij}\g_{0j}$ and
$n^2=-\g_{00}+h^{ij}\g_{0i}\g_{0j}$, so \eqref{eq:AF} and
\eqref{eq:tangential-Killing} yield
\begin{equation}\label{eq:slice-asymptotic-data}
 X=O_6(r^{-2}),\qquad n=1-1/r+O_6(r^{-2}),\qquad
 k=O_5(r^{-3}).
\end{equation}

The Killing equation and momentum constraint in \eqref{eq:KID} give
\begin{equation}\label{eq:stationary-divergence}
 \nabla^i(k_{ij}X^j)=k_{ij}\nabla^iX^j=n|k|_h^2.
\end{equation}
This identity of Sudarsky and
Wald~\cite{SudarskyWald,SudarskyWaldStaticity}, also used
in~\cite{AIKsmall}, relates the integral of $n|k|^2$ to the
boundary values of $X$. The decay \eqref{eq:slice-asymptotic-data}
removes the term at infinity, which is why vanishing of $X$ on
$\mathcal S_0$ will imply $k=0$ in
Lemma~\ref{lem:zero-spin-staticity}.

We use the normalization of the ADM mass
in~\cite{ADM}. With our convention for $k$, the mass and angular momentum are given by
\begin{equation}\label{eq:ADM-charges}
 \begin{aligned}
 M&=\frac1{16\pi}\lim_{r\to\infty}\int_{|x|=r}
       (\partial_jh_{ij}-\partial_ih_{jj})\nu^i\dd A,\\
 S^a&=\frac1{8\pi}\lim_{r\to\infty}\int_{|x|=r}
       k_{ij}\epsilon^i{}_{ab}x^b\nu^j\dd A.
 \end{aligned}
\end{equation}
Here $\nu$ and $\mathrm dA$ are the Euclidean outward normal and
area measure in the asymptotic chart. The vector with components
$\epsilon^i{}_{ab}x^b$ generates rotation about the $a$-th axis.
The trace term in the angular momentum integral is zero because
$\tr_hk=0$.

\subsection{The stationary flow}

For a vector field $Y$, let $\Psi_{t,Y}$ denote its local flow.
\begin{lemma}\label{lem:stationary-product}
The map
\[
 \Psi:\mathbb R\times\Sigma_1\longrightarrow\DOC,
 \qquad \Psi(t,x)=\Psi_{t,\T}(x)
\]
is a diffeomorphism. In particular $H^1(\DOC;\mathbb R)=0$,
and $H_2(\DOC;\mathbb R)$ is generated by a large coordinate sphere.
\end{lemma}

\begin{proof}
By GR, every stationary orbit meets $\Sigma_1$ transversally,
so $\Psi$ is a surjective local diffeomorphism. We prove that
it is proper and then use achronality to show that the resulting
covering has only one sheet.

Fix a compact set $C\subset\DOC$. By the definition of $\DOC$
and openness of $\ll$, finitely many open sets $U_i$ cover $C$,
with points $p_i^\pm=(t_i^\pm,x_i^\pm)$ in the asymptotic region satisfying
\[
 p_i^-\ll q\ll p_i^+\qquad(q\in U_i).
\]
Choose $R>\max_i(|t_i^-|+|t_i^+|)$. Timelikeness of the
stationary orbits through $p_i^\pm$ allows us to compare their
large translates with the zero-time slice. Since the flow
preserves $\ll$, for $q\in U_i$ we obtain
\begin{equation}\label{eq:stationary-flow-chronology}
 \begin{aligned}
 (0,x_i^-)&\ll\Psi_{s,\T}(p_i^-)\ll\Psi_{s,\T}(q),&&s\ge R,\\
 \Psi_{s,\T}(q)&\ll\Psi_{s,\T}(p_i^+)\ll(0,x_i^+),&&s\le-R.
 \end{aligned}
\end{equation}
Because the zero-time points in
\eqref{eq:stationary-flow-chronology} belong to the achronal
hypersurface $\overline{\Sigma}_1$, no translate
$\Psi_{s,\T}(C)$ with $|s|\ge R$ can meet that hypersurface.
Consequently, $\Psi(t,x)\in C$ implies $|t|\le R$ and
$x=\Psi_{-t,\T}(\Psi(t,x))$, whence
\[
 \Psi^{-1}(C)\subset[-R,R]\times
 \left(\Sigma_1\cap
       \{\Psi_{s,\T}(q):|s|\le R,\ q\in C\}\right).
\]
The set on the right is compact, since $\Sigma_1$ is closed
relative to $\DOC$ and the flow image of $[-R,R]\times C$
is compact. As a closed subset of this set, $\Psi^{-1}(C)$
is compact as well, proving properness.

The proper local diffeomorphism $\Psi$ is therefore a covering
of $\DOC$, which is connected as the image of
$\mathbb R\times\Sigma_1$. For a point of the asymptotic end,
its entire stationary orbit is timelike, and achronality of
$\overline{\Sigma}_1$ shows that its fiber has exactly one
point. The covering consequently has one sheet, so $\Psi$
is a diffeomorphism.

The cohomology conclusions now follow from
$\DOC\simeq\mathbb R\times\Sigma_1$ and
$\Sigma_1\simeq(1,\infty)\times\mathbb S^2$. In particular,
every closed one-form on $\DOC$ has a global potential, while
a closed two-form has a global primitive if and only if its
integral over a large coordinate sphere vanishes.
\end{proof}

\section{The Harrison family}
\label{sec:charging}

We construct the normalized Harrison family of
Section~\ref{subsec:ideas} on the whole exterior. The quotient
transformation must extend through $f=0$ and preserve the
bifurcation sphere. Its spacetime formula will also relate the
transformed conormals to the original causal geometry, as
required for the maximal-hypersurface construction.

\subsection{The transformed potentials}

An Einstein--Maxwell solution consists of a Lorentz metric $\g$
and a real electromagnetic two-form $\F$ satisfying
\begin{equation}\label{eq:Einstein-Maxwell}
 \R_{\alpha\beta}
 =2\left(\F_{\alpha\mu}\F_\beta{}^\mu
       -\tfrac14\g_{\alpha\beta}\F_{\mu\nu}\F^{\mu\nu}\right),
 \qquad \mathrm d\F=0,\qquad \mathrm d{}^*\F=0.
\end{equation}
The electric charge through an oriented sphere $C$ is
$Q(C)=(4\pi)^{-1}\int_C{}^*\F$. Maxwell's equation makes
this flux equal on homologous spheres.
To apply the transformation, we first fix the norm and twist
potentials of the original vacuum metric.

Following \cite{AIKlocal}, define the Killing two-form
and its complex self-dual form by
\[
 F_{\alpha\beta}=\D_\alpha\T_\beta,\qquad
 F=\tfrac12\mathrm d\T^\flat,\qquad
 \cF=F+i{}^*F,\qquad {}^*\cF=-i\cF.
\]
Here $({}^*F)_{\alpha\beta}
=\tfrac12\epsilon_{\alpha\beta}{}^{\mu\nu}F_{\mu\nu}$,
where $\epsilon$ is the spacetime volume form, oriented by
$\mathrm dt\wedge\mathrm dx^1\wedge\mathrm dx^2\wedge\mathrm dx^3$
in the asymptotic chart.
Vacuum gives $\mathrm dF=\mathrm d{}^*F=0$. We normalize the Ernst
potential by
\[
 \mathrm d\sigma=2i_\T\cF,\qquad \sigma\longrightarrow1
       \quad(r\longrightarrow\infty).
\]
Writing $\sigma=f+i\psi$, this means
\begin{equation}\label{eq:vacuum-ernst}
 f=-\g(\T,\T),\qquad
 \mathrm d\psi=2i_\T{}^*F,\qquad \psi\longrightarrow0
       \quad(r\longrightarrow\infty).
\end{equation}
The potential is global on $\DOC$ because Cartan's identity gives
$\mathrm d(2i_\T\cF)=2\Lie_\T\cF-2i_\T\mathrm d\cF=0$
and Lemma~\ref{lem:stationary-product} makes this closed one-form
exact. To extend it across $\mathcal S_0$, choose a tubular
neighbourhood retracting to $\mathbb S^2$ and take a potential
there for the same smooth closed form. On the connected exterior
overlap the potentials differ by a constant, which we choose
so that they agree. This defines the vacuum Ernst
potential~\cite{ErnstVacuum} with the normalization of
\cite{AIKlocal,IKuniqueness} on $\DOC$ and across $\mathcal S_0$.
The condition
\[
 \mathrm d\psi=0
 \quad\Longleftrightarrow\quad
 \T^\flat\wedge\mathrm d\T^\flat=0
\]
is the Frobenius condition for its orthogonal planes to be
integrable. Thus the twist records the obstruction to a static
description where $\T$ is timelike.

Where $f\ne0$, the original stationary data are
\begin{equation}\label{eq:stationary-reduction}
 \vartheta=-f^{-1}\T^\flat-\mathrm dt,\qquad
 \gamma=f\g+\T^\flat\otimes\T^\flat,\qquad
 \g=-f(\mathrm dt+\vartheta)^2+f^{-1}\gamma .
\end{equation}
Both $\vartheta$ and $\gamma$ descend to the space of stationary
orbits. On $\T^\perp$ the orbit metric is $\gamma=f\g$, so it
is positive definite for $f>0$ and has signature $(+,-,-)$
for $f<0$. The stationary connection $\vartheta$ changes by
an exact form under a change of time coordinate,
\[
 t'=t+\chi(x),\qquad \vartheta'=\vartheta-\mathrm d\chi,
 \qquad \mathrm dt'+\vartheta'=\mathrm dt+\vartheta.
\]

For a Maxwell field preserved by a Killing field $Y$, its complex
electric and magnetic potential is defined by $i_Y(\F+i{}^*\F)=\mathrm d\Phi$.
Following Ernst~\cite{ErnstMaxwell}, the gravitational potential is
$\sigma=-\g(Y,Y)-|\Phi|^2+i\psi$, with twist convention
\begin{equation}\label{eq:em-potentials}
 i_Y{}^*\mathrm dY^\flat
       =\mathrm d\psi+2\operatorname{Im}(\overline\Phi\,\mathrm d\Phi).
\end{equation}
Each star and metric dual uses the metric under consideration.
We use the Harrison transformation~\cite{Harrison} in the form
\begin{equation}\label{eq:harrison}
 \begin{gathered}
 \sigma_q=\frac{(1+q)\sigma-(1-q)}{1+q-(1-q)\sigma},\qquad
 \Phi_q=-\sqrt{1-q^2}\,
                    \frac{1+\sigma}{1+q-(1-q)\sigma},\\
 D_q=|1+q-(1-q)\sigma|^2,\\
 f_q=-\g_q(\partial_\tau,\partial_\tau)
     =\operatorname{Re}\sigma_q+|\Phi_q|^2=\frac{4f}{D_q}.
 \end{gathered}
\end{equation}
The potentials in \eqref{eq:harrison} are taken relative to
$\partial_\tau$ and preserve the orbit metric $\gamma$.
For $q>0$, the norm tends to $-q^{-2}$, so the unit field at
infinity is $q\partial_\tau$. We will fix the relation to the
original field $\T$ by the horizon normalization
\eqref{eq:normalized-time}. The additive constant in $\Phi_q$
has been chosen so that
\[
 (\sigma_1,\Phi_1)=(\sigma,0),\qquad
 \sigma_0=-1,\quad \Phi_0=-\frac{1+\sigma}{1-\sigma}.
\]
Although $\sigma_0$ is constant, the electromagnetic potential
still determines the original norm and twist through
\begin{equation}\label{eq:inverse-harrison}
 \sigma=\frac{\Phi_0+1}{\Phi_0-1},
\end{equation}
the inverse relation used in Cl\'ement's
construction~\cite{Clement1998}.

The formulas are defined throughout the exterior, since the
bound \eqref{eq:fbound} gives
\begin{equation}\label{eq:harrison-denominator}
 \operatorname{Re}\bigl(1+q-(1-q)\sigma\bigr)
       =2q+(1-q)(1-f)>0.
\end{equation}
In particular the denominator remains nonzero in the ergoregion.

The Schwarzschild family makes the change in the end geometry explicit.
Write
\[
 \gcirc=\mathrm d\theta^2+\sin^2\theta\,\mathrm d\phi^2,
 \qquad \mathrm d\omega=\sin\theta\,\mathrm d\theta\wedge\mathrm d\phi
\]
for the unit round metric and its area form. For Schwarzschild,
$\sigma=1-2/R$ and $\gamma=\mathrm dR^2+R(R-2)\gcirc$.
Setting $x=R-1$ in \eqref{eq:harrison} gives
\begin{equation}\label{eq:schwarzschild-family}
 \begin{gathered}
 \g_q=-\frac{x^2-1}{(1+qx)^2}\,\mathrm d\tau^2
       +(1+qx)^2\left(\frac{\mathrm dx^2}{x^2-1}+\gcirc\right),\\
 \Phi_q=-\frac{\sqrt{1-q^2}\,x}{1+qx},\qquad x>1.
 \end{gathered}
\end{equation}
At $q=0$ this is the Bertotti--Robinson
product~\cite{Bertotti,IRobinson}, with $\Phi_0=-x$ and
$\F_0=\mathrm dx\wedge\mathrm d\tau$, as in~\cite{Clement1998}.
Its Lorentzian
factor has curvature $-1$ and is called two-dimensional
anti-de Sitter space, denoted $AdS_2$. The change $x=\cosh\ell$ makes the constant-time metric
$\mathrm d\ell^2+\gcirc$, with sphere area $4\pi$.

\subsection{The family on the whole exterior}

We first fix the horizon fields used to normalize the transformation.

By the local rigidity theorem \cite{AIKlocal}, there are commuting
Hawking and rotational Killing fields near $\mathcal S_0$.
The condition $\T|_{\mathcal S_0}\not\equiv0$ makes the circle
action nontrivial. Its generator is a constant linear combination
of $\T$ and the Hawking field, with nonzero Hawking-field
coefficient, since otherwise $\T$ would have periodic exterior
orbits, contrary to Lemma~\ref{lem:stationary-product}.
We may therefore normalize the rotational field to have least
positive period $2\pi$ and rescale the Hawking field to obtain
\begin{equation}\label{eq:local-circle}
 \K=\T+\Omega_H\Z,\qquad \K|_{\mathcal S_0}=0,\qquad
 \Psi_{2\pi,\Z}=\operatorname{Id}.
\end{equation}
Here $\Omega_H$ is the angular velocity. Since $\K$ is normal
to both null branches, they are Killing horizons, with surface
gravity $\kappa$ defined by $\D_\K\K=\kappa\K$ on
$\mathcal H^+$. In our normalization,
\eqref{eq:surface-gravity-lapse} identifies $\kappa$ with the
positive normal derivative of the lapse at a fixed point of the
circle action. Thus $\kappa>0$, and $\K$ is timelike in the
nearby exterior. Fix such a point $p\in\mathcal S_0$, so that
$\T(p)=0$, and write $S_r=\{t=0,|x|=r\}$ for the outward
oriented coordinate spheres in the fixed end.

\begin{proposition}\label{prop:charging}
Under GR, SBS, and the maximal-hypersurface assumption, there are
Lorentz metrics $\g_q$ and real two-forms $\F_q$ on
$\DOC\cup\mathbf O$, for a neighbourhood $\mathbf O$ of $\mathcal S_0$, such that
\[
 (\g_1,\F_1)=(\g,0),\qquad
 \Lie_\T\g_q=\Lie_\T\F_q=0 .
\]
The metrics depend smoothly on $q\in[0,1]$ on every compact subset.
The fields $\F_q$ are continuous in $q$ and smooth in spacetime.
On $\DOC$ the pair
$(\g_q,\F_q)$ satisfies \eqref{eq:Einstein-Maxwell} and
\[
 \frac1{4\pi}\int_{S_r}{}^*_{\g_q}\F_q=\sqrt{1-q^2},\qquad
 \frac1{4\pi}\int_{S_r}\F_q=0 .
\]
The fixed fields $\K,\Z$ remain Killing near $\mathcal S_0$, and both horizon
branches remain Killing horizons with the same generator $\K$
and surface gravity. Moreover
$\operatorname{sign}\g_q(\T,\T)=\operatorname{sign}\g(\T,\T)$,
so $\T$ is timelike in the fixed end.
\end{proposition}

Where $f\ne0$, write the transformed metric with stationary
time $\tau$ and connection $\vartheta_q$ as
\[
 \g_q=-\frac{4f}{D_q}(\mathrm d\tau+\vartheta_q)^2
                         +\frac{D_q}{4f}\gamma.
\]
To extend this metric through $f=0$, write the change of
connection in the form
\begin{equation}\label{eq:transformed-connection}
 4(\mathrm d\tau+\vartheta_q)
      =D_q(\mathrm dt+\vartheta)+W_q.
\end{equation}
Substituting this expression and
$\gamma=f\g+\T^\flat\otimes\T^\flat$ from
\eqref{eq:stationary-reduction}, and using
$\T^\flat=-f(\mathrm dt+\vartheta)$, we obtain
\begin{equation}\label{eq:metric-expanded}
 \begin{aligned}
 \g_q
 &=-\frac f{4D_q}\bigl(D_q(\mathrm dt+\vartheta)+W_q\bigr)^2
       +\frac{D_q}4\g+\frac{D_qf}4(\mathrm dt+\vartheta)^2\\
 &=\frac{D_q}4\g-\frac f4\bigl((\mathrm dt+\vartheta)\otimes W_q
                    +W_q\otimes(\mathrm dt+\vartheta)\bigr)
       -\frac f{4D_q}W_q\otimes W_q\\
 &=\tfrac14\bigl(D_q\g+\T^\flat\otimes W_q
       +W_q\otimes\T^\flat-(f/D_q)W_q\otimes W_q\bigr).
 \end{aligned}
\end{equation}
The final tensor expression in \eqref{eq:metric-expanded}
has no division by $f$, so a smooth spacetime one-form $W_q$
extends its coefficients through $f=0$.
After rescaling by $4/D_q(p)$, the metric formula agrees with
Mars's spacetime transformation~\cite{MarsEhlers}, with vector
$\T$ and one-form $-W_q/D_q(p)$.

The transformed twist equation and the normalization at $p$
require
\begin{equation}\label{eq:charged-one-form}
 \begin{aligned}
 \mathrm dW_q
  &=4(1-q)\bigl(((1-q)f-(1+q))F+(1-q)\psi{}^*F\bigr),\\
 W_q(\T)&=D_q(p)-D_q .
 \end{aligned}
\end{equation}
The curl calculation is given in \eqref{eq:connection-curl}.
The polynomial dependence on $q$ allows us to obtain all the
one-forms from a single real stationary primitive $W_0$ satisfying
\begin{equation}\label{eq:W}
 \mathrm dW_0=4\bigl((f-1)F+\psi{}^*F\bigr),\qquad
 W_0(\T)=D_0(p)-D_0 ,
\end{equation}
where $D_0=|1-\sigma|^2$. The vacuum identities
\begin{equation}\label{eq:vac-forms}
 \mathrm dF=\mathrm d{}^*F=0,\qquad
 2i_\T F=\mathrm df,\qquad
 \mathrm df\wedge F=-\mathrm d\psi\wedge{}^*F
\end{equation}
show that the two-form in \eqref{eq:W} is closed, while
\eqref{eq:vacuum-ernst} gives
\begin{equation}\label{eq:W-contraction}
 i_\T\mathrm dW_0=2(f-1)\mathrm df+2\psi\,\mathrm d\psi
                =\mathrm d|1-\sigma|^2=\mathrm dD_0.
\end{equation}
Appendix~\ref{app:harrison} constructs a stationary primitive
smooth across $\mathcal S_0$, with $(W_0)_i=O(r^{-3})$ and
the differentiated bounds \eqref{eq:primitive-radial-decay}.
We set
\begin{equation}\label{eq:regular-charged}
 W_q=(1-q)^2W_0-4q(1-q)\T^\flat .
\end{equation}
Differentiating \eqref{eq:regular-charged} and using
\eqref{eq:W} and $\mathrm d\T^\flat=2F$ gives the first
identity in \eqref{eq:charged-one-form}. For the normalization,
the polynomial identity
\begin{equation}\label{eq:charged-denominator-polynomial}
 D_q=(1-q)^2D_0+4q\bigl(1-(1-q)f\bigr)
\end{equation}
and $f(p)=0$ give
\begin{equation}\label{eq:charged-primitive-normalization}
 \begin{aligned}
 W_q(\T)
 &=(1-q)^2\bigl(D_0(p)-D_0\bigr)+4q(1-q)f\\
 &=D_q(p)-D_q.
 \end{aligned}
\end{equation}
Thus $W_q$ is stationary and smooth in both $q$ and spacetime.
Equation~\eqref{eq:charged-primitive-normalization} makes $\vartheta_q$ in
\eqref{eq:transformed-connection} annihilate $\T$ when
\begin{equation}\label{eq:normalized-time}
 \tau=\tfrac14D_q(p)t,\qquad
 \partial_\tau=4D_q(p)^{-1}\T.
\end{equation}
For $q>0$, $s=\tau/q$ is the unit asymptotic time and
$\partial_s=q\partial_\tau$.

The bound \eqref{eq:harrison-denominator} keeps $D_q>0$
on the exterior, while $f\le0$ on $\mathcal S_0$ gives
$D_q\ge(1+q)^2\ge1$ there. Compactness therefore gives
a common neighbourhood where $D_q\ge\tfrac12$ for all $q$,
so \eqref{eq:metric-expanded} extends smoothly across the
sphere. To check its signature, factor it as
\begin{equation}\label{eq:metric-factorization}
 \g_q(v,w)=\frac{D_q}{4}\,
 \g\left(v+\frac{W_q(v)}{D_q}\T,
         w+\frac{W_q(w)}{D_q}\T\right).
\end{equation}
By \eqref{eq:charged-one-form}, the linear map in
\eqref{eq:metric-factorization} has determinant and inverse
\begin{equation}\label{eq:metric-map-inverse}
 1+\frac{W_q(\T)}{D_q}=\frac{D_q(p)}{D_q}>0,
 \qquad v\longmapsto v-\frac{W_q(v)}{D_q(p)}\T.
\end{equation}
It is therefore invertible, and the positive factor $D_q/4$
preserves the Lorentzian signature, including at $f=0$.
For every covector $\xi$, the inverse map gives
\begin{equation}\label{eq:inverse-charged}
 \begin{split}
 \g^{-1}\left(\xi-\frac{\xi(\T)}{D_q(p)}W_q,
              \xi-\frac{\xi(\T)}{D_q(p)}W_q\right)
 =\frac{D_q}{4}\g_q^{-1}(\xi,\xi).
 \end{split}
\end{equation}
For a $\g_q$-spacelike graph $t=u(x)$, put $G=t-u(x)$.
Since $\mathrm dG(\T)=1$, \eqref{eq:inverse-charged} makes
$\mathrm dG-W_q/D_q(p)$ timelike for $\g$. Integrating this
covector along original causal curves will give the height
estimate in Section~\ref{sec:bands}, including in the ergoregion.
Evaluating \eqref{eq:metric-factorization} on $\T$ and using
\eqref{eq:normalized-time} gives
\begin{equation}\label{eq:charged-norm}
 \g_q(\T,\T)=-\frac{fD_q(p)^2}{4D_q},\qquad
 -\g_q(\partial_\tau,\partial_\tau)=\frac{4f}{D_q}.
\end{equation}
Thus the squared norm of $\T$ keeps its sign, and $W_1=0$, $D_1=4$
give $\g_1=\g$.

To construct the Maxwell field, apply the linear map in
\eqref{eq:metric-factorization} to $\cF$.
Since $i_\T\cF=\tfrac12\mathrm d\sigma$, its pullback of
$\cF$ is $\cF+(2D_q)^{-1}W_q\wedge\mathrm d\sigma$.
The map preserves orientation by \eqref{eq:metric-map-inverse},
and the Hodge star on two-forms is conformally invariant, so
this pullback is self-dual for $\g_q$. Multiplying it by the
factor prescribed by the Harrison potentials gives
\begin{equation}\label{eq:maxwell}
 \F_q+i{}^*_{\g_q}\F_q=-\sqrt{1-q^2}
 \frac{1+q-(1-q)\overline\sigma}{1+q-(1-q)\sigma}\,
 \left(\cF+\frac1{2D_q}W_q\wedge\mathrm d\sigma\right).
\end{equation}
The potential normalization follows from
$\mathrm d\sigma(\T)=0$, \eqref{eq:charged-one-form}, and
$\partial_\tau=4\T/D_q(p)$ in \eqref{eq:normalized-time},
\begin{equation}\label{eq:charged-potential}
 \begin{aligned}
 i_{\partial_\tau}(\F_q+i{}^*_{\g_q}\F_q)
 &=-2\sqrt{1-q^2}\,
   \frac{1+q-(1-q)\overline\sigma}{1+q-(1-q)\sigma}
   \frac{D_q+W_q(\T)}{D_qD_q(p)}\,\mathrm d\sigma\\
 &=-\frac{2\sqrt{1-q^2}}{(1+q-(1-q)\sigma)^2}\,\mathrm d\sigma
   =\mathrm d\Phi_q.
 \end{aligned}
\end{equation}
The denominator bounds also make \eqref{eq:maxwell} smooth
in spacetime. Its factor $\sqrt{1-q^2}$ gives continuity in
$q$ up to one, where $\F_1=0$.

Appendix~\ref{app:harrison} verifies the Einstein--Maxwell
equations, preservation of the horizon fields, and the collar
estimates of Lemma~\ref{lem:harrison-collar}. We compute the
fluxes next.

\subsection{The connection at infinity and the charge}

We choose the stationary times so that all the connections have
zero radial component near infinity. Writing
\[
 (W_q)_{\mathrm{sp}}=W_q-[D_q(p)-D_q]\,\mathrm dt
\]
for the spatial part of $W_q$, first change the original time
to set $\vartheta_r=0$, then add to $W_0$ a stationary exact
form, cut off away from $\mathcal S_0$, to set
$(W_0)_{\mathrm{sp},r}=0$. In this choice of coordinates,
\[
 \vartheta_0=\tfrac14\bigl(|1-\sigma|^2\vartheta+(W_0)_{\mathrm{sp}}\bigr).
\]
Using this formula, \eqref{eq:regular-charged}, and
$\T_i\,\mathrm dx^i=-f\vartheta$, followed by
\eqref{eq:charged-denominator-polynomial}, we obtain
\begin{equation}\label{eq:connection-interpolation}
 \begin{aligned}
 4\vartheta_q
 &=D_q\vartheta+(W_q)_{\mathrm{sp}}\\
 &=\bigl(D_q-(1-q)^2D_0+4q(1-q)f\bigr)\vartheta
                         +4(1-q)^2\vartheta_0\\
 &=4q\vartheta+4(1-q)^2\vartheta_0.
 \end{aligned}
\end{equation}
Substitution in \eqref{eq:transformed-connection} and
\eqref{eq:charged-norm} now gives
\begin{equation}\label{eq:charged-connection}
 \g_q=-\frac{4f}{D_q}(\mathrm d\tau+\vartheta_q)^2
                          +\frac{D_q}{4f}\gamma,\qquad
 \vartheta_q=(1-q)^2\vartheta_0+q\vartheta.
\end{equation}
Thus the two radial choices make $(\vartheta_q)_r=0$ for every
$q$, as desired.

For the charge calculation we need the leading twist term at
infinity. Since $F=\tfrac12\mathrm d\T^\flat$,
\eqref{eq:AF} gives $F_{ab}=O_5(r^{-3})$ and
$F_{0a}=O_5(r^{-2})$ for spatial indices $a,b$, while inversion
of the asymptotic metric gives $\g^{0a}=O_6(r^{-2})$.
In the identity
$\partial_i\psi=\epsilon_{0i}{}^{\alpha\beta}F_{\alpha\beta}$
from \eqref{eq:vacuum-ernst}, the terms containing $F_{0a}$
also contain $\g^{b0}$ and are consequently $O_5(r^{-4})$.
The spatial terms have the same error when their inverse-metric
coefficients and volume density are replaced by their Euclidean
limits. We can therefore use the spatial curl of $\g_{0i}$
in \eqref{eq:AF}. With $x=r\omega$ and $|\omega|=1$,
the curl and the normalization $\psi(\infty)=0$ give
\begin{equation}\label{eq:ernst-end-expansion}
 \begin{aligned}
 \partial_i\psi
 &=\epsilon_{ijk}\partial_j\g_{0k}+O_5(r^{-4})\\
 &=2S_i r^{-3}-6(S\cdot x)x_i r^{-5}+O_5(r^{-4})\\
 &=2\partial_i\bigl((S\cdot x)r^{-3}\bigr)+O_5(r^{-4}),\\
 \psi(r\omega)
 &=-\int_r^\infty\omega^i\partial_i\psi(s\omega)\,\mathrm ds
   =2(S\cdot\omega)r^{-2}+O_6(r^{-3}),\\
 \sigma&=1-\frac2r+O_6(r^{-2}).
 \end{aligned}
\end{equation}
Here radial integration bounds the remainder in $\psi$ and the
formula for $\partial_i\psi$ bounds its derivatives through
order six. Combining this with $f=-\g_{00}$ gives the expansion of $\sigma$,
with constants controlled by the coefficients and remainder
bounds in \eqref{eq:AF}.

To compute the charges, fix $q>0$ and use the original
asymptotic chart. Substituting \eqref{eq:ernst-end-expansion}
in the definition of $D_q$, and using
$\T_i=\g_{0i}=O_6(r^{-2})$ and
\eqref{eq:primitive-radial-decay} in
\eqref{eq:regular-charged}, yields
\begin{equation}\label{eq:charged-asymptotic-data}
 D_q=4q^2+O_2(r^{-1}),\qquad
 W_{q,i}=O_2(r^{-2}).
\end{equation}
With $s=\tau/q$, $x'=qx$, and $r'=|x'|$,
\eqref{eq:charged-asymptotic-data} and \eqref{eq:metric-expanded}
give $\g_q=-\mathrm ds^2+|\mathrm dx'|^2+O_2((r')^{-1})$.
Expanding \eqref{eq:harrison} by \eqref{eq:ernst-end-expansion}
and using $\partial_s=q\partial_\tau$ in
\eqref{eq:charged-potential}, we obtain
\begin{equation}\label{eq:charged-potential-expansion}
 \begin{gathered}
 q\Phi_q+Q_q
   =Q_q\frac{1-\sigma}{1+q-(1-q)\sigma}
   =\frac{Q_q}{r'}+O_2((r')^{-2}),\\
 i_{\partial_s}\F_q=\mathrm d(q\operatorname{Re}\Phi_q),\qquad
 i_{\partial_s}{}^*_{\g_q}\F_q
                       =\mathrm d(q\operatorname{Im}\Phi_q).
 \end{gathered}
\end{equation}
In an oriented orthonormal coframe our Hodge convention is
$ {}^*(e^0\wedge e^1)=-e^2\wedge e^3$ and
$ {}^*(e^2\wedge e^3)=e^0\wedge e^1$. Applying this convention
to \eqref{eq:charged-potential-expansion}, with $\mathrm d\omega$
the outward area form of the unit sphere, gives
\begin{equation}\label{eq:charged-flux-limits}
 \begin{aligned}
 \frac1{4\pi}\int_{S_{r'}}{}^*_{\g_q}\F_q
 &=-\frac q{4\pi}\int_{\mathbb S^2}
       (r')^2\partial_{r'}\operatorname{Re}\Phi_q\,\mathrm d\omega
       +O((r')^{-1})
   =Q_q+O((r')^{-1}),\\
 \frac1{4\pi}\int_{S_{r'}}\F_q
 &=\frac q{4\pi}\int_{\mathbb S^2}
       (r')^2\partial_{r'}\operatorname{Im}\Phi_q\,\mathrm d\omega
       +O((r')^{-1})
   =O((r')^{-1}).
 \end{aligned}
\end{equation}
Since the Maxwell equations \eqref{eq:Einstein-Maxwell} make
both fluxes independent of the sphere, letting $r'\to\infty$
in \eqref{eq:charged-flux-limits} proves the charge formulas
in Proposition~\ref{prop:charging} for $q>0$. To include $q=0$,
evaluate the fluxes on a fixed compact sphere $S_r$ and use
the continuity of $\g_q$ and $\F_q$. Their limiting values
are $Q_0=1$ and zero, respectively.

\subsection{The two area inequalities}\label{subsec:area-comparison}

We now reduce Theorem~\ref{thm:kerr-end} to two area inequalities
for outer spheres in the Harrison family. Null focusing at
$q=0$ will give an upper bound, with equality determining a Kerr
end. The opposite bound follows from the electric charge on
compact maximal hypersurfaces. We state the two results here
and prove them in Sections~\ref{sec:residue} and~\ref{sec:bands}.

For a spacelike two-sphere $C$ in $(\M,\g_q)$, choose an
orthonormal normal frame $(T_0,\nu)$ with $T_0$ future timelike
and $\nu$ outward spacelike. With $\D^{(q)}$ the connection of $\g_q$ and $e_1,e_2$
an orthonormal tangent frame, set
\begin{equation}\label{eq:null-expansions}
 \begin{aligned}
 H&=\sum_{A=1}^2\g_q(\D^{(q)}_{e_A}\nu,e_A),&
 p&=-\sum_{A=1}^2\g_q(\D^{(q)}_{e_A}T_0,e_A),\\
 \theta_+&=H-p,&\theta_-&=-H-p.
 \end{aligned}
\end{equation}
Thus $\theta_\pm$ are the expansions along $T_0\pm\nu$.
We call the sphere strictly outward untrapped if $H>|p|$,
or equivalently $\theta_+>0>\theta_-$. A change of normal frame preserving
the time and outward orientations rescales $T_0\pm\nu$ by
positive factors, preserving these signs. The sphere therefore has
positive mean curvature as the outer boundary of any
spacelike hypersurface.

The upper bound gives the following consequence when the end
is not Kerr.

\begin{proposition}\label{prop:end-spheres}
Suppose a stationary vacuum end satisfies \eqref{eq:AF},
has $M=1$, and is not Kerr near infinity.
There are $\delta\in(0,1)$ and a smooth family of spacelike
spheres $C_q$, $\delta\le q\le1$, such that
\[
 H_q>|p_q|,\qquad |C_\delta|_{\g_\delta}<4\pi(1-\delta^2).
\]
Their stationary projections are embedded spheres enclosing a
fixed inner sphere of the end. The family can be chosen in
any sufficiently far end region. If the original end contains
the asymptotic portion of the maximal hypersurface
$\overline\Sigma_1$, then $C_1$ can be joined to a sphere in
$\overline\Sigma_1$ by a path of strictly outward untrapped
spheres in $\g_1=\g$ with the same enclosing property.
\end{proposition}

To obtain the opposite inequality, we must extend the spheres
to compact maximal hypersurfaces with inner boundary $\mathcal S_0$
by continuation from the given maximal slice.
Let $\pi:\DOC\simeq\mathbb R\times\Sigma_1\to\Sigma_1$
be the stationary projection of Lemma~\ref{lem:stationary-product}.

\begin{proposition}\label{prop:band}
Fix a vacuum spacetime satisfying GR, SBS, and maximality of
$\Sigma_1$, and its normalized Harrison family $\g_q$.
Let $q:[0,1]\to[0,1]$ and $C_s$ be smooth, with $q(0)=1$
and $C_s$ embedded spacelike spheres in the fixed asymptotic end.
Assume
\[
 \g_{q(s)}(\T,\T)<0\quad\hbox{near }C_s,
 \qquad H_s>|p_s|\quad\hbox{on }C_s.
\]
Suppose $\pi(C_s)$ is embedded and bounds, with $\mathcal S_0$,
a spatial annulus $\mathcal A_s\simeq\mathbb S^2\times[0,1]$.
Assume $C_0\subset\Sigma_1$ bounds a compact truncation of
$\overline\Sigma_1$.
Then this truncation belongs to a smooth family of compact
spacelike hypersurfaces $\Sigma(s)$ such that
\[
 \begin{gathered}
 \partial\Sigma(s)=\mathcal S_0\sqcup C_s,\qquad
 \operatorname{tr}_{h_s}k_s=0,\qquad
 \Sigma(s)\simeq\mathbb S^2\times[0,1],\\
 \pi:\Sigma(s)^\circ\longrightarrow\mathcal A_s^\circ
       \quad\text{is a diffeomorphism}.
 \end{gathered}
\]
Here $h_s,k_s$ are the induced metric and second fundamental
form for $\g_{q(s)}$. Moreover,
\begin{equation}\label{eq:area-bound}
 |C_s|_{\g_{q(s)}}\ge4\pi(1-q(s)^2).
\end{equation}
\end{proposition}

\begin{proof}[Proof of Theorem~\ref{thm:kerr-end}]
Suppose the vacuum end is not Kerr and take the path $C_q$,
$\delta\le q\le1$, supplied by Proposition~\ref{prop:end-spheres}.
To start this path on the original maximal slice, first traverse
the joining path at $q=1$ from the same proposition, and then
follow $C_q$ down to $q=\delta$. After a smooth
reparametrization, Proposition~\ref{prop:band} applies and gives
$|C_\delta|_{\g_\delta}\ge4\pi(1-\delta^2)$, contradicting
the strict inequality in Proposition~\ref{prop:end-spheres}.
\end{proof}

\section{Null focusing at the limiting end}
\label{sec:residue}

We prove Proposition~\ref{prop:end-spheres}. Intersecting future
and past null cones from the conformal boundary of $\g_0$ will
give outer spheres of area at most $4\pi$. Equality forces the
shear and transverse Maxwell field to vanish and determines the
original end as Kerr. Otherwise we deform a cut to positive
$q$, preserving the strict area deficit while making its outward
null-expansion signs strict.

\subsection{Completion and cone cuts}

Our use of null focusing is related to the rigidity arguments
of Galloway and Graf~\cite{GallowayGraf} for asymptotically
$AdS_2\times\mathbb S^2$ spacetimes, with a timelike conformal
boundary as in Galloway, Graf, and Ling~\cite{GallowayGrafLing}.
We need only a local construction near the completed stationary
end. In the static case one can instead use level sets of the
potential in the cylindrical metric $(1-f)^2\gamma/(4f)$, up to
constant scale, as in Agostiniani and Mazzieri~\cite{AgostinianiMazzieri}.
The null cones serve this purpose also when the twist is nonzero.

Put
\[
 \Omega=\frac{|1-\sigma|^2}{4f},\qquad
 \widehat\gamma=\Omega^2\gamma.
\]
At $q=0$, formula \eqref{eq:charged-connection} gives
\begin{equation}\label{eq:optical-completion}
 \overline\g=\Omega\g_0
       =-(\mathrm d\tau+\vartheta_0)^2+\widehat\gamma.
\end{equation}
Beig and Simon~\cite{BeigSimon} complete the orbit metric
$\widehat\gamma$ by adding one point $i^0$. The remaining
issue is the connection $\vartheta_0$. We show that
$\mathrm d\vartheta_0$ extends analytically across this point.

\begin{lemma}\label{lem:optical-completion}
A mass-one asymptotically flat stationary vacuum end has a
stationary gauge in which $\overline\g$ extends analytically
across $\mathbb R\times\{i^0\}$. Here $i^0$ is the point added
to the orbit space. In normal coordinates $y$ for
$\widehat\gamma$ at $i^0$, with $\rho=|y|$,
\begin{equation}\label{eq:completion-coordinates}
 \widehat\gamma_{ij}=\delta_{ij}+O(\rho^2),\qquad
 \Omega=\rho^2+O(\rho^3),\qquad
 (\vartheta_0)_i=O(\rho),\qquad y^i(\vartheta_0)_i=0.
\end{equation}
For each remainder $R=O(\rho^k)$ in
\eqref{eq:completion-coordinates}, its normal-coordinate derivatives
satisfy $\partial_y^I R=O(\rho^{k-|I|})$ for every fixed multiindex $I$.
In particular,
\begin{equation}\label{eq:optical-normal-data}
 \overline\g=-\mathrm d\tau^2+\sum_i(\mathrm dy^i)^2+O(\rho).
\end{equation}
\end{lemma}

\begin{proof}
We use the analytic completion of Beig and Simon~\cite{BeigSimon},
after the stationary regularity improvement of Kennefick and
\'O Murchadha~\cite{KOM}, in the formulation of
\cite{DainInfinity}.
The completed mass and twist potentials of
Geroch and Hansen~\cite{GerochMoments,Hansen} are
\[
 \widehat m=\frac{1-f^2-\psi^2}{4f\sqrt\Omega},\qquad
 \widehat s=\frac{\psi}{2f\sqrt\Omega}.
\]
Together with $\widehat\gamma$ and $\Omega$, they are analytic
at $i^0$ and satisfy
\begin{equation}\label{eq:completed-potential-data}
 \begin{gathered}
 \Omega(i^0)=0,\quad \mathrm d\Omega(i^0)=0,\quad
 \widehat\nabla^2\Omega(i^0)=2\widehat\gamma(i^0),\\
 \widehat m(i^0)=1,\quad \widehat s(i^0)=0,\quad
 \mathrm d\widehat s(i^0)=S_i\,\mathrm dy^i,\qquad
 \widehat m^2+\widehat s^2=1+\Omega.
 \end{gathered}
\end{equation}
Here $\widehat\nabla$ is the connection of $\widehat\gamma$,
and the last derivative follows from the twist expansion
$\psi=2(S\cdot\omega)r^{-2}+O(r^{-3})$ in
\eqref{eq:ernst-end-expansion}.
See also Ace\~na~\cite{Acena} for convergent expansions from
the multipole data. On the punctured end, the definitions and
\eqref{eq:harrison} give
\begin{equation}\label{eq:completed-vacuum-potentials}
 \Phi_0=-\frac{\widehat m+i\widehat s}{\sqrt\Omega},\qquad
 f^{-1}=1+2\Omega+2\widehat m\sqrt\Omega,\qquad
 \psi=2f\widehat s\sqrt\Omega.
\end{equation}
The Hessian normalization in
\eqref{eq:completed-potential-data} gives, in normal coordinates,
$\widehat\gamma_{ij}=\delta_{ij}+O(\rho^2)$ and
$\Omega=\rho^2+O(\rho^3)$. At $q=0$, the transformed twist equation
\eqref{eq:charged-twist} reads
\begin{equation}\label{eq:endpoint-twist}
 f_0^2 *_\gamma\mathrm d\vartheta_0
       =-2\operatorname{Im}(\overline\Phi_0\,\mathrm d\Phi_0).
\end{equation}
Substituting \eqref{eq:harrison} at $q=0$ into
\eqref{eq:endpoint-twist}, we express $\mathrm d\vartheta_0$
in terms of the mass and twist potentials, as in
Dain~\cite{DainInfinity},
\begin{equation}\label{eq:limiting-connection}
 \begin{aligned}
 \mathrm d\vartheta_0
 &=-\tfrac14*_\gamma
       \left(\frac{2\psi}{f}\,\mathrm df
           +\frac{1-f^2+\psi^2}{f^2}\,\mathrm d\psi\right)\\
 &=-2*_{\gamma}\left\{
 \frac{1-f^2-\psi^2}{4f}\,
             \mathrm d\left(\frac\psi{2f}\right)
 -\frac\psi{2f}\,
             \mathrm d\left(\frac{1-f^2-\psi^2}{4f}\right)
                         \right\}\\
 &=-2*_{\gamma}\bigl\{
   \widehat m\sqrt\Omega\,\mathrm d(\widehat s\sqrt\Omega)
  -\widehat s\sqrt\Omega\,\mathrm d(\widehat m\sqrt\Omega)\bigr\}\\
 &=-2\Omega*_{\gamma}
       (\widehat m\,\mathrm d\widehat s
                      -\widehat s\,\mathrm d\widehat m)\\
 &=-2*_{\widehat\gamma}
       (\widehat m\,\mathrm d\widehat s
                      -\widehat s\,\mathrm d\widehat m).
 \end{aligned}
\end{equation}
Here $*_{\gamma}=\Omega^{-1}*_{\widehat\gamma}$ on one-forms.
It follows from the last expression in
\eqref{eq:limiting-connection} that $\mathrm d\vartheta_0$ extends
analytically to the full ball and remains closed there, by
continuity from the punctured ball. We can therefore choose
its radial primitive
\[
 (\vartheta_0)_i(y)
   =\int_0^1 t y^j(\mathrm d\vartheta_0)_{ji}(ty)\,\mathrm dt
\]
which is analytic, is $O(\rho)$, and satisfies
$y^i(\vartheta_0)_i=0$. Since the punctured ball is simply connected,
replacing the original connection by this primitive amounts
to a stationary time change. In this gauge,
\eqref{eq:completion-coordinates} holds and
\eqref{eq:optical-completion} extends analytically across $i^0$.

\end{proof}

For small $a>0$, let $C_a$ be the intersection of the future
$\overline\g$-null cone from $(-a,i^0)$ with the past
$\overline\g$-null cone from $(a,i^0)$, as in
Figure~\ref{fig:double-cones}. Smoothness and
\eqref{eq:optical-normal-data} give
\begin{equation}\label{eq:cone-geodesic-data}
 y(s,\omega)=s\omega+O(s^2),\qquad
 \tau(s,\omega)=\mp a\pm s+O(s^2)
\end{equation}
for their affinely parametrized null geodesics. Smooth dependence
of the geodesic equation on the initial direction makes each
remainder $s^2$ times a smooth function of $(s,\omega)$.
The cone equations obtained from \eqref{eq:cone-geodesic-data}
have Jacobian determinant $2+O(\rho)$ in $(\tau,\rho)$.
Their intersection is therefore a transverse embedded sphere with
\begin{equation}\label{eq:double-cone-graphs}
 \rho=a+O_{C^m}(a^2),\qquad \tau=O_{C^m}(a^2).
\end{equation}
These $C^m$ bounds are in the angular variables,
for every fixed $m$.

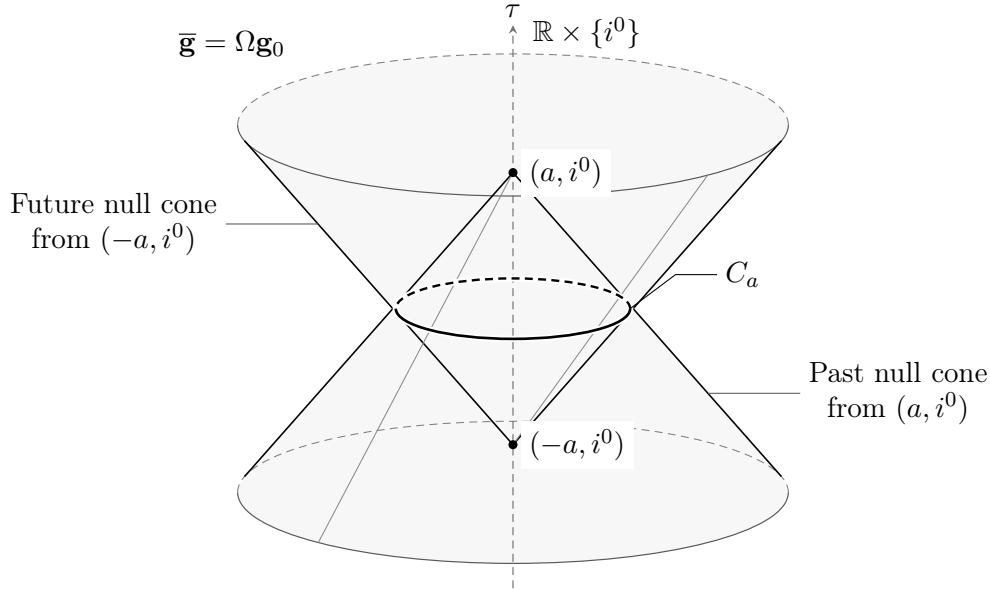
\begin{figure}[!htbp]
\centering
\leavevmode
\BeginAccSupp{method=escape,ActualText={Perspective diagram of two light
cones with one spatial dimension omitted. A future cone opens upward
from the lower vertex at time minus a, and a past cone opens downward
from the upper vertex at time a. They intersect in an ellipse labelled
C sub a, representing a two-sphere. Both vertices lie on the dashed
vertical time axis, the line added at physical infinity. Both cones
continue beyond their intersection.}}%
\begin{tikzpicture}[x=1cm,y=1cm,font=\normalsize,>=stealth,
  every node/.style={inner sep=3pt}]
 \path[fill=black!3]
   (0,-1.8) -- (-3.5514,2.2211)
   arc[start angle=192.8396,end angle=-12.8396,
       x radius=3.6425cm,y radius=.94cm] -- cycle;
 \path[fill=black!3]
   (0,1.8) -- (-3.5514,-2.2211)
   arc[start angle=167.1604,end angle=372.8396,
       x radius=3.6425cm,y radius=.94cm] -- cycle;
 \foreach \h in {-2.43,2.43} {
   \draw[densely dashed,black!50,line width=.4pt]
     (3.6425,\h) arc[start angle=0,end angle=180,
       x radius=3.6425cm,y radius=.94cm];
   \draw[black!65,line width=.45pt]
     (-3.6425,\h) arc[start angle=180,end angle=360,
       x radius=3.6425cm,y radius=.94cm];
 }
 \draw[line width=.6pt]
   (-3.5514,2.2211) -- (0,-1.8) -- (3.5514,2.2211);
 \draw[line width=.6pt]
   (-3.5514,-2.2211) -- (0,1.8) -- (3.5514,-2.2211);
 \draw[black!45,line width=.4pt]
   (0,-1.8) -- (2.5756,1.7653);
 \draw[black!45,line width=.4pt]
   (0,1.8) -- (-2.5756,-3.0947);
 \draw[densely dashed,black!55,->,line width=.45pt]
   (0,-3.7) -- (0,3.75) node[above,text=black] {$\tau$};
 \node[anchor=west] at (.15,3.65)
   {$\mathbb R\times\{i^0\}$};
 \node[anchor=east] at (-2.9,3.5)
   {$\overline\g=\Omega\g_0$};
 \fill (0,1.8) circle (1.7pt);
 \fill (0,-1.8) circle (1.7pt);
 \node[right,fill=white] at (.1,1.8) {$(a,i^0)$};
 \node[right,fill=white] at (.1,-1.8) {$(-a,i^0)$};
 \draw[white,line width=2.4pt] (0,0) ellipse (1.55cm and .4cm);
 \draw[densely dashed,line width=.9pt]
   (1.55,0) arc[start angle=0,end angle=180,
     x radius=1.55cm,y radius=.4cm];
 \draw[line width=1pt]
   (-1.55,0) arc[start angle=180,end angle=360,
     x radius=1.55cm,y radius=.4cm];
 \draw[thin] (1.55,0) -- (2.15,.45) -- (2.7,.45);
 \node[right] at (2.7,.45) {$C_a$};
 \draw[thin,black!65] (-2.58,1.12) -- (-3.8,1.12);
 \node[anchor=east,align=center] at (-3.8,1.12)
   {Future null cone\\from $(-a,i^0)$};
 \draw[thin,black!65] (2.58,-1.12) -- (3.8,-1.12);
 \node[anchor=west,align=center] at (3.8,-1.12)
   {Past null cone\\from $(a,i^0)$};
\end{tikzpicture}%
\EndAccSupp{}
\caption{The two null cones meet in the sphere $C_a$, with vertices
on the timelike line added at physical infinity. Focusing from the
vertices gives the area and null-expansion bounds
\eqref{eq:weak-cone-spheres} used to construct the outer boundaries
of the maximal hypersurfaces.}
\label{fig:double-cones}
\end{figure}

\subsection{The area bound and its equality case}

On the past cone use the parameter $s$ increasing from the
vertex and write $k=\partial_s$. The conformal connection
formula, using $\overline\g(k,k)=0$, gives
\begin{equation}\label{eq:cone-affine-normalization}
 \D^{(0)}_k k=-k(\log\Omega)k,\qquad
 L=-\Omega k,\qquad \D^{(0)}_L L=0,
 \qquad \frac{\mathrm dr}{\mathrm ds}=-\Omega^{-1}.
\end{equation}
Here $\D^{(0)}$ is the connection of $\g_0$, and $r$ is a
physical affine parameter along the future-directed generator
$L=\partial_r$. This affine parameter need not equal the
inverted radius $\rho^{-1}$. Equations~\eqref{eq:completion-coordinates}
and \eqref{eq:cone-geodesic-data} give $\Omega=s^2+O(s^3)$.
Integrating the last identity in \eqref{eq:cone-affine-normalization} gives
\begin{equation}\label{eq:cone-affine-asymptotics}
 r=s^{-1}+O(\log(1/s))\longrightarrow+\infty,
 \qquad s\sim r^{-1}.
\end{equation}

We compute the physical metric on a cut $s=\mathrm{const}$.
Differentiating \eqref{eq:cone-geodesic-data} in the angular
variables and using \eqref{eq:completion-coordinates} gives
\[
 \partial_Ay^i=s\partial_A\omega^i+O(s^2),\qquad
 \partial_A\tau+ (\vartheta_0)_i\partial_Ay^i=O(s^2).
\]
Substituting these tangent bounds and \eqref{eq:completion-coordinates}
into the pullback of \eqref{eq:optical-completion} yields
\begin{equation}\label{eq:cone-cut-metric}
 \begin{aligned}
 (\g_0)_{AB}
 &=\Omega^{-1}\Bigl\{
   \widehat\gamma_{ij}\partial_Ay^i\partial_By^j
   -\bigl(\partial_A\tau+(\vartheta_0)_i\partial_Ay^i\bigr)
    \bigl(\partial_B\tau+(\vartheta_0)_j\partial_By^j\bigr)
                       \Bigr\}\\
 &=\Omega^{-1}\{s^2\gcirc_{AB}+O(s^3)\}
  =\gcirc_{AB}+O(s).
 \end{aligned}
\end{equation}
Write the area measure as $A\,\mathrm d\omega$, with
$\mathrm d\omega$ the unit sphere measure. The smooth remainders
in \eqref{eq:cone-cut-metric} give $A=1+O(s)$ with bounded
first and second $s$ derivatives. Using
$\partial_r=-\Omega\partial_s$ from \eqref{eq:cone-affine-normalization}
and $s\sim r^{-1}$ from \eqref{eq:cone-affine-asymptotics}, we obtain
\begin{equation}\label{eq:cone-terminal-data}
 \begin{gathered}
 A=1+O(r^{-1}),\qquad
 \theta=-\Omega\partial_s\log A=O(r^{-2}),\\
 \partial_r\theta
   =\Omega\partial_s(\Omega\partial_s\log A)=O(r^{-3}).
 \end{gathered}
\end{equation}

For an orthonormal frame $e_1,e_2$ tangent to the cuts set
$\chi_{AB}=\g_0(\D^{(0)}_{e_A}L,e_B)$ and
$\widehat\chi_{AB}=\chi_{AB}-\tfrac12\theta\delta_{AB}$.
The Raychaudhuri equation and \eqref{eq:Einstein-Maxwell} give
\begin{equation}\label{eq:cone-focusing}
 \partial_r\theta=-\tfrac12\theta^2-|\widehat\chi|^2
                    -2\sum_{A=1}^2\F_0(L,e_A)^2.
\end{equation}
Using $\theta\to0$ and $A\to1$ from
\eqref{eq:cone-terminal-data}, we integrate
\eqref{eq:cone-focusing} towards infinity and obtain
\begin{equation}\label{eq:integrated-cone-focusing}
 \begin{aligned}
 \theta(r,\omega)
 &=-\int_r^\infty\partial_t\theta(t,\omega)\,\mathrm dt\ge0,\\
 -\log A(r,\omega)
 &=\int_r^\infty\theta(t,\omega)\,\mathrm dt\\
 &=-\int_r^\infty(t-r)\partial_t\theta(t,\omega)\,\mathrm dt\\
 &=\int_r^\infty(t-r)
     \Bigl(\tfrac12\theta^2+|\widehat\chi|^2
             +2\sum_{A=1}^2\F_0(L,e_A)^2\Bigr)(t,\omega)
       \,\mathrm dt .
 \end{aligned}
\end{equation}
The sum in the integrand is $-\partial_r\theta=O(r^{-3})$
by \eqref{eq:cone-terminal-data}, which justifies both integrations.
Its nonnegativity gives $A\le1$ on the whole cone. This is also
an area bound for a general cut $r=r(\omega)$, since changing
the cut adds multiples of $L$ to its tangent vectors and $L$
is orthogonal to the cone. Consequently its area is
$\int_{\mathbb S^2}A(r(\omega),\omega)\,\mathrm d\omega\le4\pi$.
Applying the time-reversed argument to the future cone gives
the other outward expansion sign, and hence
\begin{equation}\label{eq:weak-cone-spheres}
 \theta_+(C_a)\ge0\ge\theta_-(C_a),\qquad
 H(C_a)\ge|p(C_a)|,\qquad |C_a|_{\g_0}\le4\pi.
\end{equation}

\begin{lemma}\label{lem:optical-equality}
Under the hypotheses of Lemma~\ref{lem:optical-completion},
if $|C_a|_{\g_0}=4\pi$ for one sufficiently small $a>0$,
the original vacuum end is Kerr with mass one and angular
momentum $S$.
\end{lemma}

\begin{proof}
The inverse Harrison relation \eqref{eq:inverse-harrison}
requires $\Phi_0$ relative to the normalized stationary field
$\partial_\tau$. The Bertotti--Robinson product in
Cl\'ement's construction~\cite{Clement1998} is locally the same
for every $J$, while its stationary field depends on $J$.
We recover that dependence from
$\operatorname{Im}\Phi_0\to-S\cdot\omega$, derived below from
\eqref{eq:completed-potential-data}, and then determine $\Phi_0$
by the Maxwell equations.

Equality of the area with $4\pi$ forces $A=1$ at every point
of the cut, since $A\le1$. By \eqref{eq:integrated-cone-focusing},
we therefore have $\chi=0$ and $\F_0(L,e_A)=0$ along each
generator from the cut to the vertex. To use these identities
on an open set, translate the cone by the stationary flow.
The geodesics in \eqref{eq:cone-geodesic-data} then give the map
\[
 (u,s,\omega)\longmapsto
 \bigl(u-s+O(s^2),\ s\omega+O(s^2)\bigr),\qquad
 \partial_u=\partial_\tau.
\]
The pullback of $\mathrm d\tau\wedge\mathrm dy^1\wedge
\mathrm dy^2\wedge\mathrm dy^3$ by this map is
$s^2(1+O(s))\,\mathrm du\wedge\mathrm ds\wedge\mathrm d\omega$.
Together with the spatial derivative $I+O(s)$ in $s\omega$,
this shows that $(u,s,\omega)$ are coordinates on a punctured
neighbourhood of the end. In these coordinates the Killing and
affine geodesic equations, followed by the rescaling
$L=-\Omega k$ from \eqref{eq:cone-affine-normalization}, give
\begin{equation}\label{eq:equality-null-normalization}
 \partial_s\overline\g(\partial_u,k)=0,\qquad
 \overline\g(\partial_u,k)|_{s=0}=1,
 \qquad \g_0(\partial_u,\partial_r)=-1.
\end{equation}

Choose the additive constant in $r$ independently of $u$.
Since $\chi=0$ and the angular coordinates are constant
along each ray, the angular metric satisfies the transport
equation with limiting value supplied by \eqref{eq:cone-cut-metric}
\begin{equation}\label{eq:equality-angular-metric}
 \partial_r(\g_0)_{AB}
       =2\g_0(\D^{(0)}_{\partial_A}L,\partial_B)=0,
 \qquad \lim_{r\to\infty}(\g_0)_{AB}=\gcirc_{AB}.
\end{equation}
Using $\g_0(\partial_u,\partial_r)=-1$ from
\eqref{eq:equality-null-normalization} and the angular metric
from \eqref{eq:equality-angular-metric}, we write
\begin{equation}\label{eq:equality-metric-Maxwell-form}
 \g_0=\gcirc-2\,\mathrm du\,\mathrm dr-f_0\,\mathrm du^2
                         +2a_A\,\mathrm du\,\mathrm dx^A.
\end{equation}
Here $a$ is an angular one-form and the coefficients are independent
of $u$. We use the potential to determine $\F_0$ and fix the
remaining freedom in the affine parameter.
The inversion $y\mapsto y/|y|^2$ and the change from $s$ to
$r$ both reverse orientation, leaving the volume form
$\mathrm du\wedge\mathrm dr\wedge\mathrm d\omega$.
Thus \eqref{eq:charged-potential} and $(\F_0)_{rA}=0$ give
\[
 \begin{gathered}
 \partial_r\operatorname{Im}\Phi_0=({}^*\F_0)_{ur}
       =\tfrac12(\mathrm d\omega)^{AB}(\F_0)_{AB},\\
 \F_0=\mathrm du\wedge\mathrm d\operatorname{Re}\Phi_0
       +(\partial_r\operatorname{Im}\Phi_0)\,\mathrm d\omega.
 \end{gathered}
\]
Here angular indices are raised with $\gcirc$.
The determinant and inverse of \eqref{eq:equality-metric-Maxwell-form}
give $\sqrt{|\det\g_0|}=\sqrt{\det\gcirc}$,
$(\F_0)^{ru}=\partial_r\operatorname{Re}\Phi_0$, and
$(\F_0)^{Au}=0$. The radial Maxwell equations therefore reduce to
\begin{equation}\label{eq:equality-radial-Maxwell}
 \begin{aligned}
 0&=\partial_\mu\bigl(\sqrt{|\det\g_0|}\,(\F_0)^{\mu u}\bigr)
       =\sqrt{\det\gcirc}\,\partial_r^2\operatorname{Re}\Phi_0,\\
 0&=(\mathrm d\F_0)_{rAB}
       =(\partial_r^2\operatorname{Im}\Phi_0)(\mathrm d\omega)_{AB}.
 \end{aligned}
\end{equation}

The values and derivatives in \eqref{eq:completed-potential-data}
give $\widehat s=S\cdot y+O(\rho^2)$ and
$\widehat m=1+O(\rho^2)$. Formula~\eqref{eq:completed-vacuum-potentials}
therefore gives
\begin{equation}\label{eq:equality-potential-limits}
 \frac{\operatorname{Re}\Phi_0}{r}\longrightarrow-1,\qquad
 \operatorname{Im}\Phi_0\longrightarrow-S\cdot\omega.
\end{equation}
Indeed \eqref{eq:cone-geodesic-data} gives $\rho/s\to1$,
and \eqref{eq:cone-affine-asymptotics} gives $rs\to1$
along each ray.
Equations~\eqref{eq:equality-radial-Maxwell} make $\Phi_0$
affine along each ray, and \eqref{eq:equality-potential-limits}
forces $\partial_r\Phi_0=-1$. We may therefore fix the affine
origin by $r=-\operatorname{Re}\Phi_0$. This change is independent
of $u$ and preserves $\partial_r$ and the normalization
\eqref{eq:equality-null-normalization}. It adds multiples of the
null generator to the angular tangents, so it also preserves
the metric \eqref{eq:equality-angular-metric}. We obtain
\begin{equation}\label{eq:equality-angular-potential}
 \Phi_0=-r+ic,\qquad c=-S\cdot\omega,\qquad
 \F_0=\mathrm dr\wedge\mathrm du.
\end{equation}
The inverse metric now gives $(\F_0)^{rA}=a^A$, so the angular
part of $\mathrm dc=i_{\partial_u}{}^*\F_0$ reads
$\mathrm dc=*_{\gcirc}a$. Applying $*_{\gcirc}^2=-1$, and using
\eqref{eq:harrison} at $q=0$, determines both remaining coefficients,
\begin{equation}\label{eq:equality-endpoint-data}
 a=-*_{\gcirc}\mathrm dc,\qquad
 f_0=|\Phi_0|^2-1=r^2+c^2-1.
\end{equation}

Rotating the angular coordinates so that $S=(0,0,J)$, we have
\[
 c=-J\cos\theta,\qquad
 a=-J\sin^2\theta\,\mathrm d\phi.
\]
Put $R=r+1$ and $\Delta=R^2-2R+J^2$.
Applying the inverse
Harrison relation \eqref{eq:inverse-harrison} to
\eqref{eq:equality-angular-potential}, and extracting the unchanged
orbit metric from \eqref{eq:equality-metric-Maxwell-form}, gives
\begin{equation}\label{eq:equality-inverse-data}
 \sigma=\frac{\Phi_0+1}{\Phi_0-1}
          =1-\frac2{R+iJ\cos\theta},\qquad
 \gamma=(\mathrm dr-a)^2+f_0\gcirc.
\end{equation}
Using $f_0=R^2-2R+J^2\cos^2\theta$ and the angular change
$\mathrm d\varphi=\mathrm d\phi+J\Delta^{-1}\mathrm dR$,
we complete the square in \eqref{eq:equality-inverse-data},
\begin{equation}\label{eq:equality-orbit-metric}
 \begin{aligned}
 \gamma
 &=f_0\,\mathrm d\theta^2+\mathrm dR^2
       +2J\sin^2\theta\,\mathrm dR\,\mathrm d\phi
       +\Delta\sin^2\theta\,\mathrm d\phi^2\\
 &=f_0\left(\frac{\mathrm dR^2}{\Delta}
                    +\mathrm d\theta^2\right)
       +\Delta\sin^2\theta
          \left(\mathrm d\phi+\frac J\Delta\mathrm dR\right)^2\\
 &=f_0\left(\frac{\mathrm dR^2}{\Delta}
                    +\mathrm d\theta^2\right)
                  +\Delta\sin^2\theta\,\mathrm d\varphi^2.
 \end{aligned}
\end{equation}
Taking real and imaginary parts of the Ernst potential in
\eqref{eq:equality-inverse-data} gives the norm and twist
\begin{equation}\label{eq:equality-vacuum-potentials}
 f=\frac{f_0}{R^2+J^2\cos^2\theta},\qquad
 \psi=\frac{2J\cos\theta}{R^2+J^2\cos^2\theta}.
\end{equation}
Substituting \eqref{eq:equality-vacuum-potentials} into the first
identity in \eqref{eq:stationary-orbit-form-identities}, with the
Hodge star of \eqref{eq:equality-orbit-metric}, determines the
stationary connection through
\begin{equation}\label{eq:equality-vacuum-connection}
 \begin{aligned}
 \mathrm d\left(\frac{2JR\sin^2\theta}{f_0}\,\mathrm d\varphi\right)
 &=\frac{2J\sin\theta}{f_0^2}
   \Bigl[-\sin\theta(R^2-J^2\cos^2\theta)\,\mathrm dR
         \\
 &\hspace{7em}{}+2R\Delta\cos\theta\,\mathrm d\theta\Bigr]
                                      \wedge\mathrm d\varphi\\
 &=f^{-2}\sin\theta\bigl(\psi_\theta\,\mathrm dR
                 -\Delta\psi_R\,\mathrm d\theta\bigr)
                                      \wedge\mathrm d\varphi\\
 &=-f^{-2}*_{\gamma}\mathrm d\psi.
 \end{aligned}
\end{equation}
By \eqref{eq:equality-vacuum-connection}, the difference from
the original vacuum connection is closed and hence exact on
the exterior sphere region. After a stationary time change,
substituting this connection, \eqref{eq:equality-vacuum-potentials},
and \eqref{eq:equality-orbit-metric} into
\eqref{eq:stationary-reduction} gives
\[
 \begin{split}
 \g={}&-\frac{f_0}{R^2+J^2\cos^2\theta}
       \left(\mathrm dt+\frac{2JR\sin^2\theta}{f_0}
                                      \mathrm d\varphi\right)^2\\
 &+(R^2+J^2\cos^2\theta)
       \left(\frac{\mathrm dR^2}{\Delta}+\mathrm d\theta^2
             +\frac{\Delta\sin^2\theta}{f_0}\,\mathrm d\varphi^2\right).
 \end{split}
\]
This is the Kerr metric with mass one and angular momentum $S$.
\end{proof}

\subsection{Deforming the cuts}

For a non-Kerr end, Lemma~\ref{lem:optical-equality} gives
$|C_a|_{\g_0}<4\pi$, and continuity preserves the area deficit
for small positive $q$. We must also make $H-|p|$ strictly
positive, however small $q$ is. The Schwarzschild spheres in
\eqref{eq:schwarzschild-family} show how the deformation supplies
this strict inequality. For the sphere $x=R$, put $L=1+qR$ and
$\lambda=qR/L$. Its induced metric is $L^2\gcirc$, and
\begin{equation}\label{eq:schwarzschild-expansions}
 LH=\left.\frac{L\sqrt{x^2-1}}{1+qx}
          \partial_x\log(1+qx)^2\right|_{x=R}
       =2\lambda\sqrt{1-R^{-2}},\qquad p=0.
\end{equation}
To preserve this gain for the actual cuts, we need an error
of order $\lambda/R$, which we obtain by comparing parameter
derivatives and integrating from $\lambda=0$.

Return to the inverted radius $r=\rho^{-1}$ and set $R=a^{-1}$.
In Schwarzschild, $\rho=\operatorname{arccoth}x$ and hence
$x=r+O(r^{-1})$.
Write the cone cut from \eqref{eq:double-cone-graphs} as
\begin{equation}\label{eq:cone-cut-smallness}
 r=Re^{u_R(\omega)},\qquad \tau=R^{-1}v_R(\omega),\qquad
 \|u_R\|_{C^m}+\|v_R\|_{C^m}\le C_mR^{-1}.
\end{equation}
For fixed large $R$, define $C_{R,q}$ by
\begin{equation}\label{eq:fixed-cone-path}
 \omega\longmapsto
 \left(\frac{L^2}{R}v_R(\omega),
                         Re^{u_R(\omega)},\omega\right).
\end{equation}
On $x=R$, the Schwarzschild metric has
$L^{-2}(\g_q)_{\tau\tau}=-(R^2-1)/L^4$, which selects the
time scale $L^2/R$. In coordinates
\[
 z=\log(r/R),\qquad y^0=R\tau/L^2,
\]
the immersion is fixed, $X=(v_R,u_R,\omega)$. Expressing
\eqref{eq:schwarzschild-family} in $x=Re^z$ gives the rescaled
comparison metric
\begin{equation}\label{eq:uniform-annulus}
 \begin{aligned}
 \g^{\mathrm{mod}}
 ={}&-\frac{e^{2z}-R^{-2}}{(1-\lambda+\lambda e^z)^2}
                      (\mathrm dy^0)^2\\
 &+(1-\lambda+\lambda e^z)^2
       \left(\frac{\mathrm dz^2}{1-R^{-2}e^{-2z}}+\gcirc\right),\\
 L^{-2}\g_q={}&\g^{\mathrm{mod}}+O_{C^m}(R^{-1}).
 \end{aligned}
\end{equation}
We prove the remainder bound also after one $\lambda$ derivative,
on $|z|\le\log2$ and uniformly for $0\le q\le1$.
The coefficient norms use these coordinates and angular charts,
with parameter derivatives taken at fixed $(z,y^0,\omega)$.

We first bound the two connections in their radial gauges.
The Jacobian from the inverted Cartesian coordinates $y/\rho^2$ back to $y$ is
$r^{-2}(I-2\omega\otimes\omega)$, with $\omega=y/\rho$.
Since $I-2\omega\otimes\omega$ is orthogonal, substituting
\eqref{eq:completion-coordinates} into
$\gamma=\Omega^{-2}\widehat\gamma$ gives, in those Cartesian coordinates,
\begin{equation}\label{eq:inverted-orbit-bounds}
 \partial^I(\gamma_{ij}-\delta_{ij})=O(r^{-1-|I|})
       \qquad(|I|\le m).
\end{equation}
Equations~\eqref{eq:completed-potential-data} and
\eqref{eq:completed-vacuum-potentials} give $f=1-2/r+O_m(r^{-2})$ and
$\psi=2(S\cdot\omega)r^{-2}+O_m(r^{-3})$ in the same chart.
Here $m$ is any fixed nonnegative integer, and the derivative
bounds follow from Taylor expansion of the analytic completed
fields $\widehat\gamma,\Omega,\widehat m,\widehat s$.
Thus \eqref{eq:completion-coordinates} bounds $\vartheta_0$,
and integrating the vacuum twist equation
\eqref{eq:stationary-orbit-form-identities} from infinity bounds
$\vartheta_1$ in the independent radial gauges of
Section~\ref{sec:charging},
\begin{equation}\label{eq:radial-connection-bounds}
 \begin{gathered}
 (\vartheta_0)_\rho=(\vartheta_1)_\rho=0,\qquad
 (\vartheta_0)_A=O(\rho^2),\\
 (\vartheta_1)_A
  =\int_r^\infty f^{-2}(*_\gamma\mathrm d\psi)_{sA}\,\mathrm ds
  =\int_r^\infty O(s^{-2})\,\mathrm ds=O(\rho).
 \end{gathered}
\end{equation}
The integrand and these bounds retain their orders after any
fixed number of $r\partial_r$ and angular derivatives.

To verify \eqref{eq:uniform-annulus}, use $\Phi_0$ from
\eqref{eq:completed-vacuum-potentials} and the identity
$\widehat m^2+\widehat s^2=1+\Omega$ in
\eqref{eq:completed-potential-data} to obtain
\begin{equation}\label{eq:annulus-denominator-identity}
 \frac{D_q}{4f}
  =\Omega|1-q\Phi_0|^2
  =q^2+(1+q^2)\Omega+2q\widehat m\sqrt\Omega.
\end{equation}
Inserting $\widehat m=1+O(\rho^2)$ from
\eqref{eq:completed-potential-data} and
$\Omega/\rho^2=1+O(\rho)$ from \eqref{eq:completion-coordinates}
into \eqref{eq:annulus-denominator-identity}, and using
\eqref{eq:radial-connection-bounds} in the connection formula
\eqref{eq:charged-connection}, we obtain
\begin{equation}\label{eq:annulus-coefficient-bounds}
 \begin{aligned}
 \frac{D_q}{4fL^2\rho^2}
 ={}&(1-\lambda)^2\frac\Omega{\rho^2}
  +2\lambda(1-\lambda)e^z\widehat m\frac{\sqrt\Omega}{\rho}
  +\lambda^2e^{2z}
  +\frac{\lambda^2}{R^2}\frac\Omega{\rho^2}\\
 ={}&(1-\lambda+\lambda e^z)^2+O_{C^m}(R^{-1}),\\
 \frac R{L^2}\vartheta_q
 ={}&R\left(1-\lambda-\frac\lambda R\right)^2\vartheta_0
       +\lambda(1-\lambda)\vartheta_1
       =O_{C^m}(R^{-1}).
 \end{aligned}
\end{equation}
At fixed $z$, the fields $\Omega,\widehat m,\vartheta_0,\vartheta_1$
are independent of $\lambda$, since $\rho=R^{-1}e^{-z}$.
The remaining factors in \eqref{eq:annulus-coefficient-bounds}
are quadratic polynomials in $\lambda$, so one $\lambda$ derivative
preserves the stated $O_{C^m}(R^{-1})$ remainders.
Since $1-\lambda+\lambda e^z\ge1/2$ and
$1-R^{-2}e^{-2z}=1+O_{C^m}(R^{-2})$, substitution in
\eqref{eq:charged-connection}, using
\eqref{eq:completion-coordinates}, proves
\eqref{eq:uniform-annulus} with its parameter derivative.

Write $\widetilde\g=L^{-2}\g_q$. With $h=X^*\widetilde\g$,
the mean curvature vector is
\begin{equation}\label{eq:cone-mean-curvature-vector}
 (\vec H)^\mu=h^{AB}\left(
 \partial_A\partial_BX^\mu
 -\Gamma(h)^C_{AB}\partial_CX^\mu
 +\Gamma^\mu_{\alpha\beta}
                \partial_AX^\alpha\partial_BX^\beta\right),
\end{equation}
where $\Gamma$ is the connection of $\widetilde\g$.
For $\g^{\mathrm{mod}}$ and $X_0=(0,0,\omega)$, the induced
metric is $\gcirc$, and the curvatures are given by
\eqref{eq:schwarzschild-expansions}.

For the actual cut, put $e_A=\partial_AX$ and choose normal vectors
\[
 \begin{aligned}
 V&=\partial_{y^0}
       -h^{AB}\widetilde\g(\partial_{y^0},e_A)e_B,\\
 N&=\partial_z-h^{AB}\widetilde\g(\partial_z,e_A)e_B
       -\frac{\widetilde\g(\partial_z,V)}{\widetilde\g(V,V)}V.
 \end{aligned}
\]
By \eqref{eq:uniform-annulus} and
\eqref{eq:cone-cut-smallness}, for large $R$,
\[
 \begin{gathered}
 h\ge\tfrac12\gcirc,\qquad
 |\det\widetilde\g|\ge\tfrac12\det\gcirc,\\
 -\widetilde\g(V,V),\ \widetilde\g(N,N)\ge\tfrac12
 \end{gathered}
\]
along $X$. The same bounds hold when the metric coefficients
and the immersion are interpolated linearly with the model. Set
$T_0=V/\sqrt{-\widetilde\g(V,V)}$ and
$\nu=N/\sqrt{\widetilde\g(N,N)}$, with future and outward
orientations. Then
$LH=-\widetilde\g(\vec H,\nu)$ and
$Lp=\widetilde\g(\vec H,T_0)$.
The formulas for $LH,Lp$ involve the metric and its first
derivatives and the first two derivatives of $X$. Their first
two derivatives with respect to these coefficients are bounded
by \eqref{eq:uniform-annulus} and the preceding lower bounds.
Since $\partial_\lambda X=0$, the chain rule differentiates
only the metric coefficients. Applying the mean value estimate
to these formulas gives
\begin{equation}\label{eq:cone-curvature-comparison}
 \begin{aligned}
 &\|(LH-2\lambda\sqrt{1-R^{-2}},Lp)\|_{C^1_\lambda C^0_\omega}\\
 &\qquad\le C\left(
   \|\widetilde\g-\g^{\mathrm{mod}}\|_{C^1_\lambda C^2}
            +\|u_R\|_{C^3}+\|v_R\|_{C^3}\right)\le C/R.
 \end{aligned}
\end{equation}
Here $C$ is independent of $R$, and the parameter norm uses
$0\le\lambda\le R/(1+R)$. The derivative part of
\eqref{eq:cone-curvature-comparison} is
\begin{equation}\label{eq:cone-expansion-variation}
 \partial_\lambda(LH)=2\sqrt{1-R^{-2}}+O(R^{-1}),\qquad
 \partial_\lambda(Lp)=O(R^{-1}).
\end{equation}
Integrating \eqref{eq:cone-expansion-variation} from
$\lambda=0$ and applying the triangle inequality to $Lp$, we obtain
\begin{equation}\label{eq:strict-cone-path}
 \begin{aligned}
 L(H-|p|)&\ge(H-|p|)_{q=0}
       +\int_0^\lambda
          \bigl(\partial_\mu(LH)-|\partial_\mu(Lp)|\bigr)
                                      \,\mathrm d\mu\\
        &\ge(H-|p|)_{q=0}
       +\bigl(2\sqrt{1-R^{-2}}-C/R\bigr)\lambda\\
        &\ge\bigl(2\sqrt{1-R^{-2}}-C/R\bigr)\lambda>0
                     \qquad(q>0),
 \end{aligned}
\end{equation}
In the integral, $R$ is fixed and $L,q,H,p$ are evaluated
at parameter $\mu$ along \eqref{eq:fixed-cone-path}.
The initial term is nonnegative by \eqref{eq:weak-cone-spheres}.
Fixing $R$ sufficiently large in \eqref{eq:strict-cone-path},
we have $H>|p|$ along the path \eqref{eq:fixed-cone-path}
for every $q>0$. We can then choose $\delta>0$ small enough
that continuity preserves the area gap at $q=0$ in the form
$|C_{R,\delta}|<4\pi(1-\delta^2)$.

To join $C_{R,1}$ to the original maximal slice, we compute
that slice's height in the radial gauge. Begin with the time
for which it is $t=0$. By \eqref{eq:stationary-reduction},
$\vartheta_i=-\g_{0i}/f$ and
$\gamma_{ij}=f\g_{ij}+\g_{0i}\g_{0j}$, so
\eqref{eq:AF} gives
$|(\nabla^\gamma)^m\vartheta|_\gamma=O(r^{-2-m})$ for
$0\le m\le3$, where $\nabla^\gamma$ is the connection of $\gamma$.
These tensor bounds hold in the inverted chart because both
radii are comparable to the $\gamma$-distance from a fixed
compact set. Using $|\partial_r|_\gamma=1+O(r^{-1})$ and
\eqref{eq:inverted-orbit-bounds} to convert angular derivatives
to covariant derivatives gives
\begin{equation}\label{eq:radial-gauge-derivatives}
 \|\vartheta_r(r,\cdot)\|_{C^3(\mathbb S^2)}
 \le C\sum_{m=0}^3 r^m
       \sup_{\rho=r^{-1}}|(\nabla^\gamma)^m\vartheta|_\gamma
 \le Cr^{-2}.
\end{equation}
Set $t'=t+\chi$, with
\[
 \chi(r,\omega)=-\int_r^\infty\vartheta_r(s,\omega)\,\mathrm ds.
\]
Then $\partial_r\chi=\vartheta_r$ makes $\vartheta'_r=0$
by the gauge law following \eqref{eq:stationary-reduction},
and integration of \eqref{eq:radial-gauge-derivatives} gives
\begin{equation}\label{eq:original-slice-height}
 t'=\chi(r,\omega),\qquad
 \|\chi(r,\cdot)\|_{C^3(\mathbb S^2)}\le Cr^{-1}.
\end{equation}
Returning to the radial time coordinate denoted by $t$,
which equals $\tau$ at $q=1$, the required joining path is
\[
 r=Re^{b u_R},\qquad
 t=(1-b)\chi(r,\omega)+b\frac{(1+R)^2}{R}v_R,
 \qquad 0\le b\le1.
\]
By \eqref{eq:cone-cut-smallness} and
\eqref{eq:original-slice-height}, this path has radial and time
displacements bounded in $C^3$ independently of $R,b$.
Substituting in \eqref{eq:null-expansions} and using
\eqref{eq:AF}, we obtain $H=2/R+O(R^{-2})$ and $p=O(R^{-2})$,
so the joining spheres also satisfy $H>|p|$ for large $R$.
Their projections, like those of \eqref{eq:fixed-cone-path},
are radial graphs enclosing the fixed inner sphere. We may
therefore join the paths by a smooth reparametrization to
obtain all the conclusions of Proposition~\ref{prop:end-spheres}.

\section{Maximal hypersurfaces along the Harrison family}
\label{sec:bands}

We prove Proposition~\ref{prop:band}. Assuming the existence of
the maximal family, we first deduce the area inequality
\eqref{eq:area-bound} from the hypersurface constraints and
stability of a minimal surface. The remaining subsections
construct the family by continuation.

Let $T_0$ be the future unit normal to a maximal hypersurface,
and define its electric and magnetic fields by
$i_E\mathrm{vol}_h=({}^*\F_q)|_{T\Sigma}$ and
$i_B\mathrm{vol}_h=\F_q|_{T\Sigma}$.
Combining \eqref{eq:Einstein-Maxwell} with the hypersurface
constraints and maximality gives
\begin{equation}\label{eq:charged-constraint}
 \begin{aligned}
 \R(T_0,T_0)&=|E|^2+|B|^2,\\
 R_h&=2\R(T_0,T_0)+|k|^2\ge2|E|^2,\\
 (\Div_hE)\mathrm{vol}_h
     &=\mathrm d(({}^*\F_q)|_{T\Sigma})=0.
 \end{aligned}
\end{equation}

To find a stable minimal surface homologous to $C_s$, we first
check mean convexity of $\partial\Sigma(s)$. At $\mathcal S_0$,
both null second fundamental forms vanish by
Proposition~\ref{prop:charging}, so its mean curvature in
$\Sigma(s)$ is zero. At $C_s$, write the outward normal within
$\Sigma(s)$ in an orthonormal normal frame $(e_0,\nu_0)$ as
$\nu=\sinh\eta\,e_0+\cosh\eta\,\nu_0$. Then
\[
 H_{C_s\subset\Sigma(s)}
 =H_s\cosh\eta-p_s\sinh\eta
 \ge(H_s-|p_s|)\cosh\eta>0.
\]
Together with $H_2(\Sigma(s),C_s;\mathbb Z)=0$ from the
annular topology, this mean convexity allows us to apply the
minimizing construction of \cite{CruzMendes}. We obtain a
connected closed embedded two-sided stable minimal surface
$\Gamma$ with $[\Gamma]=[C_s]$ and $|\Gamma|\le|C_s|$.
By \eqref{eq:charged-constraint},
$\Gamma$ has the same electric flux as $C_s$, namely
$4\pi\sqrt{1-q(s)^2}$ by \eqref{eq:charged-flux-limits} and its
continuation to $q=0$.

We apply the area--charge argument
of~\cite{Gibbons} and
\cite{DainJaramilloReiris}.
If $\nu$, $A_\Gamma$, and $\chi(\Gamma)$ are the normal, second
fundamental form, and Euler characteristic of $\Gamma$, stability
with test function $1$, the Gauss equation, and Gauss--Bonnet give
\begin{equation}\label{eq:stable-surface-integral}
 \begin{aligned}
 0&\le-\int_\Gamma
       \bigl(\Ric(h)(\nu,\nu)+|A_\Gamma|^2\bigr)\dd A\\
  &=2\pi\chi(\Gamma)
       -\frac12\int_\Gamma(R_h+|A_\Gamma|^2)\dd A.
 \end{aligned}
\end{equation}
Combining \eqref{eq:stable-surface-integral} with
$R_h\ge2|E|^2$ from \eqref{eq:charged-constraint} and
$\chi(\Gamma)\le2$ gives $\int_\Gamma|E|^2\dd A\le4\pi$.
The conserved electric flux and Cauchy--Schwarz therefore give
\begin{equation}\label{eq:area-charge-comparison}
 16\pi^2(1-q(s)^2)
 =\left(\int_\Gamma h(E,\nu)\dd A\right)^2
 \le|\Gamma|\int_\Gamma|E|^2\dd A
 \le4\pi|\Gamma|\le4\pi|C_s|.
\end{equation}
The comparison \eqref{eq:area-charge-comparison} proves
\eqref{eq:area-bound} once the maximal hypersurface has been constructed.

\subsection{Continuation and the lapse}

We now construct the family of maximal hypersurfaces. All constants may depend on
the fixed vacuum spacetime and the prescribed boundary path.
We suppress $s$ when considering one hypersurface, and write
$h,k,T_0$ for its induced data. Unless another metric is indicated,
$\D$ and $\R_{\alpha\beta}$ refer to $\g_q$, and $\nabla$ to $h$.

For a normal displacement $wT_0$, the negative linearization
of mean curvature is
\begin{equation}\label{eq:Jacobi}
 w\longmapsto\Delta_\Sigma w
       -\bigl(|k|^2+\R(T_0,T_0)\bigr)w.
\end{equation}
Its Dirichlet kernel is zero by
\eqref{eq:charged-constraint} and the maximum principle.
Applying the implicit-function theorem to the mean curvature
of normal graphs therefore gives a branch through the
truncation of $\Sigma_1$, which continues while the
hypersurfaces remain uniformly smooth and spacelike.

Stationary translations have normal velocity
$n=-\g_q(\T,T_0)$ and preserve maximality. Applying
\eqref{eq:Jacobi} to this variation gives
\begin{equation}\label{eq:charged-lapse-equation}
 \Delta_\Sigma n=\bigl(|k|^2+\R(T_0,T_0)\bigr)n,
 \qquad n|_{\mathcal S_0}=0,\qquad n|_{C_s}>0.
\end{equation}
By the maximum principle, \eqref{eq:charged-lapse-equation}
and \eqref{eq:charged-constraint} give $n>0$ in the interior.
We claim that the resulting local diffeomorphism $\pi$ makes
$\Sigma(s)^\circ$ a graph over $\mathcal A_s^\circ$.
For properness away from the boundary projections, apply
Lemma~\ref{lem:harrison-collar} at $q=1$. Its invariant radial
coordinate restricts to a smooth defining function of
$\mathcal S_0$ on $\overline\Sigma_1$, so a sequence approaching
$\mathcal S_0$ projects towards $\mathcal S_0$ as well. Compactness and the
prescribed outer boundary then give the required properness.
The projection cannot meet the unbounded component outside
$\pi(C_s)$, since properness and openness would make it cover
that component, contrary to compactness. Nor can an interior
point project to $\pi(C_s)$, by openness. It follows that $\pi$
is a covering of $\mathcal A_s^\circ$, and connectedness of
$\Sigma(s)^\circ$ together with simple connectedness of the
annulus makes it a diffeomorphism. Thus we may write
$\Sigma(s)^\circ=\{t=u_s(x)\}$.

By \eqref{eq:charged-lapse-equation} and
\eqref{eq:charged-constraint}, it suffices to bound $n$ on
$C_s$. Although $C_s$ is prescribed, the hypersurface normal
along it is unknown. The expansion inequality bounds its
inclination through a negative quadratic term in $\partial_\nu n$.

Let $Y$ be the tangential part of $\T$ on $C_s$, and put
$e_0=(\T-Y)/a$, where $a^2=-\g_q(\T-Y,\T-Y)$.
Complete this future unit normal by the outward unit normal
$\nu_0$, and set
$H_0=\Div_{C_s}\nu_0$, $p_0=-\Div_{C_s}e_0$, and
$\omega_A=\g_q(\D_{e_A}e_0,\nu_0)$,
where $(e_A)_{A=1,2}$ is a local orthonormal frame in $TC_s$.
The expansion condition gives $H_0>|p_0|$. The decompositions
\[
 T_0=\cosh\eta\,e_0+\sinh\eta\,\nu_0,
 \qquad \nu=\sinh\eta\,e_0+\cosh\eta\,\nu_0,
\]
give, with $X$ the tangential part of $\T$ on $\Sigma$ and
$\operatorname{tr}k=0$,
\begin{equation}\label{eq:boundary-boost-data}
 \begin{aligned}
 n&=a\cosh\eta,\qquad X_\nu=-a\sinh\eta,\\
 k_{\nu\nu}&=-\sum_{A=1}^2k_{AA}
              =H_0\sinh\eta-p_0\cosh\eta,\\
 k_{\nu A}&=-\g_q(\D_{e_A}T_0,\nu)
              =-\nabla_A\eta-\omega_A.
 \end{aligned}
\end{equation}
The Killing equation cancels the dependence on $\eta$ in
\begin{equation}\label{eq:boundary-Killing-cancellation}
 \g_q(\D_\nu\T,T_0)
 = (\cosh^2\eta-\sinh^2\eta)
        \g_q(\D_{\nu_0}\T,e_0)
 =\g_q(\D_{\nu_0}\T,e_0).
\end{equation}
Thus \eqref{eq:boundary-boost-data} and
\eqref{eq:boundary-Killing-cancellation} give
\begin{equation}\label{eq:boundary-lapse}
 \begin{aligned}
 \partial_\nu n
 &=-\g_q(\D_\nu\T,T_0)+k(\nu,X)\\
 &=-\g_q(\D_{\nu_0}\T,e_0)-aH_0\sinh^2\eta
       +ap_0\sinh\eta\cosh\eta\\
 &\quad-Y\cdot(\mathrm d\eta+\omega).
 \end{aligned}
\end{equation}

At a boundary maximum with $|\sinh\eta|\ge1$,
\eqref{eq:boundary-boost-data} gives
$\mathrm d\eta=-\coth\eta\,\mathrm d\log a$ and
$|\coth\eta|\le\sqrt2$. Since
$|\sinh\eta\cosh\eta|\le\sinh^2\eta+1/2$,
\eqref{eq:boundary-lapse} yields
\begin{equation}\label{eq:boundary-boost-bound}
 0\le\partial_\nu n
 \le-a(H_0-|p_0|)\sinh^2\eta+C
 \le-c\sinh^2\eta+C.
\end{equation}
Smoothness of the prescribed compact boundary path and
$a(H_0-|p_0|)>0$ give uniform $c>0$ and $C$.
Consequently $\eta$ is bounded at a maximum, and
$n=a\cosh\eta$ gives
\begin{equation}\label{eq:lapse-bound}
 0<n\le C.
\end{equation}
On a fixed outer collar, where $\T$ is uniformly timelike,
\eqref{eq:lapse-bound} gives
\begin{equation}\label{eq:outer-collar-normal}
 -\g_q(T_0,\T/\sqrt{-\g_q(\T,\T)})\le C.
\end{equation}
This bounds the normals and stationary graph slopes there.
Integrating from the prescribed boundary also bounds $u_s$
on this collar.

\subsection{Height estimates}

The definition of $\DOC$ provides incoming and outgoing
timelike curves for $\g$, while the graphs are spacelike for
$\g_q$. To use these curves in the height argument of
Chru\'sciel and
Wald~\cite{ChruscielWald}, we apply \eqref{eq:inverse-charged}
to the graph conormal. Put $G_s=t-u_s\circ\pi$.
Since $\mathrm dG_s(\T)=1$, that identity gives
\begin{equation}\label{eq:cone-pullback}
 \g^{-1}\left(\mathrm dG_s-\frac{W_q}{D_q(p)},
              \mathrm dG_s-\frac{W_q}{D_q(p)}\right)
 =\frac{D_q}{4}\g_q^{-1}(\mathrm dG_s,\mathrm dG_s)<0.
\end{equation}
At the chosen zero $p\in\mathcal S_0$ of $\T$,
$D_q(p)=(1+q)^2+(1-q)^2\psi(p)^2\ge1$, and
\eqref{eq:charged-one-form} gives
$(\mathrm dG_s-W_q/D_q(p))(\T)=D_q/D_q(p)>0$.
Since $\T$ is future timelike at the outer boundary, the
corrected covector is positive on future causal vectors there.
Timelikeness and connectedness preserve this sign throughout
the annulus. Thus every nonconstant future causal path
$c:z_-\to z_+$ whose spatial projection stays there satisfies
$0<\int_c(\mathrm dG_s-W_q/D_q(p))$, or equivalently
\begin{equation}\label{eq:height-access}
 D_q(p)\bigl[u_s(\pi z_+)-u_s(\pi z_-)\bigr]
 <\int_c\bigl(D_q(p)\,\mathrm dt-W_q\bigr).
\end{equation}

For an interior point $x$, the original domain of outer
communications gives future timelike paths $c_+$ from $(0,x)$
to the outer stationary cylinder and $c_-$ from that cylinder
to $(0,x)$. Truncate $c_+$ at its first crossing and $c_-$
at its last crossing, so their spatial projections stay in
the annulus. With $y_+,y_-$ the boundary projections,
\eqref{eq:height-access} gives the two bounds
\begin{equation}\label{eq:two-sided-height-access}
 \begin{split}
 u_s(y_+)-\frac1{D_q(p)}
       \int_{c_+}\bigl(D_q(p)\,\mathrm dt-W_q\bigr)
 &<u_s(x)\\
 &<u_s(y_-)+\frac1{D_q(p)}
       \int_{c_-}\bigl(D_q(p)\,\mathrm dt-W_q\bigr).
 \end{split}
\end{equation}
The boundary heights are prescribed by $C_s$. Choose the
crossings transverse, so the truncated paths persist as $x$
and $C_s$ vary. After smoothly identifying the moving annuli,
a finite cover of the compact parameter interval and any
$K\Subset\mathcal A^\circ$ gives such path families in fixed
compact sets, with bounded lengths in an auxiliary Riemannian
metric. Formula~\eqref{eq:regular-charged} bounds the coefficients
of $D_q(p)\,\mathrm dt-W_q$ on these sets uniformly in $q$.
Together with $D_q(p)\ge1$ and the bounded boundary heights,
\eqref{eq:two-sided-height-access} therefore gives
\begin{equation}\label{eq:interior-height}
 \sup_{x\in K}|u_s(x)|\le C_K
       \qquad(K\Subset\mathcal A^\circ),
\end{equation}
uniformly in $s$.

Near $\mathcal S_0$, use the corotating coordinates of
Lemma~\ref{lem:harrison-collar}, with $\partial_t=\K$.
To write the hypersurface as a graph in these coordinates, we show that it
meets each orbit of $\K$ once. The field is timelike for small
$\rho>0$, and \eqref{eq:local-circle} with the commuting flows gives
\begin{equation}\label{eq:corotating-graph}
 G_s(\Psi_{\lambda,\K}p)
 =\lambda+G_s(\Psi_{\Omega_H\lambda,\Z}p).
\end{equation}
The boundedness of the circle orbit makes the right side tend
to $\pm\infty$ with $\lambda$, whereas timelikeness makes the
left side strictly increasing. It therefore has exactly one
zero, giving a corotating graph $t=v_s(\rho,y)$ whose height
at any fixed $\rho=\rho_0$ is bounded by
\eqref{eq:interior-height}.

We extend the bound for $v_s$ to $\rho=0$ by radial barriers.
Set $\bar h=\mathrm d\rho^2+h_\rho$ and, with inner products
taken in $\bar h$, the metric \eqref{eq:ADM-collar}
gives the graph volume density and its derivative as
\begin{equation}\label{eq:collar-graph-volume}
 L(\xi)=\sqrt{(1+\langle\beta,\xi\rangle)^2-a^2|\xi|^2},
 \qquad
 L_\xi(\xi)=
 \frac{(1+\langle\beta,\xi\rangle)\beta-a^2\xi}{L(\xi)}.
\end{equation}
Maximality is $\operatorname{div}L_\xi(\nabla v)=0$.
Since $|\beta|<a$, differentiating \eqref{eq:collar-graph-volume}
shows that this operator is strictly decreasing on the convex
set of future spacelike graph gradients,
\begin{equation}\label{eq:collar-volume-concavity}
 D^2L(\xi)[z,z]
 =\frac{\langle\beta,z\rangle^2-a^2|z|^2
                  -\langle L_\xi(\xi),z\rangle^2}{L(\xi)}<0
       \qquad(z\ne0).
\end{equation}
Take
$b_+=\max_{\rho=\rho_0}v+\rho_0-\rho$ and
$b_-=\min_{\rho=\rho_0}v-\rho_0+\rho$.
Substituting these gradients into \eqref{eq:collar-graph-volume},
using $\beta^\rho=0$ and the differentiated bounds for
$a,\beta,h_\rho$ in \eqref{eq:collar-orders}, gives
\begin{equation}\label{eq:collar-barrier-divergence}
 \begin{aligned}
 L_\xi(\nabla b_\pm)
   &=\frac{\beta\pm a^2\partial_\rho}{\sqrt{1-a^2}},\\
 \operatorname{div}L_\xi(\nabla b_\pm)
   &=\pm\frac1{\sqrt{\det h_\rho}}
       \partial_\rho\left(
         \frac{\sqrt{\det h_\rho}\,a^2}{\sqrt{1-a^2}}\right)
       +\operatorname{div}_{h_\rho}
          \left(\frac\beta{\sqrt{1-a^2}}\right)\\
   &=\pm2\kappa^2\rho+O(\rho^2).
 \end{aligned}
\end{equation}

To compare $v$ with these barriers without prescribing its
value at $\rho=0$, we estimate the flux through a small inner
boundary. For each fixed smooth spacelike hypersurface, regular
normal coordinates give
$x^0=a_0(y)x^1+O((x^1)^2)$ with a smooth remainder and $|a_0|<1$.
Using \eqref{eq:collar-boost-coordinates} and then
\eqref{eq:collar-graph-volume}, we obtain
\begin{equation}\label{eq:collar-inner-flux}
 \begin{gathered}
 \rho=x^1\sqrt{1-(x^0/x^1)^2},\qquad
 \left.\partial_{x^1}\rho\right|_{x^1=0}
        =\sqrt{1-a_0(y)^2}>0,\\
 v=\kappa^{-1}\operatorname{arctanh}(x^0/x^1)
       =\kappa^{-1}\operatorname{arctanh}a_0(y)+O(\rho),\\
 L(\nabla v)=1+O(\rho^2),\qquad
 L_\xi(\nabla v)\cdot\partial_\rho
      =-\frac{a^2\partial_\rho v}{L(\nabla v)}=O(\rho^2).
 \end{gathered}
\end{equation}
The first row makes $x^1$ a smooth function of $(\rho,y)$, so
$v$ is smooth up to $\rho=0$. The derivative bounds and the
constants in \eqref{eq:collar-inner-flux} may vary along the branch.
Subtracting \eqref{eq:collar-barrier-divergence} from the equation
for $v$ and testing with $\varphi=(v-b_+)_+$ on
$\epsilon<\rho<\rho_0$, we can therefore integrate by parts.
The upper-boundary term vanishes because $\varphi=0$ there,
and the inner-boundary term is $O(\epsilon^2)$ by
\eqref{eq:collar-inner-flux} and
\eqref{eq:collar-barrier-divergence}. The concavity inequality
\eqref{eq:collar-volume-concavity} then gives
\begin{equation}\label{eq:collar-barrier-comparison}
 \begin{aligned}
 0&\le-\int
  \langle L_\xi(\nabla v)-L_\xi(\nabla b_+),\nabla\varphi\rangle
                                      \,\mathrm d\mu_{\bar h}\\
  &=-\int\operatorname{div}L_\xi(\nabla b_+)\varphi
                                      \,\mathrm d\mu_{\bar h}
       +O(\epsilon^2).
 \end{aligned}
\end{equation}
Let $\epsilon\downarrow0$ in \eqref{eq:collar-barrier-comparison}
for the fixed hypersurface. By \eqref{eq:collar-barrier-divergence},
the integral on the right is strictly negative unless $\varphi=0$. Thus
$v\le b_+$, with a bound independent of its height at $\rho=0$.
The lower barrier proves
\begin{equation}\label{eq:collar-height}
 \min_{\rho=\rho_0}v-\rho_0\le v\le
                    \max_{\rho=\rho_0}v+\rho_0.
\end{equation}
By \eqref{eq:collar-boost-coordinates}, the bound
\eqref{eq:collar-height} puts the hypersurfaces, after shortening
the collar, in a common regular normal neighbourhood, where
\begin{equation}\label{eq:wedge}
 |x^0|\le\theta x^1\quad\hbox{for some }\theta<1.
\end{equation}

\subsection{Closing the continuation argument}

The bounds for $u_s$ in \eqref{eq:interior-height} and $v_s$ in
\eqref{eq:collar-height} keep the hypersurfaces in a common
compact set. The remaining difficulty is to control their
normals where $\T$ is spacelike and thereby preserve
spacelikeness in the limit.

\begin{lemma}\label{lem:maximal-compactness}
Under the hypotheses of Proposition~\ref{prop:band}, let
$s_j\to s_*$ and let $\Sigma_j$ be compact spacelike maximal
stationary graphs with boundary $\mathcal S_0\sqcup C_{s_j}$
and topology $\mathbb S^2\times[0,1]$.
Assume their lapses and stationary heights satisfy
\eqref{eq:lapse-bound} and \eqref{eq:interior-height} uniformly.
Then a subsequence converges smoothly, including at the boundary,
to a compact spacelike maximal stationary graph $\Sigma_*$ with
boundary $\mathcal S_0\sqcup C_{s_*}$ and the same topology.
\end{lemma}

We explain how the lapse and height bounds enter the proof,
which is given in Appendix~\ref{app:maximal-compactness}.
Write $\g_j=\g_{q(s_j)}$, so that $\g_j\to\g_{q(s_*)}$
smoothly on compact sets by Proposition~\ref{prop:charging},
and extend the normals and lapses by stationary translations. The translated hypersurfaces
are the level sets of $G_j=t-u_j(x)$ with $\T G_j=1$, whence
\begin{equation}\label{eq:translated-graph-time}
 \mathrm dG_j=-n_j^{-1}T_{0,j}^{\flat},\qquad
 -\g_j^{-1}(\mathrm dG_j,\mathrm dG_j)=n_j^{-2}\ge C^{-2}.
\end{equation}
The height bound \eqref{eq:interior-height} controls the
oscillation of $G_j$ on compact sets. Together with
\eqref{eq:translated-graph-time}, it bounds proper time in the
limiting metric by \eqref{eq:limit-proper-time} below.
Lemma~\ref{prop:compact-causal} converts this bound into a fixed
smooth function with timelike gradient, allowing us to apply
Bartnik's gradient estimate in
Lemma~\ref{lem:temporal-gradient} to bound the normals on
compact interior sets. Near $\mathcal S_0$,
the corresponding input is $|x^0|\le\theta x^1$ with
$\theta<1$ from \eqref{eq:wedge}. The proof in
Appendix~\ref{app:maximal-compactness} carries these normal
bounds through to smooth spacelike convergence, including at
the boundary.

If the branch had an endpoint before $s=1$,
Lemma~\ref{lem:maximal-compactness} would give a smooth spacelike
limit there. Since the Dirichlet operator \eqref{eq:Jacobi}
is invertible at this limit, the implicit-function theorem
would extend the branch through the endpoint, with local
uniqueness ensuring agreement with the original family.
Consequently the family exists for $0\le s\le1$, as required
in Proposition~\ref{prop:band}.

\section{Staticity and small angular momentum}
\label{sec:finish}

Let $\eps_{\mathrm{AIK}}(A)>0$ be the threshold in
\cite{AIKsmall} for $\overline A\le A$. To prove
Theorem~\ref{thm:main}, it suffices to find $j_*(A)>0$ such that
its hypotheses with $M=1$, $\overline A\le A$, and $J<j_*(A)$ imply
\begin{equation}\label{eq:horizon-smallness}
 \sup_{\mathcal S_0}\g(\T,\T)<\eps_{\mathrm{AIK}}(A)^2.
\end{equation}
We obtain this estimate by passing to the Schwarzschild limit
as $J\to0$.

A direct extension of the Kerr isometry across the boundary
of $\{f>0\}$ would require the defining function $-f$ to satisfy
the conditional pseudoconvexity inequality used
in~\cite{IKuniqueness},
\[
 \D^2(-f)(V,V)<0
 \quad\hbox{if}\quad
 \g(V,V)=V(f)=\g(\T,V)=0.
\]
At a rotating Kerr ergosurface, the vector $V=\T$ is
null there and satisfies both orthogonality conditions. The
Killing equation and $\T(f)=0$ give
\begin{equation}\label{eq:stationary-norm-hessian}
 \D_\T\T=\tfrac12\operatorname{grad}_{\g}f,\qquad
 \D^2(-f)(\T,\T)=\tfrac12\g^{-1}(\mathrm df,\mathrm df).
\end{equation}
For mass-one Kerr with $0<J<1$, substituting the equatorial
Boyer--Lindquist values at $R=2$ into
\eqref{eq:stationary-norm-hessian} gives
\[
 \mathrm df=\tfrac12\mathrm dR,\qquad
 \g^{RR}=\tfrac14J^2,\qquad
 \D^2(-f)(\T,\T)=\tfrac1{32}J^2>0.
\]
We therefore continue the scalar Kerr identities through the
component where $\T$ is timelike and pass to a limit as $J\to0$.

\subsection{The Kerr identities and the twist}

On the Kerr end of Theorem~\ref{thm:kerr-end}, the coordinates
and potential in \eqref{eq:equality-inverse-data} give
\[
 P=(1-\sigma)^{-1}=\frac{R+iJ\cos\theta}{2},\qquad
 z=\operatorname{Im}P.
\]
With $\Delta=R^2-2R+J^2$, the metric gives
\[
 \g^{RR}=\frac\Delta{4|P|^2},\qquad
 \g^{\theta\theta}=\frac1{4|P|^2},
\]
and hence, using $f$ from \eqref{eq:equality-vacuum-potentials},
\begin{equation}\label{eq:Kerr-potential-gradients}
 \begin{aligned}
 \D_\alpha z\D^\alpha z
   &=\frac{J^2\sin^2\theta}{16|P|^2}
     =\frac{J^2/4-z^2}{4|P|^2},\\
 \D_\alpha\sigma\D^\alpha\sigma
   &=\frac{\Delta-J^2\sin^2\theta}
           {|P|^2(R+iJ\cos\theta)^4}
     =\frac{4f}{(R+iJ\cos\theta)^4}.
 \end{aligned}
\end{equation}
Writing $\cF^2=\cF_{\mu\nu}\cF^{\mu\nu}$ and using
${}^*\cF=-i\cF$ and $\mathrm d\sigma=2i_\T\cF$, we obtain
\begin{equation}\label{eq:Ernst-self-dual-contraction}
 \begin{gathered}
 \cF_{\alpha\rho}\cF_\beta{}^\rho
       =\tfrac14\g_{\alpha\beta}\cF^2,\\
 \D_\rho\sigma\D^\rho\sigma
       =4\T^\alpha\T^\beta\cF_{\alpha\rho}\cF_\beta{}^\rho
       =\g(\T,\T)\cF^2=-f\cF^2.
 \end{gathered}
\end{equation}
The first line of \eqref{eq:Kerr-potential-gradients} gives
the identity for $z$. Substituting its second line into
\eqref{eq:Ernst-self-dual-contraction} gives, on the Kerr end,
\begin{equation}\label{eq:Kerr-identities}
 4|P|^2\D_\alpha z\D^\alpha z+z^2=\frac{J^2}{4},\qquad
 \cF^2=-\frac{(1-\sigma)^4}{4}.
\end{equation}
These are the identities of Mars~\cite{Mars}, in the normalization
of~\cite{IKuniqueness}.

Let $U$ be the component of $\{f>0\}\cap\Sigma_1$ containing
the end. Timelikeness of $\T$ gives analytic stationary coordinates
by~\cite{MullerHagen}, in which $\sigma$ is analytic because
$\mathrm d\sigma=2i_\T\cF$. Since $f<1$ by \eqref{eq:fbound},
$P=(1-\sigma)^{-1}$ has no singularity on $U$.
Theorem~\ref{thm:kerr-end} and analytic continuation therefore
give \eqref{eq:Kerr-identities} throughout $U$. As $\T(z)=0$,
the gradient of $z$ lies in the positive definite orthogonal
complement of $\T$, and the first identity gives
\begin{equation}\label{eq:twist-bound-Kerr}
 |z|\le\tfrac12J\quad\hbox{on }U.
\end{equation}
For $J<1$, this bounds the twist itself. Indeed, since $0<f<1$,
a point where $|\psi|=1$ would satisfy
\[
 1=|z|\bigl((1-f)^2+1\bigr)\le J<1.
\]
Connectedness of $U$ and the normalization $\psi\to0$ at infinity
therefore give $|\psi|<1$ throughout $U$, and
\eqref{eq:twist-bound-Kerr} implies
\begin{equation}\label{eq:twist-smallness}
 |\psi|=|z|\bigl((1-f)^2+\psi^2\bigr)\le J
       \quad\hbox{on }U.
\end{equation}

\subsection{A staticity criterion}

We prove staticity from vanishing twist on the timelike component
containing the end, assuming that $f$ has nonvanishing differential
at its interior boundary. The Schwarzschild scalar identity will
verify this condition in the application.

\begin{lemma}\label{lem:zero-spin-staticity}
Let $\g$ and $\T$ be $C^5$ stationary vacuum fields near a spacelike
maximal hypersurface
$\overline\Sigma_1\simeq[1,\infty)\times\mathbb S^2$ with boundary
$\mathcal S_0$. Write $\T=nT_0+X$, where $T_0$ is the future
unit normal, and $f=-\g(\T,\T)$.
Assume that its end is the zero-time section of
a stationary chart satisfying \eqref{eq:AF} through five
derivatives with $M=1$, that $\T$ is tangent to $\mathcal S_0$, and
that $n>0$ on $\Sigma_1$. Let $U$ be the component of
$\{f>0\}\cap\Sigma_1$ containing the end and set
$F=\tfrac12\mathrm d\T^\flat$. If
\begin{equation}\label{eq:staticity-hypotheses}
 \begin{gathered}
 \T^\flat\wedge F=0\quad\hbox{on }U,\\
 \left.\mathrm df\right|_{T\Sigma_1}\ne0
       \quad\hbox{on }\partial_{\Sigma_1}U,
 \end{gathered}
\end{equation}
then
\[
 \T|_{\mathcal S_0}=0,\qquad k=0,\qquad X=0
       \quad\hbox{on }\overline\Sigma_1.
\]
In particular $f=n^2>0$ throughout $\Sigma_1$, and the stationary
metric is static in the region swept out by this slice.
\end{lemma}

\begin{proof}
We first prove $\overline U\cap\mathcal S_0\ne\varnothing$.
The $C^5$ fields satisfy \eqref{eq:KID}, with $n=0$ on
$\mathcal S_0$ by tangency of $\T$ and $n\to1$ at infinity.
The maximum-principle argument proving \eqref{eq:lapse-boundary}
therefore gives $\partial_{\nu_{\mathcal S_0}}n>0$.
On $U$, contracting $\T^\flat\wedge F=0$ with $\T$ gives
\begin{equation}\label{eq:static-two-form}
 F=-\frac1{2f}\T^\flat\wedge\mathrm df,
 \qquad F\wedge F=0\quad\hbox{on }U.
\end{equation}
By continuity, $F\wedge F=0$ also holds on $\overline U$.
Suppose now that $\overline U\cap\mathcal S_0=\varnothing$,
and set $B=\partial_{\Sigma_1}U$. This is a nonempty compact
subset of the open slice, with $f=0$ on $B$. Since $n>0$
there, $\T$ does not vanish, and the first condition in
\eqref{eq:staticity-hypotheses}, extended by continuity, gives
$F=\T^\flat\wedge Y$ at each point of $B$. Consequently
\begin{equation}\label{eq:static-zero-level-relations}
 \begin{gathered}
 \mathrm df=2i_\T F=-2Y(\T)\T^\flat,\\
 \mathrm df|_{T\Sigma_1}=-2Y(\T)X^\flat,
 \qquad |X|_h=n>0\quad\text{on }B.
 \end{gathered}
\end{equation}
The second condition in \eqref{eq:staticity-hypotheses} makes $f$
a local coordinate near $B$. Its positive side belongs to $U$,
so $B$ is a smooth level surface, and
\eqref{eq:static-zero-level-relations} makes $X/n$ its unit normal.
We now compute the flux through this surface. Since
$\Div_hX=0$ and $X$ is tangent to $\mathcal S_0$,
\[
 \mathrm d(i_X\mathrm{vol}_h)=(\Div_hX)\mathrm{vol}_h=0,
 \qquad \int_{\mathcal S_0} i_X\mathrm{vol}_h=0.
\]
Because $\overline{\Sigma}_1\simeq[1,\infty)\times\mathbb S^2$,
the integral over $\mathcal S_0$ determines the cohomology class
of every closed two-form. The preceding identities therefore
make $i_X\mathrm{vol}_h$ exact, so its flux through each connected
component $B_0$ of $B$ must vanish. But with the unit normal $X/n$
from \eqref{eq:static-zero-level-relations}, this says
\[
 0=\int_{B_0}i_X\mathrm{vol}_h=\int_{B_0}n\dd A>0,
\]
a contradiction. Hence there is a point
$p\in\overline U\cap\mathcal S_0$.

We use the point $p$ to prove that $X$ vanishes on the whole
sphere. Its restriction is Killing, since $n=0$ and
$\Lie_Xh=2nk$ imply $\Lie_X(h|_{T\mathcal S_0})=0$, and
$p\in\overline U$ together with $f|_{\mathcal S_0}=-|X|^2$
gives $X(p)=0$. It remains to show that its tangential derivative
at $p$, a rotation with coefficient $\omega_p$, also vanishes.
The flow of $\T$ fixes $p$ and preserves $\mathcal S_0$, so its
derivative preserves both the tangent and normal planes.
Since $n=X=0$ at $p$,
\[
 \g(\D_{\nu_{\mathcal S_0}}\T,T_0)
       =-\partial_{\nu_{\mathcal S_0}}n.
\]
Preservation of the normal plane and the Killing equation give
\begin{equation}\label{eq:static-normal-boost}
 \D_{\nu_{\mathcal S_0}}\T
       =(\partial_{\nu_{\mathcal S_0}}n)T_0,
 \qquad
 \D_{T_0}\T
       =(\partial_{\nu_{\mathcal S_0}}n)\nu_{\mathcal S_0}.
\end{equation}
The normal derivatives in \eqref{eq:static-normal-boost} and the
tangential rotation give, in an adapted orthonormal coframe,
\begin{equation}\label{eq:static-boundary-wedge}
 \begin{aligned}
 F|_p&=(\partial_{\nu_{\mathcal S_0}}n)\,
                   \theta^0\wedge\theta^1
                  +\omega_p\,\theta^2\wedge\theta^3,\\
 (F\wedge F)|_p
   &=2\omega_p(\partial_{\nu_{\mathcal S_0}}n)\,
                   \theta^0\wedge\theta^1\wedge\theta^2\wedge\theta^3
     =0.
 \end{aligned}
\end{equation}
The vanishing of $F\wedge F$ on $\overline U$ from
\eqref{eq:static-two-form}, together with the positive normal
lapse derivative in \eqref{eq:lapse-boundary}, forces
$\omega_p=0$ in \eqref{eq:static-boundary-wedge}. Thus the
Killing field $X|_{\mathcal S_0}$ and its derivative vanish at
$p$. Restricting it to geodesics gives Jacobi fields with zero
initial data, so it vanishes on the connected sphere and
$\T|_{\mathcal S_0}=0$.

We integrate the Sudarsky--Wald identity
\eqref{eq:stationary-divergence}. On $S_R=\{r=R\}$, with outward
unit normal $\nu_R$, the decay $k=O(r^{-3})$, $X=O(r^{-2})$
from \eqref{eq:slice-asymptotic-data} makes the flux tend to zero.
Since $X=0$ on $\mathcal S_0$, we obtain
\begin{equation}\label{eq:staticity-integral}
 \begin{aligned}
 \int_{\Sigma_1}n|k|^2\dd\mu_h
   &=\lim_{R\to\infty}\int_{S_R}k(X,\nu_R)\dd A\\
   &\quad-\int_{\mathcal S_0}k(X,\nu_{\mathcal S_0})\dd A=0.
 \end{aligned}
\end{equation}
Since $n>0$, \eqref{eq:staticity-integral} gives $k=0$, and
\eqref{eq:KID} makes $X$ Killing on the slice. Its tangential
derivatives vanish on $\mathcal S_0$, and the Killing equation
then gives the normal derivative as well. The same uniqueness
argument along geodesics gives $X=0$ on the connected slice.
Thus $\T=nT_0$ and $f=n^2>0$, and the stationary translates of
the slice are orthogonal to $\T$.
\end{proof}

\subsection{The small angular momentum limit}

\begin{proof}[Proof of Theorem~\ref{thm:main}]
We prove \eqref{eq:horizon-smallness}. If this estimate has no
positive threshold $j_*(A)$ for a fixed $A$, there are smooth
spacetimes $(\g_j,\T_j,\Sigma_{1,j})$ satisfying the hypotheses with
$M_j=1$ and $\overline A_j\le A$, and angular momenta $S_j$ of
magnitude $J_j=|S_j|$, such that
\begin{equation}\label{eq:small-spin-contradiction}
 J_j\longrightarrow0,\qquad
 \sup_{\mathcal S_0}\g_j(\T_j,\T_j)
             \ge\eps_{\mathrm{AIK}}(A)^2>0.
\end{equation}
Pull back the fields by the charts of Section~\ref{subsec:assumptions}.
To retain a Lorentzian limit, we need a uniform bound for the
inverse metric. Suppress $j$ and let $h^{ij}$ denote the inverse of the spatial block
$\g_{ij}$ in these coordinates. The bounds for $h$ at
$t=0$ and for $\partial_t\g_{ij}$ in
Section~\ref{subsec:assumptions} keep $\g_{ij}$ uniformly positive
definite after decreasing the time interval uniformly in $j$. We restrict this strip to
$r>3/4$. The time-coordinate estimate \eqref{eq:coordinate-time-bound}
and $R_1\ge A_1^4$ give $-\g_{00}\ge c(A)>0$.
Putting $a^2=-\g_{00}+h^{ij}\g_{0i}\g_{0j}\ge c(A)$,
block inversion gives
\begin{equation}\label{eq:coordinate-inverse-bounds}
 \begin{aligned}
 \g^{00}&=-a^{-2},\qquad
 \g^{0i}=a^{-2}h^{ij}\g_{0j},\\
 \g^{ij}&=h^{ij}-a^{-2}h^{ik}\g_{0k}h^{j\ell}\g_{0\ell}.
 \end{aligned}
\end{equation}
The inverse metric is therefore uniformly bounded, and
repeated differentiation of $\g^{-1}\g=\operatorname{Id}$,
starting with $\partial\g^{-1}=-\g^{-1}(\partial\g)\g^{-1}$,
gives the same control through six derivatives by
\eqref{eq:effective-bound}. We can thus apply Arzel\`a--Ascoli and
pass to a diagonal subsequence for which
\begin{equation}\label{eq:stationary-compactness}
 \begin{gathered}
 \g_j\longrightarrow \g,\qquad \T_j\longrightarrow \T
        \quad\text{in }C^{5,\alpha}(K)\\
        \quad\text{for every compact }K,\qquad 0<\alpha<1.
 \end{gathered}
\end{equation}
The bound for $\g_j^{-1}$ keeps $\g$ Lorentzian, and
\eqref{eq:stationary-compactness} passes the vacuum and Killing
equations, tangency to $\mathcal S_0$, and \eqref{eq:KID} to
the limit. We use $\g,\T,h,k,n,X$ for the limiting fields
from now on. To identify their mass and angular momentum, we
retain the uniform remainders in \eqref{eq:AF}, which give,
for $|I|\le5$,
\[
 \begin{aligned}
 |\partial^I(\g_{00,j}+1-2/r)|&\le C_Ir^{-2-|I|},\\
 |\partial^I(\g_{0i,j}+2\epsilon_{iak}S_j^a x^kr^{-3})|
                                      &\le C_Ir^{-3-|I|}.
 \end{aligned}
\]
Passing to the limit at each fixed point of the end, using
\eqref{eq:stationary-compactness} and $J_j\to0$, yields
$\g_{00}=-1+2/r+O_5(r^{-2})$ and $\g_{0i}=O_5(r^{-3})$,
so the limiting fields have $M=1$ and $J=0$. The lapse bound
\eqref{eq:lapse-boundary} also passes to the limit, with
$0\le n\le1$ and $n=0$ on $\mathcal S_0$. Since the retained
asymptotics give $n\to1$ at infinity, the strong maximum
principle applied to the limiting equation in \eqref{eq:KID}
gives $0<n<1$ in the interior, and hence
$f=n^2-|X|^2<1$.

Let $U_j$ and $U$ be the components containing the end of
$\{f_j>0\}\cap\Sigma_{1,j}$ and $\{f>0\}\cap\Sigma_1$,
respectively. Given a smooth curve $c:[0,1]\to U$, choose a point
$x_*$ in the common timelike asymptotic end and join it to
$c(0)$ by a compact path in $U$. Since $f$ has a
positive minimum on the union of this path and $c$,
\eqref{eq:stationary-compactness} places the union in
$\{f_j>0\}$ for large $j$, and hence in $U_j$. In particular,
\eqref{eq:twist-smallness} gives $\psi_j\to0$ throughout $U$,
so $\sigma_j=f_j+i\psi_j\to f$ there.

To identify the limiting twist, set $F_j=\tfrac12\mathrm d\T_j^\flat$
and $F=\tfrac12\mathrm d\T^\flat$. By \eqref{eq:stationary-compactness},
\eqref{eq:vacuum-ernst}, and \eqref{eq:twist-smallness},
\begin{equation}\label{eq:limiting-twist-zero}
 \begin{aligned}
 \int_c2i_\T{}^*_{\g}F
 &=\lim_{j\to\infty}\int_c2i_{\T_j}{}^*_{\g_j}F_j\\
 &=\lim_{j\to\infty}\bigl[\psi_j(c(1))-\psi_j(c(0))\bigr]=0.
 \end{aligned}
\end{equation}
As $c$ is arbitrary, \eqref{eq:limiting-twist-zero} shows that
$2i_\T{}^*_{\g}F$ vanishes on $T\Sigma_1$ over $U$. It also
annihilates $\T$, which is timelike on $U$ and therefore
transverse to the spacelike slice. Thus $i_\T{}^*_{\g}F=0$
there, or equivalently $\T^\flat\wedge F=0$.

With $\cF=F+i{}^*_{\g}F$, we can now pass to the second identity
of \eqref{eq:Kerr-identities}, using
\eqref{eq:stationary-compactness} and $\sigma_j\to f$. On $U$
we obtain
\begin{equation}\label{eq:Schwarz-self-dual-limit}
 \cF^2=-\frac{(1-f)^4}{4},\qquad
 \cF^2=2\bigl(F_{\alpha\beta}F^{\alpha\beta}
               +iF_{\alpha\beta}({}^*F)^{\alpha\beta}\bigr).
\end{equation}
Taking real parts in \eqref{eq:Schwarz-self-dual-limit} and using continuity gives
$F_{\alpha\beta}F^{\alpha\beta}=-(1-f)^4/8$ on $\overline U$.
At an interior boundary point of $U$, we have $f=0$ and $n>0$,
so $\T$ is nonzero and null. The vanishing twist therefore gives
$F=\T^\flat\wedge Y$, and
\begin{equation}\label{eq:Schwarz-identities}
 \begin{gathered}
 F_{\alpha\beta}F^{\alpha\beta}=-2Y(\T)^2=-\tfrac18,
 \qquad \mathrm df=-2Y(\T)\T^\flat,\\
 |\mathrm df|_h^2=4Y(\T)^2|X|_h^2
       =-2F_{\alpha\beta}F^{\alpha\beta}n^2
       =\tfrac14 n^2>0.
 \end{gathered}
\end{equation}
Equation~\eqref{eq:Schwarz-identities} verifies the remaining
condition in \eqref{eq:staticity-hypotheses}, so
Lemma~\ref{lem:zero-spin-staticity}, with the lapse, asymptotics,
and boundary conditions already verified, gives $\T|_{\mathcal S_0}=0$.
Restricting \eqref{eq:stationary-compactness} to $\mathcal S_0$ gives
$\sup_{\mathcal S_0}\g_j(\T_j,\T_j)\to0$.
This contradicts \eqref{eq:small-spin-contradiction}, proving
\eqref{eq:horizon-smallness} for some $j_*(A)>0$. The Kerr
conclusion follows from
the global rigidity theorem of~\cite{AIKsmall},
with the same effective constant $A$.
\end{proof}

\appendix

\section{Verification of the Harrison construction}
\label{app:harrison}

We complete the proof of Proposition~\ref{prop:charging}, beginning
with the stationary primitive $W_0$ required by \eqref{eq:W}.
Its construction near $\mathcal S_0$ will preserve the local
horizon symmetries, which we will also use to obtain coordinates
for the transformed horizons uniformly in $q$. Between these
two constructions we verify the Einstein--Maxwell equations,
including at points where $f=0$ and the stationary reduction
is unavailable.

\subsection{The stationary primitive at the bifurcation sphere}

We verify positivity of the surface gravity in the normalization
\eqref{eq:local-circle}. At the chosen fixed point $p$ of the circle
action, let $\nu$ be the unit normal to $\mathcal S_0$ in
$\overline\Sigma_1$ pointing into $\Sigma_1$.
Then $\D_\nu\Z=0$, since a compact, time-oriented group of
boosts acts trivially on the Lorentzian normal plane.
Thus $\D_\nu\K=\kappa T_0$ and the boundary
derivative in \eqref{eq:lapse-boundary} give
\begin{equation}\label{eq:surface-gravity-lapse}
 \kappa=-\g(\D_\nu\T,T_0)|_p=\nu(n)|_p>0.
\end{equation}
To construct $W_0$ with the same symmetries, we first check
that $\psi$ is invariant. Since $\K,\Z$ preserve $\g$ and
commute with $\T$, they preserve
$\mathrm d\psi=2i_\T{}^*F$ in \eqref{eq:vacuum-ernst}, so
$\mathrm d(\Z\psi)=0$. The constant $\Z\psi$ must vanish by
periodicity, and $\T\psi=0$ then implies $\K\psi=0$.
We may therefore use circle averaging and radial integration
in the normal bundle without losing these symmetries.

Denote the closed two-form in \eqref{eq:W} by
\begin{equation}\label{eq:source}
 \Theta_0=4\bigl((f-1)F+\psi{}^*F\bigr).
\end{equation}
Contracting $F\wedge F+{}^*F\wedge{}^*F=0$ with $\T$ gives
$\mathrm df\wedge F=-\mathrm d\psi\wedge{}^*F$ and hence
\begin{equation}\label{eq:primitive-compatibility}
 \mathrm d\Theta_0=0,\qquad i_\T\Theta_0=\mathrm dD_0.
\end{equation}
The contraction identity is the calculation \eqref{eq:W-contraction}.
The spatial decay \eqref{eq:primitive-source-decay} and
closedness in \eqref{eq:primitive-compatibility} give
\begin{equation}\label{eq:primitive-periods}
 \int_{\mathcal S_0}\Theta_0=\int_{S_r}\Theta_0
                     =\lim_{R\to\infty}\int_{S_R}\Theta_0=0.
\end{equation}

The normal exponential map represents nearby points as
$\exp_p\xi$, where $p\in\mathcal S_0$ and $\xi$ is normal to
$\mathcal S_0$ at $p$. In these coordinates define the retraction
$\pi(p,\xi)=p$, the scaling $\delta_\lambda(p,\xi)=(p,\lambda\xi)$, and
$\mathcal R=\xi^\alpha\partial_{\xi^\alpha}$.
Equation~\eqref{eq:primitive-periods} gives a smooth one-form $\beta_0$ on
$\mathcal S_0$ with $\mathrm d\beta_0=\Theta_0|_{\mathcal S_0}$.
Averaging $\beta_0$ over the circle preserves this identity,
since $\Theta_0$ is invariant. We then extend $\beta_0$ to the
normal bundle by the relative Poincar\'e homotopy formula,
\begin{equation}\label{eq:Wbar}
 W_0=\pi^*\beta_0+
       \int_0^1\lambda^{-1}\delta_\lambda^*
                        (i_{\mathcal R}\Theta_0)\,\mathrm d\lambda .
\end{equation}
The integrand extends smoothly to $\lambda=0$, and Cartan's
formula gives
\begin{equation}\label{eq:primitive-homotopy-derivative}
 \begin{aligned}
 \mathrm dW_0
 &=\pi^*\mathrm d\beta_0+
       \int_0^1\partial_\lambda(\delta_\lambda^*\Theta_0)\,\mathrm d\lambda\\
 &=\pi^*(\Theta_0|_{\mathcal S_0})+
       \Theta_0-\pi^*(\Theta_0|_{\mathcal S_0})=\Theta_0.
 \end{aligned}
\end{equation}
The construction \eqref{eq:Wbar} also preserves both symmetries.
Normal scaling commutes with isometries preserving $\mathcal S_0$,
the circle preserves $\beta_0$, and the $\K$ flow fixes
$\mathcal S_0$ pointwise. Thus Cartan's identity and
\eqref{eq:primitive-compatibility} give
\begin{equation}\label{eq:primitive-symmetries}
 \begin{gathered}
 \Lie_\K W_0=\Lie_\Z W_0=0,\\
 \mathrm d\bigl(W_0(\T)+D_0\bigr)
    =\Lie_\T W_0-i_\T\mathrm dW_0+\mathrm dD_0=0.
 \end{gathered}
\end{equation}
Evaluating the constant $W_0(\T)+D_0$ at a zero of $\T$
proves \eqref{eq:W}, with the same constant at either pole.

In the exterior write
$W_0=(D_0(p)-D_0)\,\mathrm dt+(W_0)_{\mathrm{sp}}$.
The equation
\[
 \mathrm d(W_0)_{\mathrm{sp}}
                   =\Theta_0-\mathrm dt\wedge\mathrm dD_0
\]
has a stationary right-hand side whose contraction with $\T$
vanishes. It descends to $\Sigma_1$, where
\eqref{eq:primitive-periods} gives a primitive. We choose its decay explicitly.
The expansion \eqref{eq:ernst-end-expansion} gives the decay
of $f-1$ and $\psi$. In the stationary chart of \eqref{eq:AF},
$F_{ij}=\tfrac12(\partial_i\g_{0j}-\partial_j\g_{0i})$ and
$F_{0i}=-\tfrac12\partial_i\g_{00}=O_5(r^{-2})$.
The Hodge-star formula then gives
\[
 f-1=O_6(r^{-1}),\quad \psi=O_6(r^{-2}),\quad
 F_{ij}=O_5(r^{-3}),\quad
 ({}^*F)_{ij}=O_5(r^{-2}).
\]
Substitution in \eqref{eq:source} gives
\begin{equation}\label{eq:primitive-source-decay}
 (\Theta_0)_{ij}=O_5(r^{-4}).
\end{equation}
The radial primitive is
\begin{equation}\label{eq:primitive-radial-formula}
 (W_0)_i(x)=-\int_1^\infty
          s x^j(\Theta_0)_{ji}(sx)\,\mathrm ds .
\end{equation}
Differentiating \eqref{eq:primitive-radial-formula} and using
closedness of the restriction of $\Theta_0$ to the stationary
slice gives
\begin{equation}\label{eq:primitive-radial-curl}
 \begin{aligned}
 \partial_i(W_0)_j-\partial_j(W_0)_i
 &=-\int_1^\infty\frac{\mathrm d}{\mathrm ds}
                \bigl[s^2(\Theta_0)_{ij}(sx)\bigr]\,\mathrm ds\\
 &=(\Theta_0)_{ij}(x).
 \end{aligned}
\end{equation}
The boundary term at infinity vanishes by
\eqref{eq:primitive-source-decay}. Differentiating
\eqref{eq:primitive-radial-formula} and using the same decay gives
\begin{equation}\label{eq:primitive-radial-decay}
 |\partial^I(W_0)_i(x)|
 \le C_I r^{-3-|I|}\int_1^\infty s^{-3}\,\mathrm ds,
 \qquad |I|\le5.
\end{equation}
We now modify the global primitive to agree with the local
choices \eqref{eq:primitive-radial-formula} near infinity and
\eqref{eq:Wbar} near $\mathcal S_0$. Its difference from each
local primitive is closed by
\eqref{eq:primitive-radial-curl} and
\eqref{eq:primitive-homotopy-derivative}, and both the end and
the exterior collar have vanishing first cohomology. On the end, an exact
correction with a cutoff therefore gives agreement with
\eqref{eq:primitive-radial-formula} and the required decay
\eqref{eq:primitive-radial-decay}. On the collar, write the
needed correction as $\mathrm d\chi$, with $\T\chi=0$, and
add $\mathrm d(\eta\chi)$, where $\eta$ is stationary and
equals one near $\mathcal S_0$. For this last cutoff we use the
original stationary collar constructed at the end of the appendix.
At $q=1$, that construction depends only on $\g,\K,\Z$ and
the original slice because $W_1=0$, so it is available before
the choice of $W_0$.

\subsection{The Einstein--Maxwell equations}

Having recovered the norm and electromagnetic potentials in
\eqref{eq:charged-norm} and \eqref{eq:charged-potential},
we now recover the twist and verify the three reduced equations
where $f\ne0$. We then return to the spacetime formulas to
include points where $\T$ is null.

Where $f\ne0$, orient the orbit metric by
$\operatorname{vol}_\gamma=f\,i_\T\operatorname{vol}_\g$.
In both signatures of $\gamma$, $*_\gamma^2=1$ on one- and
two-forms. With the twist convention
\eqref{eq:vacuum-ernst}, the stationary identities
in~\cite{CCH} are
\begin{equation}\label{eq:stationary-orbit-form-identities}
 \mathrm d\vartheta=-f^{-2}*_\gamma\mathrm d\psi,\qquad
 {}^*F=-\frac1{2f}*_\gamma\mathrm df
                     +\tfrac12(\mathrm dt+\vartheta)\wedge\mathrm d\psi .
\end{equation}
Substituting \eqref{eq:stationary-orbit-form-identities} and
$F=\tfrac12(\mathrm dt+\vartheta)\wedge\mathrm df
       -\tfrac f2\mathrm d\vartheta$ into
\eqref{eq:charged-one-form}, we compute the charged connection
\begin{equation}\label{eq:connection-curl}
 \begin{aligned}
 \mathrm d\{D_q(\mathrm dt+\vartheta)+W_q\}
  &=\mathrm dD_q\wedge(\mathrm dt+\vartheta)
       +(\mathrm dt+\vartheta)\wedge\mathrm dD_q\\
  &\quad+\bigl[D_q-2(1-q)f\bigl((1-q)f-(1+q)\bigr)\bigr]
                                       \,\mathrm d\vartheta\\
  &\qquad-\frac{2(1-q)^2\psi}{f}*_\gamma\mathrm df\\
  &=\bigl((1+q)^2-(1-q)^2f^2+(1-q)^2\psi^2\bigr)
                                       \,\mathrm d\vartheta\\
  &\qquad-\frac{2(1-q)^2\psi}{f}*_\gamma\mathrm df .
 \end{aligned}
\end{equation}
The potentials in \eqref{eq:harrison} satisfy
\begin{equation}\label{eq:harrison-derivatives}
 \begin{gathered}
 \sigma_q'=\frac{4q}{(1+q-(1-q)\sigma)^2},\qquad
 \Phi_q'=-\frac{2\sqrt{1-q^2}}{(1+q-(1-q)\sigma)^2},\\
 \begin{aligned}
 \mathrm d\sigma_q+2\overline\Phi_q\,\mathrm d\Phi_q
 &=\frac{4\bigl[q(1+q-(1-q)\overline\sigma)
                   +(1-q^2)(1+\overline\sigma)\bigr]}
             {(1+q-(1-q)\sigma)D_q}\,\mathrm d\sigma\\
 &=\frac{4\bigl((1+q)^2-(1-q)^2\overline\sigma^{\,2}\bigr)}{D_q^2}
                                                  \,\mathrm d\sigma\\
 &=f_q\left(\frac1f+\frac{2(1-q)}{1+q-(1-q)\sigma}\right)
                                                  \mathrm d\sigma.
 \end{aligned}
 \end{gathered}
\end{equation}
Combining \eqref{eq:transformed-connection} with
\eqref{eq:connection-curl} and the Hodge identities
\eqref{eq:stationary-orbit-form-identities}, and then taking
the imaginary part of \eqref{eq:harrison-derivatives}, gives
\begin{equation}\label{eq:charged-twist}
 \begin{aligned}
 f_q^2*_\gamma\mathrm d\vartheta_q
 &=-\frac4{D_q^2}\Bigl[
   \bigl((1+q)^2-(1-q)^2f^2+(1-q)^2\psi^2\bigr)\,\mathrm d\psi\\
 &\hspace{7em}{}+2(1-q)^2f\psi\,\mathrm df\Bigr]\\
 &=-\mathrm d\operatorname{Im}\sigma_q
       -2\operatorname{Im}(\overline\Phi_q\,\mathrm d\Phi_q).
 \end{aligned}
\end{equation}
Equation~\eqref{eq:charged-twist} verifies the twist identity
\eqref{eq:em-potentials} for the transformed fields. At $q=0$
it also gives the formula for $\mathrm d\vartheta_0$ in
\eqref{eq:limiting-connection}.

The reduced Einstein--Maxwell equations in~\cite{CCH}, with
the potential convention \eqref{eq:em-potentials} verified by
\eqref{eq:charged-potential} and \eqref{eq:charged-twist}, are
\begin{equation}\label{eq:reduced-electrovacuum}
 \begin{aligned}
 f_q\Delta_\gamma\sigma_q
  &=\langle\mathrm d\sigma_q+2\overline\Phi_q\,\mathrm d\Phi_q,
                                      \mathrm d\sigma_q\rangle_\gamma,\\
 f_q\Delta_\gamma\Phi_q
  &=\langle\mathrm d\sigma_q+2\overline\Phi_q\,\mathrm d\Phi_q,
                                      \mathrm d\Phi_q\rangle_\gamma,\\
 \Ric(\gamma)
  &=\frac1{2f_q^2}\operatorname{Re}
    \big[(\mathrm d\sigma_q+2\overline\Phi_q\,\mathrm d\Phi_q)
       \otimes(\mathrm d\overline\sigma_q+2\Phi_q\,\mathrm d\overline\Phi_q)\big]\\
  &\qquad-\frac2{f_q}\operatorname{Re}
                         (\mathrm d\Phi_q\otimes\mathrm d\overline\Phi_q).
 \end{aligned}
\end{equation}
Here $\langle\cdot,\cdot\rangle_\gamma$ is complex bilinear,
and bars denote complex conjugation. We verify
\eqref{eq:reduced-electrovacuum} using the vacuum equations
$f\Delta_\gamma\sigma=\langle\mathrm d\sigma,\mathrm d\sigma\rangle_\gamma$
and
$\Ric(\gamma)=(2f^2)^{-1}\operatorname{Re}
(\mathrm d\sigma\otimes\mathrm d\overline\sigma)$.
Since the transformed potentials are functions of $\sigma$,
the scalar equations follow by the chain rule, while the Ricci
equation requires its right-hand side to reduce to the vacuum
expression with the same orbit metric.
For either $u=\sigma_q$ or $u=\Phi_q$, including the constant
endpoint cases, differentiation gives
$u''=2(1-q)u'/(1+q-(1-q)\sigma)$. Substituting the vacuum
equation and \eqref{eq:harrison-derivatives} gives
\begin{equation}\label{eq:harrison-scalar-verification}
 \begin{aligned}
 f_q\Delta_\gamma u
 &=f_q\bigl(u''\langle\mathrm d\sigma,\mathrm d\sigma\rangle_\gamma
                              +u'\Delta_\gamma\sigma\bigr)\\
 &=f_qu'\left(\frac1f+\frac{2(1-q)}{1+q-(1-q)\sigma}\right)
                   \langle\mathrm d\sigma,\mathrm d\sigma\rangle_\gamma\\
 &=\langle\mathrm d\sigma_q+2\overline\Phi_q\,\mathrm d\Phi_q,
                                      \mathrm du\rangle_\gamma.
 \end{aligned}
\end{equation}
In the Ricci equation of \eqref{eq:reduced-electrovacuum},
formula \eqref{eq:harrison-derivatives} reduces the coefficient
of $\operatorname{Re}(\mathrm d\sigma\otimes\mathrm d\overline\sigma)$ to
\begin{equation}\label{eq:harrison-Ricci-verification}
 \frac{|1+q+(1-q)\sigma|^2-4(1-q^2)f}{2f^2D_q}
 =\frac{|1+q-(1-q)\sigma|^2}{2f^2D_q}
 =\frac1{2f^2}.
\end{equation}
Equations~\eqref{eq:harrison-scalar-verification} and
\eqref{eq:harrison-Ricci-verification} prove
\eqref{eq:reduced-electrovacuum} where $f\ne0$. Their extension
to an open null region is a calculation at one point. We
approximate the values and first two derivatives of $\g,\psi,W_0$
by stationary data satisfying the vacuum and potential equations
at that point, with $f>0$. No local vacuum solution realizing
these data is needed. The spacetime formulas and their field
equations depend continuously on these derivatives when
$D_q>0$ and $D_q+W_q(\T)>0$, so they pass to the limit.
We first change $\g_{tt}$ to make $\T$ timelike, then correct
spatial second derivatives to restore the ten vacuum equations.

Since $\T\ne0$ in the exterior, choose local coordinates with
$\T=\partial_t$. All metric derivatives containing a $t$
derivative vanish. At fixed metric and first derivatives, the
Einstein tensor is affine in the second derivatives. Denote by
$P_\eta H$ its change under
$\delta(\partial_a\partial_b\g_{\mu\nu})
=\eta_a\eta_bH_{\mu\nu}$, where $H$ is symmetric and
$\eta(\T)=0$. Raising indices with $\g$, the change in the Einstein tensor is
\begin{equation}\label{eq:stationary-einstein-symbol}
 \begin{split}
 2(P_\eta H)_{\mu\nu}={}&
 \eta_\mu\eta^\rho H_{\rho\nu}+\eta_\nu\eta^\rho H_{\rho\mu}
 -\eta^2H_{\mu\nu}-\eta_\mu\eta_\nu\operatorname{tr}_{\g}H\\
 &-\g_{\mu\nu}\bigl(\eta^\rho\eta^\sigma H_{\rho\sigma}
                  -\eta^2\operatorname{tr}_{\g}H\bigr),
 \qquad \eta^2=\g^{-1}(\eta,\eta).
 \end{split}
\end{equation}
Contraction with $\eta^\mu$ gives
$(P_\eta H)(\eta^\sharp,\cdot)=0$. Conversely, if
$B(\eta^\sharp,\cdot)=0$ and $\eta^2\ne0$, set
\[
 H=-\frac2{\eta^2}
       \left(B-\tfrac12\operatorname{tr}_{\g}(B)\g\right).
\]
Its contractions and substitution in
\eqref{eq:stationary-einstein-symbol} give
\begin{equation}\label{eq:stationary-symbol-inverse}
 \begin{gathered}
 \operatorname{tr}_{\g}H=\frac{2\operatorname{tr}_{\g}B}{\eta^2},
 \qquad H(\eta^\sharp,\cdot)
                 =\frac{\operatorname{tr}_{\g}B}{\eta^2}\eta,
 \qquad H(\eta^\sharp,\eta^\sharp)=\operatorname{tr}_{\g}B,\\
 2P_\eta H
 =\bigl(2B-(\operatorname{tr}_{\g}B)\g\bigr)
       +(\operatorname{tr}_{\g}B)\g=2B.
 \end{gathered}
\end{equation}
The three contractions in \eqref{eq:stationary-symbol-inverse}
cancel the terms proportional to $\eta\otimes\eta$ in
\eqref{eq:stationary-einstein-symbol}. Together with
$(P_\eta H)(\eta^\sharp,\cdot)=0$, the inverse computation proves
\begin{equation}\label{eq:stationary-symbol-range}
 \operatorname{Im}P_\eta
   =\{B:B(\eta^\sharp,\cdot)=0\}.
\end{equation}
Choose independent covectors $\alpha,\beta$ in a positive definite
two-plane in the annihilator of $\T$. By \eqref{eq:stationary-symbol-range}, the annihilator of
$\operatorname{Im}P_\eta$ under the metric pairing consists of the
symmetric products $\eta\odot v=(\eta\otimes v+v\otimes\eta)/2$.
In the cotangent space at the chosen point $x$,
\begin{equation}\label{eq:stationary-symbol-annihilators}
 (\alpha\odot T_x^*\M)\cap(\beta\odot T_x^*\M)
       \cap((\alpha+\beta)\odot T_x^*\M)=\{0\}.
\end{equation}
Indeed the first two spaces intersect in the line spanned by
$\alpha\odot\beta$, which has zero intersection with the third.
Taking annihilators in \eqref{eq:stationary-symbol-annihilators}
and using \eqref{eq:stationary-symbol-range} shows that the
ranges of $P_\alpha$, $P_\beta$, and $P_{\alpha+\beta}$ span all
symmetric two-tensors.
We use this surjectivity to correct the perturbed metric data.
Choose ten pairs $(\eta_a,H_a)$, with
$\eta_a\in\{\alpha,\beta,\alpha+\beta\}$, whose images
$P_{\eta_a}H_a$ form a basis of the symmetric two-tensors.
The corresponding matrix remains invertible, with bounded
inverse, for nearby metric values. If we change $\g_{tt}$ to
$-\epsilon$ and initially keep its first and second derivatives
fixed, the Einstein tensor becomes $E_\epsilon=O(\epsilon)$,
since the original data are vacuum. Computing $P_{\eta_a}$
at the perturbed metric value, we can therefore solve for
$c_\epsilon\in\mathbb R^{10}$ so that
\begin{equation}\label{eq:stationary-vacuum-correction}
 \begin{gathered}
 \sum_{a=1}^{10}c_\epsilon^aP_{\eta_a}H_a=-E_\epsilon,
       \qquad |c_\epsilon|\le C|E_\epsilon|=O(\epsilon),\\
 \delta(\partial_i\partial_j\g_{\mu\nu})
   =\sum_{a=1}^{10}c_\epsilon^a(\eta_a)_i(\eta_a)_j(H_a)_{\mu\nu}
   =O(\epsilon).
 \end{gathered}
\end{equation}
Because the Einstein tensor is affine in the second derivatives,
\eqref{eq:stationary-vacuum-correction} restores the vacuum
equations exactly. The correction introduces no time derivatives,
as $\eta_a(\T)=0$, and tends to zero with $\epsilon$.
We have thus obtained stationary vacuum data converging through
order two to the original data, with $f_\epsilon=\epsilon>0$.

We must make the same correction for $\psi$ and $W_0$, since
the transformed equations depend on their derivatives as well.
Their compatibility follows from the vacuum equations restored
in \eqref{eq:stationary-vacuum-correction}. Indeed,
\[
 \mathrm d(2i_\T{}^*F_\epsilon)
   =2\Lie_\T{}^*F_\epsilon-2i_\T\mathrm d{}^*F_\epsilon=0
\]
at the point. Keeping the value of $\psi$ fixed, the one-form
$2i_\T{}^*F_\epsilon$ and its first derivative therefore prescribe
the gradient and symmetric Hessian of $\psi_\epsilon$. They depend
continuously on $\epsilon$ and have vanishing time derivatives.
Define $\Theta_{0,\epsilon}$ by \eqref{eq:source} using the
corrected fields. The calculation \eqref{eq:primitive-compatibility}
gives
\begin{equation}\label{eq:primitive-jet-compatibility}
 \mathrm d\Theta_{0,\epsilon}=0,\qquad
 i_\T\Theta_{0,\epsilon}=\mathrm dD_{0,\epsilon},
\end{equation}
with $i_\T\Theta_{0,\epsilon}=\mathrm dD_{0,\epsilon}$ also valid
through one derivative. Keep $D_0(p)$ fixed and
set $(W_{0,\epsilon})_t=D_0(p)-D_{0,\epsilon}$ through order two.
For the spatial part, center the coordinates at the point and write
$\delta\Theta=\Theta_{0,\epsilon}-\Theta_0$.
By \eqref{eq:primitive-jet-compatibility}, the linear Taylor
polynomial of its spatial part is closed. Applying the homotopy
formula \eqref{eq:primitive-homotopy-derivative} with retraction
to the origin gives the correction
\begin{equation}\label{eq:primitive-jet-correction}
 \begin{aligned}
 (\delta W)_i(x)
  &=\int_0^1t x^j\bigl(\delta\Theta_{ji}(0)
              +t x^k\partial_k\delta\Theta_{ji}(0)\bigr)\,\mathrm dt\\
  &=\tfrac12x^j\delta\Theta_{ji}(0)
       +\tfrac13x^jx^k\partial_k\delta\Theta_{ji}(0).
 \end{aligned}
\end{equation}
Thus $\mathrm d(\delta W)$ agrees with the spatial part of
$\delta\Theta$ through order one. The time-space components
agree by the contraction identity in
\eqref{eq:primitive-jet-compatibility}, valid through one derivative. Thus $W_{0,\epsilon}\to W_0$ through order two, with
$W_{0,\epsilon}(\T)=D_0(p)-D_{0,\epsilon}$.

Apply \eqref{eq:harrison-scalar-verification},
\eqref{eq:harrison-Ricci-verification},
\eqref{eq:charged-potential}, and \eqref{eq:charged-twist}
to the corrected data. The convergence through order two in
\eqref{eq:stationary-vacuum-correction} and
\eqref{eq:primitive-jet-correction} then gives the
Einstein--Maxwell equations at the original null point, since
\eqref{eq:metric-expanded} and \eqref{eq:maxwell} have no
division by $f$. This proves the equations wherever $\T\ne0$.
They extend to its zeros by smoothness of $\g_q,\F_q$ and
density of $\{\T\ne0\}$ for a nontrivial Killing field.

Changing the primitive by an exact stationary one-form amounts
to changing the time coordinate. For $\T\chi=0$,
the stationary time change
$\Psi(t,x)=(t+\chi(x)/D_q(p),x)$ satisfies
\[
 \mathrm d\Psi(v)+D_q^{-1}W_q(\mathrm d\Psi(v))\T
       =v+D_q^{-1}(W_q+\mathrm d\chi)(v)\T .
\]
Substitution in \eqref{eq:metric-expanded} and \eqref{eq:maxwell}
therefore shows that adding an exact stationary form to $W_q$
pulls back both fields by this coordinate change.
\subsection{The horizon collar}
\label{subsec:collar}

\begin{lemma}\label{lem:harrison-collar}
For each $q\in[0,1]$, there is an outward unit spacelike normal
field along $\mathcal S_0$ whose geodesics form a spacelike
exterior collar, parametrized by length $\rho$ and base point $y$.
This collar meets each $\T$-orbit through
it once. Its $\K$-translates have coordinates $(t,\rho,y)$ with
$0<\rho<\rho_0$, in which
\begin{equation}\label{eq:ADM-collar}
 \begin{gathered}
 \g_q=\mathrm d\rho^2+
 (h_\rho)_{AB}(\mathrm dy^A+\beta^A\mathrm dt)
                    (\mathrm dy^B+\beta^B\mathrm dt)-a^2\mathrm dt^2,\\
 \partial_t=\K,\qquad -\infty<t<\infty .
 \end{gathered}
\end{equation}
The slices $t=\mathrm{const}$ are spacelike, and
\begin{equation}\label{eq:collar-orders}
 \begin{gathered}
 a=\kappa\rho+O(\rho^3),\qquad
 \beta=O(\rho^2),\qquad \partial_\rho\beta=O(\rho),\\
 h_\rho=h_0+O(\rho^2),\qquad
 \partial_\rho h_\rho=O(\rho).
 \end{gathered}
\end{equation}
Here $\T(\rho)=0$ and, at $q=1$, the restriction of $\rho$ to
$\overline\Sigma_1$ is a smooth defining function of $\mathcal S_0$.
Regular normal coordinates $(x^0,x^1,y)$ extend across the sphere.
On the exterior part of their neighbourhood they satisfy
\begin{equation}\label{eq:collar-boost-coordinates}
 x^0=\rho\sinh(\kappa t),\qquad
 x^1=\rho\cosh(\kappa t).
\end{equation}
Here $\partial_{x^0},\partial_{x^1}$ form an orthonormal normal
frame on $\mathcal S_0$, with $\partial_{x^0}$ future timelike.
The regular coordinates and the coefficients in
\eqref{eq:ADM-collar} depend smoothly on $q$, with differentiated
remainder bounds uniform for $q\in[0,1]$.
\end{lemma}

We first show that the same bifurcate horizon survives the
transformation. The invariance of $W_0$ in \eqref{eq:primitive-symmetries},
together with that of $\g,\T,\psi$, gives from
\eqref{eq:metric-expanded} and \eqref{eq:maxwell},
\begin{equation}\label{eq:charged-horizon-symmetries}
 \Lie_\K\g_q=\Lie_\Z\g_q=0,\qquad
 \Lie_\K\F_q=\Lie_\Z\F_q=0.
\end{equation}
On either horizon branch, invariance along the $\K$ flow makes
$W_0(\K)$ equal to its limit at $\mathcal S_0$. Since $\K=0$ there,
this limit vanishes. Also $\g(\K,\T)=0$, since $\K$ is normal and
$\T$ tangent to the horizon. Thus $W_q(\K)=0$. The metric formula gives
\begin{equation}\label{eq:charged-horizon-norm}
 \begin{aligned}
 \g_q(\K,\K)
 &=\frac{D_q}{4}\g(\K,\K)
       +\frac12\g(\T,\K)W_q(\K)
       -\frac f{4D_q}W_q(\K)^2,\\
 \g_q(\K,\cdot)&=\frac{D_q}{4}\g(\K,\cdot)
                                  \quad\text{on }\mathcal H^+.
 \end{aligned}
\end{equation}
Since $\g(\K,\T)=W_q(\K)=0$ on the horizon, differentiating
\eqref{eq:charged-horizon-norm} there gives
\begin{equation}\label{eq:charged-surface-gravity}
 \mathrm d\bigl(\g_q(\K,\K)\bigr)
  =\frac{D_q}{4}\mathrm d\bigl(\g(\K,\K)\bigr)
  =-2\kappa\,\g_q(\K,\cdot).
\end{equation}
Applying the same calculation on the past branch, we obtain
from \eqref{eq:charged-horizon-symmetries}--\eqref{eq:charged-surface-gravity}
two nondegenerate Killing horizons for $\K$, with the original
surface gravity. To use their vanishing null second fundamental
forms at $\mathcal S_0$, observe that the invertible linear change
$v\mapsto v+D_q^{-1}W_q(v)\T$ preserves $T\mathcal S_0$
because $\T$ is tangent to the sphere.
Thus the induced metric remains positive definite. Writing
$\D^{(q)}$ for the connection of $\g_q$ and $L_\pm$ for the
two nonzero null normals, we obtain, for $Y,V$ tangent to
$\mathcal S_0$,
\begin{equation}\label{eq:charged-sphere-second-fundamental}
 \g_q(\D^{(q)}_YV,L_\pm)
   =-\g_q(V,\D^{(q)}_YL_\pm)=-\chi^\pm(Y,V)=0.
\end{equation}
Since $L_+,L_-$ span the normal plane,
\eqref{eq:charged-sphere-second-fundamental} proves that
$\D^{(q)}_YV$ is tangent to $\mathcal S_0$, and hence that
$\mathcal S_0$ is totally geodesic for $\g_q$. The circle action
used in the construction is unchanged and retains period $2\pi$.

The orders in \eqref{eq:collar-orders} follow from the boost
action of $\K$ on the normal plane and total geodesy of
$\mathcal S_0$. We compute the first terms using Jacobi fields.
Choose a smooth circle-invariant $\g$-orthonormal normal frame
$(e_0,e_1)$ along $\mathcal S_0$, with $e_0$ future and $e_1$
outward. Such a frame follows by averaging a future timelike
normal over the circle and normalizing. By \eqref{eq:Wbar},
$W_0$ is tangential at $\mathcal S_0$. Thus, for an original normal
vector $N$,
\begin{equation}\label{eq:charged-normal-frame}
 \begin{gathered}
 W_q(N)=\T^\flat(N)=0,\qquad
 \g_q(N,\cdot)=\frac{D_q}{4}\g(N,\cdot),\\
 e_{a,q}=2D_q^{-1/2}e_a\quad(a=0,1).
 \end{gathered}
\end{equation}
Use $\exp_y^{\g_q}(\rho e_{1,q})$ for the zero-time slice and
then flow by $\K$. Write $J_A$ for the fields obtained by varying
the base point.
For $v=\partial_\rho$, the radial geodesic equation and
$[v,J_A]=0$, with initial orthogonality
\eqref{eq:charged-normal-frame}, give
\begin{equation}\label{eq:collar-radial-orthogonality}
 \begin{gathered}
 \partial_\rho\g_q(v,\K)=\g_q(v,\D^{(q)}_v\K)=0,\qquad
 \partial_\rho\g_q(v,J_A)=\tfrac12J_A\g_q(v,v)=0,\\
 \g_{q\,\rho t}=\g_{q\,\rho A}=0,\qquad \g_{q\,\rho\rho}=1.
 \end{gathered}
\end{equation}
Set $\zeta_A=\g_q(e_{0,q},\D^{(q)}_{e_A}e_{1,q})$ on $\mathcal S_0$.
Total geodesy in \eqref{eq:charged-sphere-second-fundamental}
makes $\D^{(q)}_{e_A}e_{1,q}$ normal to $\mathcal S_0$.
The derivative of $\K$ there is independent of the connection,
since $\K=0$, and its normal boost has the coefficient
\eqref{eq:surface-gravity-lapse}. Thus the Jacobi data are
\begin{equation}\label{eq:collar-jacobi-data}
 \begin{gathered}
 \K(0)=0,\qquad \D^{(q)}_\rho\K(0)=\kappa e_{0,q},\\
 J_A(0)=e_A,\qquad
 \D^{(q)}_\rho J_A(0)=-\zeta_Ae_{0,q}.
 \end{gathered}
\end{equation}
Writing $\R_q$ for the curvature of $\g_q$, the curvature
commutator, $[v,J]=0$, and $\D^{(q)}_vv=0$ give, for any $Y$,
\begin{equation}\label{eq:collar-jacobi-equation}
 \begin{gathered}
 \g_q(\D^{(q)}_v\D^{(q)}_vJ,Y)=\R_q(Y,v,v,J),
 \qquad J=\K,J_A,\\
 \g_q(\D^{(q)}_v\D^{(q)}_v\K(0),Y)
       =\R_q(Y,e_{1,q},e_{1,q},\K(0))=0.
 \end{gathered}
\end{equation}
With the initial data \eqref{eq:collar-jacobi-data}, Taylor
expansion in a parallel frame therefore yields
\begin{equation}\label{eq:collar-jacobi-expansions}
 \K=\kappa\rho e_{0,q}+O(\rho^3),\qquad
 J_A=e_A-\rho\zeta_Ae_{0,q}+O(\rho^2).
\end{equation}
Taking inner products in \eqref{eq:collar-jacobi-expansions}
and completing the square yields
\begin{equation}\label{eq:collar-metric-expansions}
 \begin{aligned}
 \g_q(\K,\K)&=-\kappa^2\rho^2+O(\rho^4),&
 \g_q(\K,J_A)&=\kappa\zeta_A\rho^2+O(\rho^3),\\
 (h_\rho)_{AB}&=(h_0)_{AB}+O(\rho^2),&
 \beta^A&=(h_\rho)^{AB}\g_q(\K,J_B),\\
 a^2&=-\g_q(\K,\K)+|\beta|_{h_\rho}^2.
 \end{aligned}
\end{equation}
This proves \eqref{eq:collar-orders}. To obtain the asserted
uniformity, use the smooth dependence of $\g_q$ and the initial
frame \eqref{eq:charged-normal-frame} on $(q,y)$ in the
geodesic equation. Compactness of $[0,1]\times\mathcal S_0$
gives a common normal neighbourhood, on which differentiation
of \eqref{eq:collar-jacobi-equation} bounds every fixed finite
number of $(q,\rho,y)$ derivatives of each remainder divided
by its stated power of $\rho$.

Equivariance of the normal exponential map under the $\K$ flow
gives \eqref{eq:collar-boost-coordinates}, with
$\partial_{x^a}=e_{a,q}$ on $\mathcal S_0$ for $a=0,1$.
Circle invariance of the initial frame gives $\Z(\rho)=0$,
so $\T(\rho)=\K(\rho)-\Omega_H\Z(\rho)=0$.
To extend the coordinates to all Killing times, we prove that
the zero-time collar is a section of the stationary flow,
first for $q=1$.
The original slice is $x^0=x^1u(x^1,y)$, with $|u(0,y)|<1$.
Thus $\rho=x^1\sqrt{1-u^2}$ is a smooth defining function
there, and \eqref{eq:collar-boost-coordinates} writes the slice
as $t=v_0(\rho,y)$, smooth up to $\rho=0$.
The conormal of the original spacelike slice gives
\begin{equation}\label{eq:collar-orbit-transversality}
 1+\Omega_H\Z v_0
 =\mathrm d(t-v_0)(\T)
 =n\sqrt{-\g^{-1}(\mathrm d(t-v_0),\mathrm d(t-v_0))}>0.
\end{equation}
The coefficient bounds \eqref{eq:collar-orders} give
\begin{equation}\label{eq:collar-graph-conormal}
 -\g^{-1}(\mathrm d(t-v_0),\mathrm d(t-v_0))
 =\frac{(1+\beta(v_0))^2}{a^2}-|\nabla v_0|^2
 =\frac1{\kappa^2\rho^2}+O(1).
\end{equation}
Substituting \eqref{eq:collar-graph-conormal} into
\eqref{eq:collar-orbit-transversality} gives, on the original slice,
\[
 (1+\Omega_H\Z v_0)|_{\rho=0}
   =\kappa^{-1}(\partial_\rho n)|_{\rho=0}>0
\]
by \eqref{eq:lapse-boundary}.
The projection of this part of $\Sigma_1$ onto the zero-time
slice is
\[
 (\rho,y)\longmapsto
    (\rho,\Psi_{\Omega_Hv_0(\rho,y),\Z}y).
\]
Its angular differential satisfies
\begin{equation}\label{eq:collar-projection-differential}
 \begin{gathered}
 \mathrm d_y\{\Psi_{\Omega_Hv_0(\rho,y),\Z}y\}
 =\mathrm d\Psi_{\Omega_Hv_0,\Z}\circ
       (\operatorname{Id}+\Omega_H\Z\otimes\mathrm d_yv_0),\\
 \det(\operatorname{Id}+\Omega_H\Z\otimes\mathrm d_yv_0)
       =1+\Omega_H\Z v_0>0.
 \end{gathered}
\end{equation}
For fixed $\rho\ge0$ this is a local diffeomorphism of the sphere
of degree one, since replacing $v_0$ by $sv_0$, $0\le s\le1$,
gives a homotopy to the identity. It is therefore a diffeomorphism.
Since $v_0$ is bounded on a compact collar, we can shorten the
collar so that its stationary segments to the zero-time slice
lie in one circle-invariant regular normal neighbourhood. Each
complete stationary orbit meets $\Sigma_1$ once, so the same
is true of the zero-time slice, and the map
\[
 \mathbb R\times\{t=0,\ 0<\rho<\rho_0\}\longrightarrow\DOC,
 \qquad (\tau,p)\longmapsto\Psi_{\tau,\T}(p)
\]
is injective and supplies the required stationary product.
We extend $\Z$ by stationary translation on this product,
preserving its Killing equation, its commutation with $\T$,
and its period. The commuting complete flows of $\T$ and
$\Z$ then make $\K=\T+\Omega_H\Z$ complete there as well.

For the charged metrics the Gaussian zero-time slice may be
different. In the original regular normal chart, the initial
normal vectors in \eqref{eq:charged-normal-frame} give
\begin{equation}\label{eq:charged-geodesic-expansions}
 x^0=O(\rho_q^2),\qquad
 x^1=2D_q^{-1/2}\rho_q+O(\rho_q^2),\qquad
 Y=y+O(\rho_q^2).
\end{equation}
Substituting \eqref{eq:charged-geodesic-expansions} into
\eqref{eq:collar-boost-coordinates} gives
\begin{equation}\label{eq:charged-boost-expansions}
 \begin{gathered}
 r=\sqrt{(x^1)^2-(x^0)^2}
       =2D_q^{-1/2}\rho_q+O(\rho_q^2),\\
 t=\kappa^{-1}\operatorname{arctanh}(x^0/x^1)=O(\rho_q).
 \end{gathered}
\end{equation}
Smooth dependence of the geodesic equation on $(q,y)$ makes
the remainders in \eqref{eq:charged-geodesic-expansions} and
\eqref{eq:charged-boost-expansions}, divided by their stated
powers of $\rho_q$, smooth with bounded angular derivatives
on the common collar, uniformly for $q\in[0,1]$.
The positive radial derivative in \eqref{eq:charged-boost-expansions}
and the angular expansion in \eqref{eq:charged-geodesic-expansions}
allow us to invert $(\rho_q,y)\mapsto(r,Y)$. Thus the charged
zero-time slice is $t=b_q(r,Y)$, with $b_q=O(r)$ smoothly. Its stationary
projection is
\begin{equation}\label{eq:charged-stationary-projection}
 (r,Y)\longmapsto(r,\Psi_{\Omega_Hb_q(r,Y),\Z}Y),\qquad
 1+\Omega_H\Z b_q=1+O(r)>0.
\end{equation}
Applying \eqref{eq:collar-projection-differential} with $b_q$
in place of $v_0$ proves that this projection is a diffeomorphism.
The uniform bound $1+O(r)$ gives a common smaller collar on
which every charged zero-time slice is a section of the
stationary product. We may therefore flow each slice
for all times by
$\Psi_{t,\K}=\Psi_{t,\T}\circ\Psi_{\Omega_Ht,\Z}$, obtaining
\eqref{eq:ADM-collar} for $t\in\mathbb R$ and completing the
proof of Proposition~\ref{prop:charging}.

\section{Compactness of the maximal hypersurfaces}
\label{app:maximal-compactness}

We use Bartnik's gradient estimate~\cite{Bartnik} relative to a
smooth function $\tau$ with past timelike gradient, called a
temporal function. The following lemmas give the normal estimate
and a fixed temporal function. We then apply them to
Lemma~\ref{lem:maximal-compactness} before proving them.

\begin{lemma}\label{lem:temporal-gradient}
Let $\g_j$ be smooth Lorentz metrics converging in $C^2$ to a
Lorentz metric near a fixed compact set $Q$.
Let $\tau$ be a fixed smooth function with
past timelike gradient and
\[
 \alpha_j^2=-\g_j^{-1}(\mathrm d\tau,\mathrm d\tau)\ge a_0^2>0
       \quad\hbox{on }Q\cap\{\tau\ge0\}.
\]
Suppose $\Sigma_j$ are smooth spacelike maximal hypersurfaces
with future unit normals $T_{0,j}$, whose portions $\Sigma_j\cap\{\tau\ge0\}$ are compact and
contained in $Q$. Put
\[
 V_j=-\alpha_j^{-1}\operatorname{grad}_{\g_j}\tau,
 \qquad v_j=-\g_j(T_{0,j},V_j),
 \qquad \tau_{\max}=\max_Q\max(\tau,0).
\]
There is an exponent $\mu>4$, depending only on $a_0$,
$\sup_j(\|\g_j\|_{C^2(Q)}+\|\g_j^{-1}\|_{C^0(Q)})$, and
$\|\tau\|_{C^3(Q)}$ in a fixed finite coordinate cover, such that
\begin{equation}\label{eq:temporal-gradient}
 \sup_{\Sigma_j\cap\{\tau>0\}}\tau^\mu v_j
 \le2\max\left\{2\tau_{\max}^\mu,
       \sup_{\partial\Sigma_j\cap\{\tau>0\}}\tau^\mu v_j\right\}.
\end{equation}
The boundary supremum is zero when its set is empty.
\end{lemma}

\begin{lemma}\label{prop:compact-causal}
Let $\g$ be a $C^2$ time-oriented Lorentz metric on
$\mathcal B$. Suppose that every compact $C\subset\mathcal B$
admits a constant $L_C$ such that
\begin{equation}\label{eq:compact-proper-time}
 L_\g(c)\le L_C
\end{equation}
for every Lipschitz future curve $c:[a,b]\to C$ with
$\g(\dot c,\dot c)\le-\epsilon_c<0$ almost everywhere.
Then every compact $Q\Subset\mathcal B$ has a neighbourhood
admitting a smooth function with past timelike gradient.
\end{lemma}

\begin{proof}[Proof of Lemma~\ref{lem:maximal-compactness}]
Write $\g_j=\g_{q(s_j)}$ and $n_j=-\g_j(\T,T_{0,j})$.
The estimate \eqref{eq:outer-collar-normal} bounds the normals
on a fixed outer collar.

At a fixed small $\rho_0>0$, smooth parameter dependence in
Lemma~\ref{lem:harrison-collar} and
\eqref{eq:charged-boost-expansions} place the zero-time collar
spheres and their $\Z$-orbits in a common compact subset of
$\DOC$, with stationary projections in one compact interior
annulus. Consequently \eqref{eq:interior-height} bounds $G_j$
on this set, and \eqref{eq:corotating-graph} bounds the corotating
heights $v_j(\rho_0,\cdot)$ on both sides. Applying
\eqref{eq:collar-height} to each fixed hypersurface gives this
uniform bound throughout the collar and hence \eqref{eq:wedge}.

In the regular normal coordinates of
Lemma~\ref{lem:harrison-collar}, the coefficients converge smoothly
and $-\g_j(\partial_{x^0},\partial_{x^0})\ge1/2$ on a common
smaller neighbourhood. The corotating graphs $t=v_j(\rho,y)$
satisfy $|x^0|\le\theta x^1$ by \eqref{eq:wedge}, while
\eqref{eq:collar-boost-coordinates} gives $x^1\ge\rho$.
Along each angular ray, the derivative of $x^1$ cannot vanish,
since the tangent would then be
a nonzero multiple of the timelike field $\partial_{x^0}$,
and it is positive because $x^1\to0$ as $\rho\downarrow0$.
Thus the hypersurfaces are graphs $x^0=w_j(z)$, $z=(x^1,y)$,
on a common half-cylinder. The inequality
$|w_j|\le\theta x^1$ gives at $x^1=0$
\begin{equation}\label{eq:horizon-boundary-normal}
 \begin{gathered}
 w_j=0,\qquad \partial_yw_j=0,\qquad
 |\partial_{x^1}w_j|\le\theta<1,\\
 -\g_j(T_{0,j},\partial_{x^0})
       \le(1-\theta^2)^{-1/2}.
 \end{gathered}
\end{equation}

To extend the normal bound from $\mathcal S_0$ into the collar, choose
$\theta<c<1$ and angular coordinates orthonormal at a fixed
point of $\mathcal S_0$. On a sufficiently small coordinate
half-ball $|z|<r_0$ put
\[
 \delta=\frac{(c-\theta)r_0}{4},\qquad
 \epsilon=\frac\delta{4c},\qquad
 \tau=\delta+x^0-c\sqrt{\epsilon^2+|z|^2}.
\]
For the constant Minkowski form,
$(\partial_{x^0}\tau)^2-|\partial_z\tau|^2\ge1-c^2$.
Shrinking the chart and using smooth metric convergence, together
with $|w_j|\le\theta|z|$ on the hypersurfaces, gives
\begin{equation}\label{eq:horizon-temporal-cutoff}
 \begin{gathered}
 -\g_j^{-1}(\mathrm d\tau,\mathrm d\tau)\ge(1-c^2)/2,\\
 \tau\le-\tfrac34(c-\theta)r_0\quad\hbox{when }|z|=r_0,\\
 \tau\ge\delta/2\quad\hbox{when }
       |z|\le\delta/[4(c+\theta)].
 \end{gathered}
\end{equation}
The parameters $r_0,\epsilon$ fix the $C^3$ bound for $\tau$.
With $V_j,\alpha_j$ as in Lemma~\ref{lem:temporal-gradient},
$\partial_{x^0}\tau=1$ gives
\[
 -\g_j\left(V_j,
       \frac{\partial_{x^0}}{\sqrt{-\g_{j,00}}}\right)
       =\frac1{\alpha_j\sqrt{-\g_{j,00}}}\le C.
\]
Here \eqref{eq:horizon-temporal-cutoff} bounds $\alpha_j$ below,
and $-\g_{j,00}\ge1/2$. The same cutoff excludes $|z|=r_0$
from $\{\tau>0\}$, while this angle bound and
\eqref{eq:horizon-boundary-normal} control the boundary term
in Lemma~\ref{lem:temporal-gradient}. The lemma bounds the
normals where $\tau\ge\delta/2$, and a finite cover of
$\mathcal S_0$ gives a fixed collar with bounded normals.

A normal bound gives ellipticity of the graph equation in
both boundary and interior charts. For a local graph $x^0=w(z)$ with timelike
vertical field, set
\[
 e_a=\partial_{z^a}+p_a\partial_{x^0},\qquad
 h_{ab}=\g_j(e_a,e_b),\qquad
 L_j(z,w,p)=\sqrt{\det h}.
\]
At $(w,p)=(w_j,\mathrm dw_j)$, maximality and differentiation
of the volume density give
\begin{equation}\label{eq:maximal-graph-ellipticity}
 \begin{gathered}
 -\partial_aL_{j,p_a}+L_{j,w}=0,\\
 L_{j,p_a}=L_jh^{ab}\g_j(\partial_{x^0},e_b),\\
 \begin{aligned}
 -L_{j,p_ap_b}
 &=L_jh^{ab}\bigl[h^{cd}\g_j(\partial_{x^0},e_c)
                       \g_j(\partial_{x^0},e_d)-\g_{j,00}\bigr]\\
 &=L_jh^{ab}[\g_j(\partial_{x^0},T_{0,j})]^2,
 \end{aligned}\\
 \det h=-\det\g_j/[\g_j(\partial_{x^0},T_{0,j})]^2.
 \end{gathered}
\end{equation}
To obtain uniform ellipticity from
\eqref{eq:maximal-graph-ellipticity}, use the metric bounds and
$\g_{j,00}\le-c<0$ to obtain
\[
 \begin{gathered}
 0<h_{aa}=\g_{j,aa}+2p_a\g_{j,0a}+p_a^2\g_{j,00}
       \le C(1+|p_a|)-cp_a^2,\\
 |p|\le C,\qquad h\le C\operatorname{Id}.
 \end{gathered}
\]
The normal bound, reverse Cauchy--Schwarz inequality, and
nondegeneracy of the ambient metric then give
\[
 \begin{gathered}
 c\le-\g_{j,00}\le[\g_j(\partial_{x^0},T_{0,j})]^2\le C,\\
 c\le\det h=-\frac{\det\g_j}
                  {[\g_j(\partial_{x^0},T_{0,j})]^2}\le C,
 \qquad \lambda_{\min}(h)\ge C^{-2}\det h\ge c.
 \end{gathered}
\]
Substitution in the Hessian formula
\eqref{eq:maximal-graph-ellipticity} yields
\begin{equation}\label{eq:graph-elliptic-bounds}
 \lambda|\xi|^2\le-L_{j,p_ap_b}\xi_a\xi_b\le\Lambda|\xi|^2
\end{equation}
with constants independent of $j$. Smooth convergence of
$\g_j$ gives bounds for every fixed number of derivatives of
$L_j$ on these compact sets of $(z,w,p)$. At the horizon the
Dirichlet data are $w_j=0$ by \eqref{eq:horizon-boundary-normal}.
The bound \eqref{eq:graph-elliptic-bounds}, boundary
$C^{1,\alpha}$ estimates, and Schauder estimates therefore give
uniform $C^{m,\alpha}$ bounds for each $m\ge2$ on smaller
half-balls. At the outer boundary, smooth convergence of
$C_{s_j}$ gives bounded Dirichlet data and bounded boundary
straightening maps, so the same estimates apply there.

We next bound the interior normals. Identify the moving annuli
by smooth stationary diffeomorphisms and write
$\mathcal B=\mathbb R\times\mathcal A^\circ$ for their common
stationary region, with limiting metric $\g_*$.
Fix a Lipschitz future $\g_*$-timelike segment $c$ in a compact
set $Q$, with $\g_*(\dot c,\dot c)\le-\epsilon_c<0$ almost
everywhere. Its bounded tangent and metric convergence keep it
$\g_j$-timelike for large $j$. The reverse Cauchy--Schwarz
inequality, \eqref{eq:translated-graph-time}, and the height bound
\eqref{eq:interior-height} give
\begin{equation}\label{eq:limit-proper-time}
 \begin{aligned}
 L_{\g_*}(c)=\lim_jL_{\g_j}(c)
 &\le C\limsup_j\bigl[G_j(c_+)-G_j(c_-)\bigr]\\
 &\le C\bigl(\operatorname{osc}_Q t+2C_{\pi(Q)}\bigr).
 \end{aligned}
\end{equation}
Here the chain rule for $G_j=t-u_j(x)$ holds almost everywhere.
Thus \eqref{eq:compact-proper-time} holds for $\g_*$. Choose a compact
spatial annulus $K\Subset\mathcal A^\circ$ whose boundary lies
in the collars with bounded normals. Its truncated hypersurfaces
\[
 M_j=\Sigma_j\cap\pi^{-1}(K)
       \subset Q=[-C_K,C_K]\times K
\]
have bounded boundary normals. Lemma~\ref{prop:compact-causal}
supplies a fixed temporal function $\tau$ near $Q$, which we
shift so that $1\le\tau\le\tau_{\max}$. Compactness and convergence
of the inverse metrics give, for large $j$,
\[
 \inf_Q[-\g_j^{-1}(\mathrm d\tau,\mathrm d\tau)]
 \ge\tfrac12\min_Q[-\g_*^{-1}(\mathrm d\tau,\mathrm d\tau)]>0.
\]
This bound, the fixed $C^3$ norm of $\tau$, and smooth convergence
of $\g_j$ control the coefficients in
Lemma~\ref{lem:temporal-gradient}, whose boundary term is
controlled by the collar estimates. Equation~\eqref{eq:temporal-gradient}
therefore bounds the interior normals.

All hypersurfaces now lie in a fixed compact region and have
uniformly bounded normals. In a sufficiently small cylinder
with timelike vertical field, each is a graph over a fixed
spatial ball. To verify that the graph has a uniform domain,
let $L$ bound its coordinate slope and choose the cylinder
$(-4Lr,4Lr)\times B_{4r}$. Starting at a point near its centre,
continuation along radial segments over $B_r$ satisfies
\[
 |w_j(y)|\le|w_j(y_j)|+L|y-y_j|\le2Lr.
\]
Thus the graph cannot reach $x^0=\pm4Lr$ and continues
at any interior endpoint by closedness of the hypersurface
and transversality to $\partial_{x^0}$. It is the only sheet
in a smaller cylinder, since $G_j$ has timelike gradient and
therefore increases strictly along each entire vertical fiber.
We may consequently apply the ellipticity bound
\eqref{eq:graph-elliptic-bounds} on a fixed domain, obtaining
uniform $C^{m,\alpha}$ estimates on smaller balls.

Choose a Hausdorff convergent subsequence in the common compact
region. The graph estimates upgrade its convergence to smooth
convergence near interior limit points, and the collar estimates
give the same conclusion at the boundary. Uniqueness on vertical
fibers prevents crossing sheets or coincident sheets with
multiplicity, so these local limits agree to form a smooth
compact connected spacelike maximal hypersurface $\Sigma_*$,
with $\partial\Sigma_*=\mathcal S_0\sqcup C_{s_*}$.

It remains to show that $\Sigma_*^\circ$ meets each stationary
orbit in $\mathcal B$ exactly once.
Passing \eqref{eq:charged-lapse-equation} to the smooth limit gives
\[
 n_*\ge0,\qquad
 \Delta_{\Sigma_*}n_*
   =\bigl(|k_*|^2+\R_*(T_{0,*},T_{0,*})\bigr)n_*.
\]
The nonnegative coefficient in \eqref{eq:charged-constraint}
and the strict outer-collar bound
$n_*\ge\sqrt{-\g_*(\T,\T)}>0$ allow us to apply the strong
maximum principle. We obtain $n_*>0$ on the connected interior,
so the stationary projection is a local diffeomorphism.
It is also bijective. For each fixed $x$, boundedness of
$u_j(x)$ gives a limit point above $x$, while two such points
would, by transversality and smooth local convergence, produce
two points of each late graph above $x$. Hence
$\Sigma_*^\circ=\{t=u_*(x)\}$.

To identify the topology of the limit, straighten
$C_{s_j}$ to $C_{s_*}$ by diffeomorphisms tending to the identity
and fixing $\mathcal S_0$. Smooth convergence up to both
boundary collars then makes normal projection onto $\Sigma_*$
a diffeomorphism for large $j$, proving
$\Sigma_*\simeq\mathbb S^2\times[0,1]$.
\end{proof}

\begin{proof}[Proof of Lemma~\ref{lem:temporal-gradient}]
Suppress $j$. We estimate $v=-\g(T_0,V)$ by differentiating
along $\Sigma$, keeping the terms containing its second
fundamental form explicit. The remaining coefficients are
bounded by the hypotheses, since
$V=-\alpha^{-1}\operatorname{grad}_{\g}\tau$ and $\alpha\ge a_0$
give bounds for $V,\D V,\D^2V$ and $\Ric(\g)$ in the fixed
coordinate cover. Choose a tangent orthonormal frame with
$\nabla_{e_a}e_b=0$ at the point, put $V^a=\g(V,e_a)$, and note that
$v^2=1+\alpha^{-2}|\nabla_\Sigma\tau|^2$.
The Gauss formula $\D_{e_a}e_b=-k_{ab}T_0$ and
$\operatorname{tr}k=0$ give
\begin{equation}\label{eq:tilt-derivatives}
 \begin{aligned}
 \nabla_a v&=k_{ab}V^b-\g(T_0,\D_{e_a}V),\\
 \nabla_aV^b&=\g(\D_{e_a}V,e_b)+vk_{ab},\\
 \Delta_\Sigma v
 &=(\nabla^ak_{ab})V^b
       +k^{ab}\bigl[\g(\D_{e_a}V,e_b)+vk_{ab}\bigr]\\
 &\quad+k^{ab}\g(e_b,\D_{e_a}V)
       -\g\left(T_0,\sum_a(\D^2V)(e_a,e_a)\right)\\
 &=|k|^2v+(\nabla^ak_{ab})V^b
       +2k^{ab}\g(e_b,\D_{e_a}V)\\
 &\qquad-\g\left(T_0,\sum_a(\D^2V)(e_a,e_a)\right).
 \end{aligned}
\end{equation}
In an ambient orthonormal frame with time axis $V$, the
components of $T_0,e_a$ are bounded by $Cv$.
Using Codazzi and the ambient Ricci bound in
\eqref{eq:tilt-derivatives} gives, for every fixed $\varepsilon>0$,
\[
 \Delta_\Sigma v
 \ge |k|^2v-C|k|v^2-Cv^3
 \ge(1-\varepsilon)|k|^2v-C_\varepsilon v^3.
\]
For $\eta=\operatorname{arcosh}v$ on $\{v>2\}$, this yields
\begin{equation}\label{eq:tilt-laplacian}
 \begin{aligned}
 \Delta_\Sigma\eta
 &=\coth\eta\left(\frac{\Delta_\Sigma v}{v}
                           -|\nabla_\Sigma\eta|^2\right)\\
 &\ge\coth\eta\bigl\{
 (1-\varepsilon)|k|^2-|\nabla_\Sigma\eta|^2
                                  -C_\varepsilon v^2\bigr\}.
 \end{aligned}
\end{equation}
To retain a positive gradient term, let $\lambda_1$ be the
eigenvalue of $k$ of greatest absolute value. Maximality gives
\[
 |k|^2\ge\lambda_1^2+\tfrac12(\lambda_2+\lambda_3)^2
          =\tfrac32\lambda_1^2.
\]
The first identity in \eqref{eq:tilt-derivatives} and
$\sum_a(V^a)^2=v^2-1$ give, for fixed $\varepsilon>0$,
\begin{equation}\label{eq:tilt-gradient-control}
 \begin{aligned}
 |\nabla_\Sigma\eta|
 &\le\frac{|k_{ab}V^b|+Cv^2}{\sqrt{v^2-1}}
       \le|\lambda_1|+Cv,\\
 |\nabla_\Sigma\eta|^2
 &\le(1+\varepsilon)\lambda_1^2+C_\varepsilon v^2
       \le\tfrac23(1+\varepsilon)|k|^2+C_\varepsilon v^2.
 \end{aligned}
\end{equation}
Fix $0<\varepsilon\le1/11$. Substituting
\eqref{eq:tilt-gradient-control} in
\eqref{eq:tilt-laplacian} gives
\begin{equation}\label{eq:coercive-tilt}
 \begin{gathered}
 \begin{aligned}
 \Delta_\Sigma\eta
 &\ge\coth\eta\left\{
   \left[\frac{3(1-\varepsilon)}{2(1+\varepsilon)}-1\right]
        |\nabla_\Sigma\eta|^2-C_\varepsilon v^2\right\}\\
 &\ge\coth\eta
       \bigl\{\tfrac14|\nabla_\Sigma\eta|^2-Cv^2\bigr\},
 \end{aligned}\\
 |\Delta_\Sigma\tau|
       =|\operatorname{tr}_{T\Sigma}\D^2\tau|\le Cv^2.
 \end{gathered}
\end{equation}
At an interior maximum of $\eta+\mu\log\tau$,
\[
 \nabla_\Sigma\eta=-\frac\mu\tau\nabla_\Sigma\tau.
\]
Using both bounds in \eqref{eq:coercive-tilt}, we obtain
\begin{equation}\label{eq:weighted-tilt}
 \begin{aligned}
 0&\ge\Delta_\Sigma\eta+\frac\mu\tau\Delta_\Sigma\tau
                  -\frac\mu{\tau^2}|\nabla_\Sigma\tau|^2\\
  &\ge\coth\eta\left\{
    \frac{\mu^2/4-\mu\tanh\eta}{\tau^2}
                     |\nabla_\Sigma\tau|^2
       -Cv^2\left(1+\frac{\mu\tanh\eta}{\tau}\right)\right\}\\
  &\ge\coth\eta\left\{
    \frac{\mu^2/4-\mu}{\tau^2}|\nabla_\Sigma\tau|^2
       -Cv^2\left(1+\frac\mu\tau\right)\right\}.
 \end{aligned}
\end{equation}
The positive coefficient in \eqref{eq:weighted-tilt} is quadratic
in $\mu$, while the error is linear, with $C$ independent of
$\mu$ and the hypersurface. Choose $\mu$ sufficiently large that
\[
 (\mu^2/4-\mu)a_0^2
       \ge2C(\tau_{\max}^2+\mu\tau_{\max}).
\]
Using this choice in \eqref{eq:weighted-tilt}, with
$0<\tau\le\tau_{\max}$, gives
\[
 2a_0^{-2}|\nabla_\Sigma\tau|^2
 \le v^2=1+\alpha^{-2}|\nabla_\Sigma\tau|^2
 \le1+a_0^{-2}|\nabla_\Sigma\tau|^2.
\]
This implies $|\nabla_\Sigma\tau|^2\le a_0^2$ and $v^2\le2$,
contradicting the assumption $v>2$ at the interior maximum.
Since the logarithmic weight tends to $-\infty$ as $\tau\to0$,
we need only compare with the boundary maxima and points where
$v\le2$. The inequalities
$v\le e^{\operatorname{arcosh}v}\le2v$ then give
\eqref{eq:temporal-gradient}.
\end{proof}

\begin{proof}[Proof of Lemma~\ref{prop:compact-causal}]
Timelike displacement of a causal curve, with the displacement
chosen by a scalar differential equation, is used in the proof of
\cite{ChruscielGrant}. Here we need a displacement
that stays uniformly small along an inextendible curve while
producing infinite proper time. We construct it to show that
\eqref{eq:compact-proper-time} excludes a causal curve confined
to a compact set with no future endpoint in $\mathcal B$.
Suppose that $\gamma$ is such a curve.
Choose a $C^2$ future unit timelike field $V$ on a neighbourhood of
its compact closure, and use the Riemannian metric
$\g+2V^\flat\otimes V^\flat$ to measure tangent vectors.
Parametrize $\gamma$ so that its
almost-everywhere tangent $l$ satisfies
\begin{equation}\label{eq:causal-tangent-normalization}
 \g(l,V)=-1,\qquad \g(l,l)\le0,\qquad
 |l|_{\g+2V^\flat\otimes V^\flat}^2\le2 .
\end{equation}
The parameter in \eqref{eq:causal-tangent-normalization}
ranges over $[0,\infty)$, since a finite range would give
finite Riemannian length in a compact set and hence an endpoint,
contrary to inextendibility. Put
\[
 b(s)=\g(l,\D_lV),\qquad |b|\le K,\qquad K\ge1.
\]
The bound follows from $|l|\le\sqrt2$ in
\eqref{eq:causal-tangent-normalization} and the bound for $\D V$
on the compact closure of $\gamma$.
Under a displacement $uV$, the linear change in the squared
norm of the tangent is $2bu-2u'$. The coefficient $b$ has no
sign, so a constant displacement does not suffice. We choose
$u'-bu=c_0u^2$ so that $2bu-2u'=-2c_0u^2$ dominates
the quadratic error. Infinite proper time will then follow from
$\int_0^\infty|u|\,\mathrm ds=\infty$, while a small bound for
$|u|$ keeps the displaced curve in a compact neighbourhood.
Put
\[
 \Lambda'(s)=\exp\left(\int_0^sb(t)\,\mathrm dt\right),
 \qquad \Lambda(s)=\int_0^s\Lambda'(t)\,\mathrm dt.
\]
The equation for the displacement becomes
\[
 \left(\frac{\Lambda'}u\right)'=-c_0\Lambda'.
\]
Thus $\Lambda'/u+c_0\Lambda$ is constant. Choosing its value
according to whether $\Lambda$ is bounded gives, for $c_0$
sufficiently large,
\begin{equation}\label{eq:causal-displacement}
 u(s)=
 \begin{cases}
 -\dfrac{\Lambda'(s)}{c_0(1+\Lambda(s))},
       &\Lambda(s)\longrightarrow\infty,\\[7pt]
 \dfrac{\Lambda'(s)}{c_0(\Lambda_\infty-\Lambda(s))},
       &\Lambda(s)\longrightarrow\Lambda_\infty<\infty .
 \end{cases}
\end{equation}
The two choices in \eqref{eq:causal-displacement} satisfy,
almost everywhere,
\begin{equation}\label{eq:Riccati}
 u'=b u+c_0u^2,\qquad
 |u|\le\frac K{c_0},\qquad |u'|\le2K|u|,
 \qquad \int_0^\infty|u|\,\mathrm ds=\infty .
\end{equation}
Indeed, in the first case
$\Lambda'\le1+K\Lambda\le K(1+\Lambda)$.
In the second case
\[
 \Lambda_\infty-\Lambda(s)
 =\int_s^\infty\Lambda'(t)\,\mathrm dt
 \ge\Lambda'(s)\int_s^\infty e^{-K(t-s)}\,\mathrm dt
 =\frac{\Lambda'(s)}K .
\]
Integrating \eqref{eq:causal-displacement} gives, in the two
cases respectively, $\int_0^s|u(t)|\,\mathrm dt$ equal to
\[
 c_0^{-1}\log(1+\Lambda(s))\qquad\text{and}\qquad
 c_0^{-1}\log[\Lambda_\infty/(\Lambda_\infty-\Lambda(s))].
\]

Choose $u_0>0$ and a compact neighbourhood $C'$ containing
$\Psi_{u,V}(\overline\gamma)$ for $|u|\le u_0$.
Set $\widetilde\gamma(s)=\Psi_{u(s),V}(\gamma(s))$.
For $c_0\ge K/u_0$, the bound $|u|\le K/c_0$ in
\eqref{eq:Riccati} keeps $\widetilde\gamma$ in $C'$.
Its tangent is
\[
 \dot{\widetilde\gamma}
       =\mathrm d\Psi_{u,V}(l+u'V).
\]
The $C^2$ bounds for $\g,V$ on $C'$ bound
$\mathcal L_V\g$ and $\mathcal L_V^2\g$ there.
Using $(\mathcal L_V\g)(l,l)=2b$ and
\eqref{eq:causal-tangent-normalization} gives the Taylor
expansions below, with remainder constants independent of
$s$ and $c_0$,
\begin{equation}\label{eq:flow-metric-expansion}
 \begin{aligned}
 (\Psi_{u,V}^*\g)(l,l)&=\g(l,l)+2b u+O(u^2),\\
 (\Psi_{u,V}^*\g)(l,V)&=-1+O(|u|),\\
 (\Psi_{u,V}^*\g)(V,V)&=-1.
 \end{aligned}
\end{equation}
Combining \eqref{eq:flow-metric-expansion} with the equation
for $u$ in \eqref{eq:Riccati} gives
\begin{equation}\label{eq:displaced-curve-norm}
 \begin{aligned}
 \g(\dot{\widetilde\gamma},\dot{\widetilde\gamma})
 &=\g(l,l)+2b u-2u'-(u')^2+O(u^2+|uu'|)\\
 &\le-2c_0u^2-(u')^2+O(u^2)
 \le-c_0u^2<0 .
 \end{aligned}
\end{equation}
Here the bound $|u'|\le2K|u|$ in \eqref{eq:Riccati} makes
the remainder constant independent of $c_0$. To check that
the timelike displacement preserves the future orientation,
use the mixed term in \eqref{eq:flow-metric-expansion} to obtain
\[
 \g(\dot{\widetilde\gamma},V)
       =-1-u'+O(|u|)<-\tfrac12.
\]
Thus each finite segment of $\widetilde\gamma$ is future
timelike and satisfies the hypotheses of
\eqref{eq:compact-proper-time}. Indeed, its tangent is bounded,
and the positive minimum of $|u|$ on that interval makes
\eqref{eq:displaced-curve-norm} a uniform timelike bound.
But \eqref{eq:displaced-curve-norm} and \eqref{eq:Riccati}
give
\[
 L_\g(\widetilde\gamma|_{[0,S]})
       \ge\sqrt{c_0}\int_0^S|u|\,\mathrm ds
       \longrightarrow\infty .
\]
This is a contradiction. Time reversal excludes
past-inextendible causal curves confined to compact sets.

Having excluded future- and past-inextendible causal curves
in compact sets, we apply Beem's theorem to an open
$\mathcal O$ with $Q\Subset\mathcal O\Subset\mathcal B$.
The theorem makes $\mathcal O$ stably causal, meaning that
its cones can be widened while preserving
causality~\cite{MinguzziCausality}. It therefore admits a
smooth temporal function with past timelike
gradient~\cite{MinguzziCausality}.
\end{proof}

\section*{Acknowledgements}

The author was partially supported by a grant of the Ministry of Research,
Innovation and Digitization, CCCDI - UEFISCDI, project number
ROSUA-2024-0001, within PNCDI IV.

\section*{Generative AI}

ChatGPT 5.6 Sol Ultra was used for literature checks and editorial work including transcribing handwritten notes into LaTeX. The author wrote and verified all mathematical content and takes full responsibility for the manuscript.

\end{document}